\documentclass[12pt]{amsart}%
\usepackage{amsmath}
\usepackage{graphicx}
\usepackage[small,nohug,heads=littlevee]{diagrams}
\usepackage{amscd}
\usepackage{geometry}
\usepackage{amsfonts}
\usepackage{amssymb}
\usepackage{amsmath, amsfonts, latexsym, amssymb, graphicx, mathtools,
wasysym,}
\usepackage{hyperref}%
\providecommand{\U}[1]{\protect\rule{.1in}{.1in}}
\newtheorem{theorem}{Theorem}[section]
\theoremstyle{plain}

\newtheorem{corollary}[theorem]{Corollary}

\newtheorem{example}[theorem]{Example}
\newtheorem{exercise}[theorem]{Exercise}
\newtheorem{conjecture}[theorem]{Conjecture}
\newtheorem{Question}{Question}
\newtheorem{problem}[theorem]{Problem}
\newtheorem{remark}[theorem]{Remark}
\newtheorem{lemma}[theorem]{Lemma}

\newtheorem{proposition}[theorem]{Proposition}

\numberwithin{equation}{section}
\numberwithin{Question}{section}

\hypersetup{colorlinks,citecolor=blue,linkcolor=blue}
\let\oldtocsection=\tocsection
\let\oldtocsubsection=\tocsubsection
\let\oldtocsubsubsection=\tocsubsubsection
\renewcommand{\tocsection}[2]{\hspace{0em}\oldtocsection{#1}{#2}}
\renewcommand{\tocsubsection}[2]{\hspace{2em}\oldtocsubsection{#1}{#2}}
\renewcommand{\tocsubsubsection}[2]{\hspace{4em}\oldtocsubsubsection{#1}{#2}}

\begin{document}
\title[Chapter 3: Ends, shapes, and boundaries ]{Chapter 3 \medskip{\linebreak} Ends, shapes, and boundaries in manifold
topology and geometric group theory}
\author{C. R. Guilbault}
\address{Department of Mathematical Sciences, University of Wisconsin-Milwaukee,
Milwaukee, Wisconsin 53201}
\email{craigg@uwm.edu}
\thanks{This project was aided by a Simons Foundation Collaboration Grant. }
\date{February 18, 2021 (revised enumeration scheme to match published version)}
\subjclass{Primary 57N15, 57Q12; Secondary 57R65, 57Q10}
\keywords{end, shape, boundary, manifold, group, fundamental group at infinity, tame,
open collar, pseudo-collar, Z-set, Z-boundary, Z-structure}

\begin{abstract}
This survey/expository article covers a variety of topics related to the
\textquotedblleft topology at infinity\textquotedblright\ of noncompact
manifolds and complexes. In manifold topology and geometric group theory, the
most important noncompact spaces are often contractible, so distinguishing one
from another requires techniques beyond the standard tools of algebraic
topology. One approach uses end invariants, such as the number of ends or the
fundamental group at infinity. Another approach seeks nice compactifications,
then analyzes the boundaries. A thread connecting the two approaches is shape theory.

In these notes we provide a careful development of several topics: homotopy
and homology properties and invariants for ends of spaces, proper maps and
homotopy equivalences, tameness conditions, shapes of ends, and various types
of $\mathcal{Z}$-compactifications and $\mathcal{Z}$-boundaries. Classical and
current research from both manifold topology and geometric group theory
provide the context. Along the way, several open problems are encountered. Our
primary goal is a casual but coherent introduction that is accessible to
graduate students and also of interest to active mathematicians whose research
might benefit from knowledge of these topics.

\end{abstract}
\maketitle
\tableofcontents

\newpage

\section*{Preface}

In \cite{Si2}, a paper that plays a role in these notes, Siebenmann mused that
his work was initiated at a time \textquotedblleft when `respectable'
geometric topology was necessarily compact\textquotedblright. That attitude
has long since faded; today's topological landscape is filled with research in
which noncompact spaces are the primary objects. Even so, past traditions have
impacted today's topologists, many of whom developed their mathematical tastes
when noncompactness was viewed more as a nuisance than an area for
exploration. For that and other reasons, many useful ideas and techniques have
been slow to enter the mainstream. One goal of this set of notes is to provide
quick and intuitive access to some of those results and methods by weaving
them together with more commonly used approaches, well-known examples, and
current research. In this way, we attempt to present a coherent
\textquotedblleft theory of ends\textquotedblright\ that will be useful to
mathematicians with a variety of interests.

Numerous topics included here are fundamental to manifold topology and
geometric group theory: Whitehead and Davis manifolds, Stallings'
characterization of Euclidean spaces, Siebenmann's Thesis, Chapman and
Siebenmann's $\mathcal{Z}$-com\-pact\-ific\-at\-ion Theorem, the
Freudenthal-Hopf-Stallings Theorem on ends of groups, and applications of the
Gromov boundary to group theory---to name just a few. We hope these notes give
the reader a better appreciation for some of that work. Many other results and
ideas presented here are relatively new or still under development:
generalizations of Siebenmann's thesis, Bestvina's $\mathcal{Z}$-structures on
groups, use of $\mathcal{Z}$-boundaries in manifold topology, and applications
of boundaries to non-hyperbolic groups, are among those discussed. There is
much room for additional work on these topics; the natural path of our
discussion will bring us near to a number of interesting open problems.

The style of these notes is to provide a lot of motivating examples. Key
definitions are presented in a rigorous manner---often preceded by a
non-rigorous, but (hopefully) intuitive introduction. Proofs or sketches of
proofs are included for many of the fundamental results, while many others are
left as exercises. We have not let issues of mathematical rigor prevent the
discussion of important or interesting work. If a theorem or example is
relevant, we try to include it, even when the proof is too long or deep for
these pages. When possible, an outline or key portions of an argument are
provided---with implied encouragement for the reader to dig deeper.

These notes originated in a series of four one-hour lectures given at the
workshop on \emph{Geometrical Methods in High-dimensional Topology},\emph{
}hosted by Ohio State University in the spring of 2011. Notes from those talks
were expanded into a one-semester topics course at the University of
Wisconsin-Milwaukee in the fall of that year. The author expresses his
appreciation to workshop organizers Jean-Francois Lafont and Ian Leary for the
opportunity to speak, and acknowledges all fellow participants in the OSU
workshop and the UWM graduate students in the follow-up course; their feedback
and encouragement were invaluable. Special thanks go to Greg Friedman and the
anonymous referee who read the initial version of this document, pointed out
numerous errors, and made many useful suggestions for improving both the
mathematics and the presentation. Finally, thanks to my son Phillip Guilbault
who created most of the figures in this document.

\renewcommand{\thesection}{3.\arabic{section}}

\renewcommand{\thefigure}{3.\arabic{figure}}

\section{Introduction}

A fundamental concept in the study of noncompact spaces is the
\textquotedblleft number of ends\textquotedblright. For example, the real line
has two ends, the plane has one end, and the uniformly trivalent tree
$\mathbb{T}_{3}$ has infinitely many ends. Counting ends has proven remarkably
useful, but certainly there is more---after all, there is a qualitative
difference between the single end of the ray $[0,\infty)$ and that of $%
\mathbb{R}
^{2}$. This provides an idea: If, in the topological tradition of counting
things, one can (somehow) use the $\pi_{0}$- or $H_{0}$-functors to measure
the number of ends, then maybe the $\pi_{1}$- and $H_{1}$-functors (or, for
that matter $\pi_{k}$ and $H_{k}$), can be used in a similar manner to measure
other properties of those ends. Turning that idea into actual
mathematics---the \textquotedblleft end invariants\textquotedblright\ of a
space---then using those invariants to solve real problems, is one focus of
the early portions of these notes.

Another approach to confronting noncompact spaces is to
compactify.\footnote{Despite our affinity for noncompact spaces, we are not
opposed to the practice of compactification, provided it is done in a
(geometrically) sensitive manner.} The 1-point com\-pact\-ific\-at\-ion of $%
\mathbb{R}
^{1}$ is a circle and the 1-point com\-pact\-ific\-at\-ion of $%
\mathbb{R}
^{2}$ a $2$-sphere. A \textquotedblleft better\textquotedblright%
\ com\-pact\-ific\-at\-ion of $%
\mathbb{R}
^{1}$ adds one point to each end, to obtain a closed interval---a space that
resembles the line far more than does the circle. This is a special case of
\textquotedblleft end-point com\-pact\-ific\-at\-ion\textquotedblright,
whereby a single point is added to each end of a space. Under that procedure,
an entire Cantor set is added to $\mathbb{T}_{3}$, resulting in a compact, but
still tree-like object. Unfortunately, the end-point com\-pact\-ific\-at\-ion
of $%
\mathbb{R}
^{2}$ again yields a $2$-sphere. From the point of view of preserving
fundamental properties, a far better com\-pact\-ific\-at\-ion of $%
\mathbb{R}
^{2}$ adds an entire circle at infinity. This is a prototypical
\textquotedblleft$\mathcal{Z}$-com\-pact\-ific\-at\-ion\textquotedblright,
with the circle as the \textquotedblleft$\mathcal{Z}$%
-boundary\textquotedblright. (The end-point com\-pact\-ific\-at\-ions of $%
\mathbb{R}
^{1}$ and $\mathbb{T}_{3}$ are also $\mathcal{Z}$-com\-pact\-ific\-at\-ions.)
The topic of $\mathcal{Z}$-com\-pact\-ific\-at\-ion and $\mathcal{Z}%
$-boundaries is a central theme in the latter half of these notes.

Shape theory is an area of topology developed for studying compact spaces with
bad local properties, so it may seem odd that \textquotedblleft
shapes\textquotedblright\ is one of three topics mentioned in the title of an
article devoted to \emph{noncompact} spaces with \emph{nice} local properties.
This is not a mistake! As it turns out, the tools of shape theory are easily
adaptable to the study of ends---and the connection is not just a similarity
in approaches. Frequently, the shape of an appropriately chosen compactum
precisely captures the illusive \textquotedblleft topology at the end of a
space\textquotedblright. In addition, shape theory plays a clarifying role by
connecting end invariants hinted at in paragraph one of this introduction to
the $\mathcal{Z}$-boundaries mentioned in paragraph two. To those who know
just enough about shape theory to judge it too messy and set-theoretical for
use in manifold topology or geometric group theory (a belief briefly shared by
this author), patience is encouraged. At the level of generality required for
our purposes, shape theory is actually quite elegant and geometric. In fact,
very little set-theoretic topology is involved---instead spaces with bad
properties are quickly replaced by simplicial and CW complexes, where
techniques are clean and intuitive. A working knowledge of shape theory is one
subgoal of these notes.

\subsection{Conventions and notation}

Throughout this article, all spaces are separable metric. A \emph{compactum
}is a compact space.\emph{ }We often restrict attention to\emph{ absolute
neighborhood retracts} (or \emph{ANRs})---a particularly nice class of spaces,
whose most notable property is local contractibility. In these notes, ANRs are
required to be locally compact. Notable examples of ANRs are: manifolds,
locally finite polyhedra, locally finite CW complexes, proper CAT(0)
spaces\footnote{A \emph{proper }metric space is one in which every closed
metric ball is compact.}, and Hilbert cube manifolds. Due to their unavoidable
importance, a short appendix with precise definitions and fundamental results
about ANRs has been included. Readers anxious get started can safely begin, by
viewing \textquotedblleft ANR\textquotedblright\ as a common label for the
examples just mentioned. An \emph{absolute retract} (or AR) is a contractible
ANR, while an \emph{ENR} [resp., \emph{ER}] is a finite-dimensional ANR
[resp., AR].

The unmodified term \emph{manifold} means \textquotedblleft finite-dimensional
manifold\textquotedblright. A manifold is \emph{closed} if it is compact and
has no boundary and \emph{open} if it is noncompact with no boundary; if
neither is specified, boundary is permitted. For convenience, all manifolds
are assumed to be piecewise-linear (PL); in other words, they may be viewed as
simplicial complexes in which all links are PL homeomorphic to spheres of
appropriate dimensions. A primary application of PL topology will be the
casual use of general position and regular neighborhoods. A good source for
that material is \cite{RS}. Nearly all that we do can be accomplished for
smooth or topological manifolds as well; readers with expertise in those
categories will have little trouble making the necessary adjustments.

Hilbert cube manifolds are entirely different objects. The \emph{Hilbert cube}
is the countably infinite product $\mathcal{Q}=\prod_{i=1}^{\infty}\left[
-1,1\right]  $, endowed with the product topology. A space $X$ is a
\emph{Hilbert cube manifold} if each $x\in X$ has a neighborhood homeomorphic
to $\mathcal{Q}$. Like ANRs, Hilbert cube manifolds play an unavoidably key
role in portions of these notes. For that reason, we have included a short and
simple appendix on Hilbert cube manifolds.

Symbols will be used as follows: $\approx$ denotes homeomorphism, while
$\simeq$ indicates homotopic maps or homotopy equivalent spaces; $\cong$
indicates isomorphism. When $M^{n}$ is a manifold, $n$ indicates its dimension
and $\partial M^{n}$ its manifold boundary. When $A$ is a subspace of $X$,
$\operatorname*{Bd}_{X}A$ (or when no confusion can arise, $\operatorname*{Bd}%
X$) denotes the set-theoretic boundary of $A$. The symbols $\overline{A}$ and
$\operatorname*{cl}_{X}A$ (or just $\operatorname*{cl}A$) denote the closure
of $A$ in $X$, while $\operatorname*{int}_{X}A$ (or just $\operatorname*{int}%
A$) denotes the interior. The symbol $\widetilde{X}$ always denotes the
universal cover of $X$. Arrows denote (continuous) maps or homomorphisms, with
$\hookrightarrow$, $\rightarrowtail$, and $\twoheadrightarrow$ indicating
inclusion, injection and surjection, respectively.

\section{Motivating examples: contractible open
manifolds\label{Section: Motivating examples}}

Let us assume that space-time is a large boundaryless $4$-dimensional
manifold. Recent evidence suggests that this manifold is noncompact (an
\textquotedblleft open universe\textquotedblright). By running time backward
to the Big Bang, we might reasonably conclude that space-time is
\textquotedblleft just\textquotedblright\ a contractible open
manifold\footnote{No expertise in cosmology is being claimed by the author.
This description of space-time is intended only to motivate discussion.}.
Compared to the possibilities presented by a closed universe ($\mathbb{S}^{4}%
$, $\mathbb{S}^{2}\times\mathbb{S}^{2}$, $%
\mathbb{R}
P^{4}$, $\mathbb{C}P^{2}$, the $E_{8}$ manifold$,\cdots\ ?$), the idea of a
contractible open universe seems rather disappointing, especially to a
topologist primed for the ultimate example on which to employ his/her tools.
But there is a mistake in this thinking---an implicit assumption that a
contractible open manifold is topologically uninteresting (no doubt just a
blob, homeomorphic to an open ball). In this section we take a quick look at
the surprisingly rich world of contractible open manifolds.

\subsection{Classic examples of exotic contractible open
manifolds\label{Subsection: classic examples contractible manifolds}}

For $n=1$ or $2$, it is classical that every contractible open $n$-manifold is
topologically equivalent to $%
\mathbb{R}
^{n}$; but when $n\geq3$, things become interesting. J.H.C Whitehead was among
the first to underestimate contractible open manifolds. In an attempt to prove
the Poincar\'{e} Conjecture, he briefly claimed that, in dimension 3, each is
homeomorphic to $%
\mathbb{R}
^{3}$. In \cite{Wh} he corrected that error by constructing the now famous
\emph{Whitehead contractible }$3$\emph{-manifold}---an object surprisingly
easy to describe.

\begin{example}
[Whitehead's contractible open 3-manifold]%
\label{Example: Definition of Whitehead manifold}Let $\mathcal{W}%
^{3}=\mathbb{S}^{3}-T_{\infty}$, where $T_{\infty}$ is the compact set (the
\textbf{Whitehead continuum}) obtained by intersecting a nested sequence
$T_{0}\supseteq T_{1}\supseteq T_{2}\supseteq\cdots$ of solid tori, where each
$T_{i+1}$ is embedded in $T_{i}$ in the same way that $T_{1}$ is embedded in
$T_{0}$. See Figure \ref{Figure: Whitehead}.
\begin{figure}[ptb]%
\centering
\includegraphics[
height=1.785in,
width=2.5884in
]%
{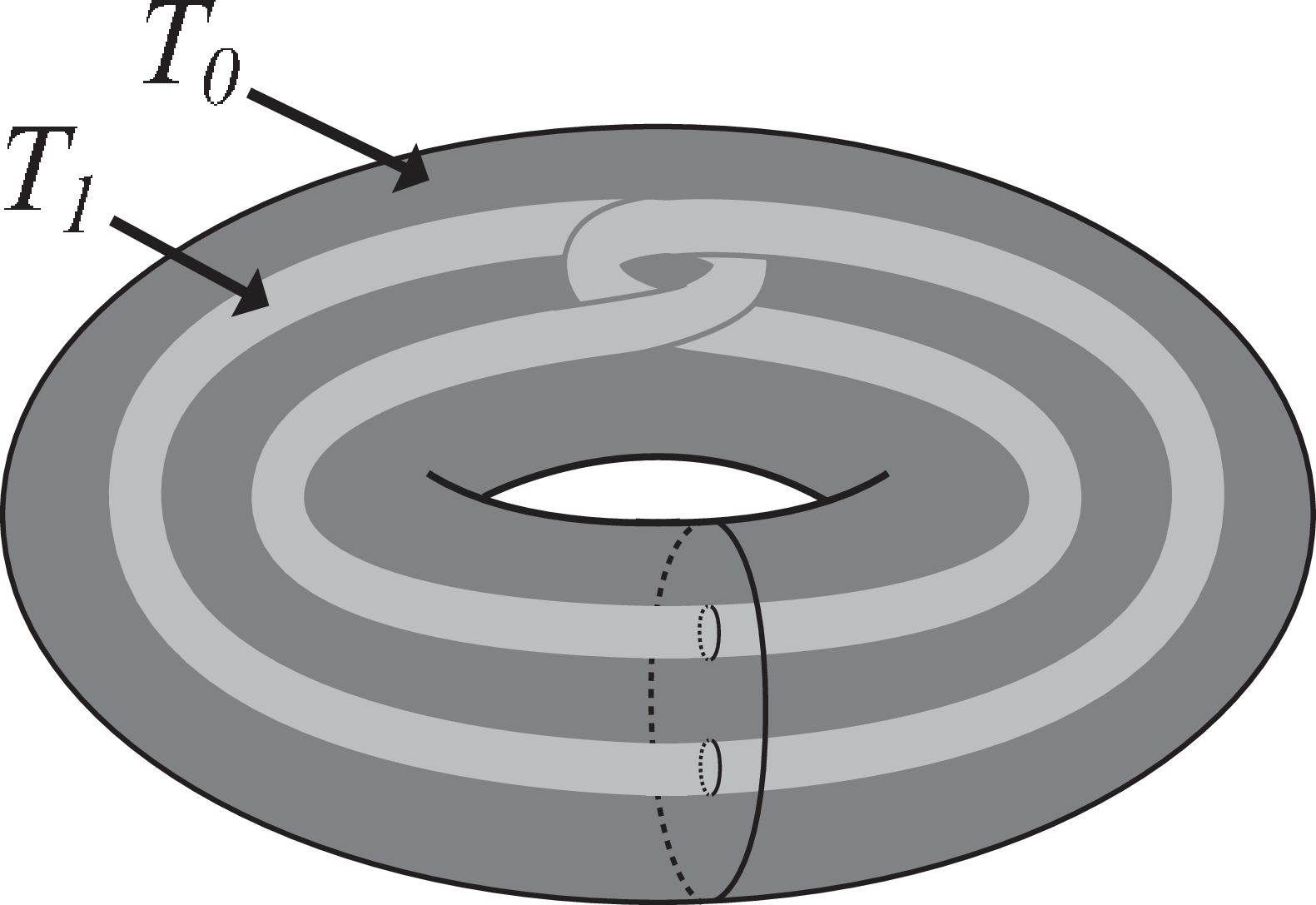}%
\caption{Constructing the Whitehead manifold}%
\label{Figure: Whitehead}%
\end{figure}
Standard tools of algebraic topology show that $\mathcal{W}^{3}$ is
contractible. For example, first show that $\mathcal{W}^{3}$ is simply
connected (this takes some thought), then show that it is acyclic with respect
to $%
\mathbb{Z}
$-homology.
\end{example}

The most interesting question about $\mathcal{W}^{3}$ is: \emph{Why is it not
homeomorphic to }$%
\mathbb{R}
^{3}$\emph{? }Standard algebraic invariants are of little use, since
$\mathcal{W}^{3}$ has the homotopy type of a point. But a variation on the
fundamental group---the \textquotedblleft fundamental group at
infinity\textquotedblright---does the trick. Before developing that notion
precisely, we describe a few more examples of \emph{exotic} contractible open
manifolds, i.e., contractible open manifolds not homeomorphic to a Euclidean space.

It turns out that exotic examples are rather common; moreover, they play
important roles in both manifold topology and geometric group theory. But for
now, let us just think of them as possible universes.

In dimension $\leq2$ there are no exotic contractible open manifolds, but in
dimension 3, McMillan \cite{Mc} constructed uncountably many. In some sense,
his examples are all variations on the Whitehead manifold. Rather than
examining those examples, let us move to higher dimensions, where new
possibilities emerge.

For $n\geq4$, there exist \emph{compact} contractible $n$-manifolds not
homeomorphic to the standard $n$-ball $\mathbb{B}^{n}$. We call these
\emph{exotic compact contractible manifolds}. Taking interiors provides a
treasure trove of easy-to-understand exotic contractible open manifolds. We
provide a simple construction for some of those objects.

Recall that a group is \emph{perfect} if its abelianization is the trivial
group. A famous example, the \emph{binary icosahedral group}, is given by the
presentation $\left\langle s,t\mid\left(  st\right)  ^{2}=s^{3}=t^{5}%
\right\rangle $.

\begin{example}
[Newman contractible manifolds]Let $G$ be a perfect group admitting a finite
presentation with an equal number of generators and relators. The
corresponding presentation $2$-complex, $K_{G}$ has the homology of a point.
Embed $K_{G}$ in $\mathbb{S}^{n}$ ($n\geq5$) and let $N$ be a regular
neighborhood of $K_{G}$. By general position, loops and disks may be pushed
off $K_{G}$, so inclusion induces an isomorphism $\pi_{1}\left(  \partial
N\right)  \cong\pi_{1}\left(  N\right)  \cong G$. By standard algebraic
topology arguments $\partial N$ has the $%
\mathbb{Z}
$-homology of an $\left(  n-1\right)  $-sphere and $C^{n}=\mathbb{S}%
^{n}-\operatorname*{int}N$ has the homology of a point. A second general
position argument shows that $C^{n}$ is simply connected, and thus
contractible---but $C^{n}$ is clearly not a ball. A compact contractible
manifold constructed in this manner is called a \textbf{Newman compact
contractible manifold} and its interior an\textbf{ open Newman manifold}.
\end{example}

\begin{exercise}
Verify the assertions made in the above example. Be prepared to use numerous
tools from a first course in algebraic topology: duality, universal
coefficients, the Hurewicz theorem and a theorem of Whitehead (to name a few).
\end{exercise}

The Newman construction can also be applied to acyclic 3-complexes. From that
observation, one can show that every finitely presented \emph{superperfect
group }$G$ (that is, $H_{i}\left(  G;%
\mathbb{Z}
\right)  =0$ for $i=1,2$) can be realized as $\pi_{1}\left(  \partial
C^{n}\right)  $ for some compact contractible $n$-manifold ($n\geq7$). A
related result \cite{Ke}, \cite{FQ} asserts that \emph{every} $\left(
n-1\right)  $-manifold with the homology of $\mathbb{S}^{n-1}$ bounds a
compact contractible $n$-manifold. For an elementary construction of
$4$-dimensional examples, see \cite{Maz}.

\begin{exercise}
By applying the various Poincar\'{e} Conjectures, show that a compact
contractible $n$-manifold is topologically an $n$-ball if and only if its
boundary is simply connected. (An additional nontrivial tool, the Generalized
Sch\"{o}nflies Theorem, may also be helpful.)
\end{exercise}

A place where open manifolds arise naturally, even in the study of closed
manifolds, is as covering spaces. A place where contractible open manifolds
arise naturally is as universal covers of \emph{aspherical}
manifolds\emph{\footnote{A connected space $X$ is \emph{aspherical }if
$\pi_{k}\left(  X\right)  =0$ for all $k\geq2$.}. }Until 1982, the following
was a major open problem:\medskip

\begin{quotation}
\emph{Does an exotic contractible open manifold ever cover a closed manifold?
Equivalently: Can the universal cover of a closed aspherical manifold fail to
be homeomorphic to} \emph{$%
\mathbb{R}
^{n}$?\medskip}
\end{quotation}

\noindent In dimension $3$ this problem remained open until Perelman's
solution to the Geometrization Conjecture. It is now known that the universal
cover of a closed aspherical $3$-manifold is always homeomorphic to \emph{$%
\mathbb{R}
^{3}$}. In all higher dimensions, a remarkable construction by Davis
\cite{Dav} produced aspherical $n$-manifolds with exotic universal covers.

\begin{example}
[Davis' exotic universal covering spaces]%
\label{Example: Construction of Davis manifolds}The construction begins with
an exotic (piecewise-linear) compact contractible oriented manifold $C^{n}$.
Davis' key insight was that a certain Coxeter group $\Gamma$ determined by a
triangulation of $\partial C^{n}$ provides precise instructions for assembling
infinitely many copies of $C^{n}$ into a contractible open n-manifold
$\mathcal{D}^{n}$ with enough symmetry to admit a proper cocompact\footnote{An
action by $\Gamma$ on $X$ is \emph{proper} if, for each compact $K\subseteq X$
at most finitely many $\Gamma$-translates of $K$ intersect $K$. The action is
\emph{cocompact} if there exists a compact $C$ such that $\Gamma C=X$.} action
by $\Gamma$. Figure \ref{Figure: Davis manifold}
\begin{figure}[ptb]%
\centering
\includegraphics[
height=1.9346in,
width=2.4837in
]%
{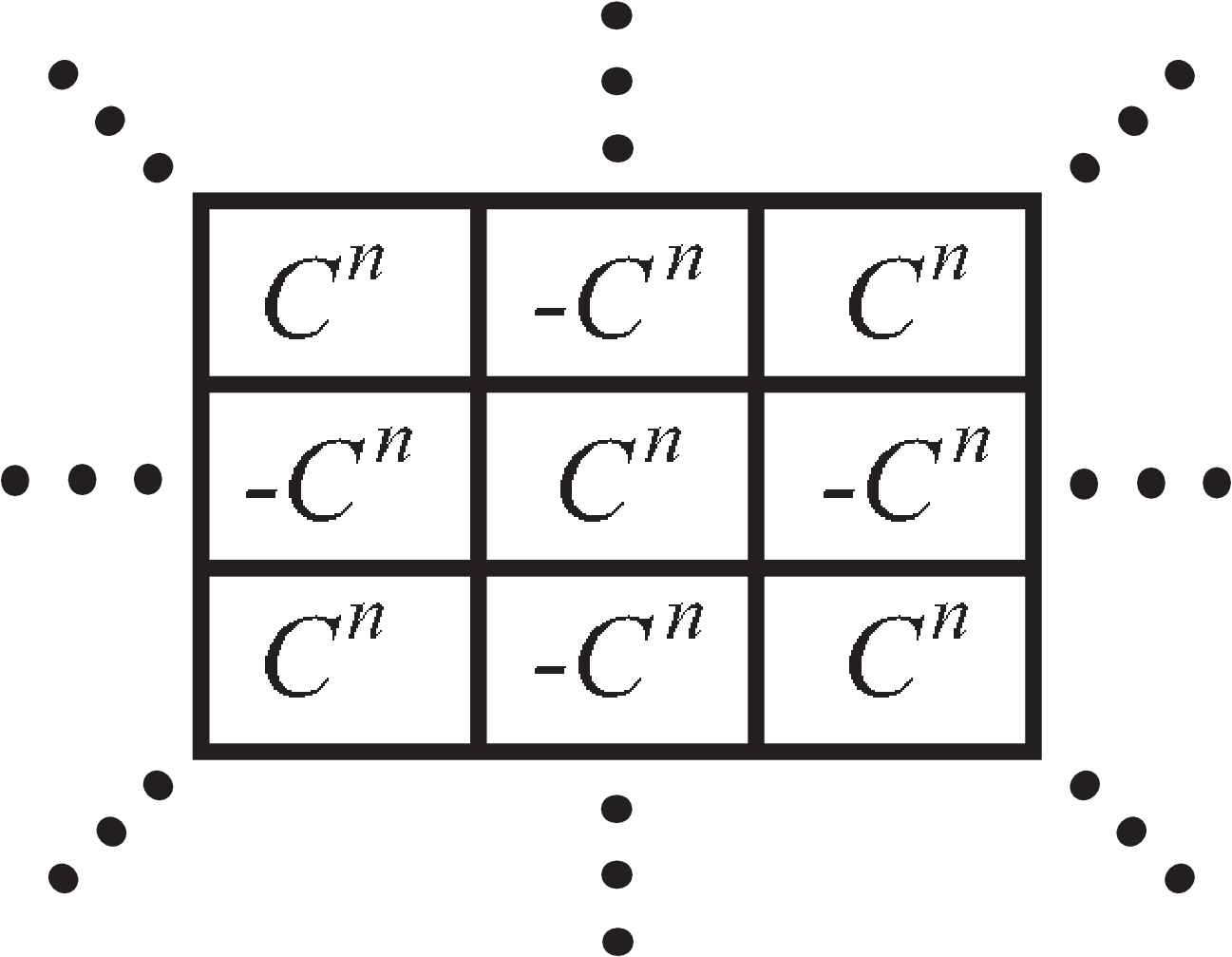}%
\caption{A Davis manifold}%
\label{Figure: Davis manifold}%
\end{figure}
provides a schematic, of $\mathcal{D}^{n}$, where $-C^{n}$ denotes a copy of
$C^{n}$ with reversed orientation. Intuitively, $\mathcal{D}^{n}$ is obtained
by repeatedly reflecting copies of $C^{n}$ across $\left(  n-1\right)  $-balls
in $\partial C^{n}$. The reflections explain the reversed orientations on half
of the copies. By Selberg's Lemma, there is a finite index torsion-free
$\Gamma^{\prime}\leq\Gamma$. By properness, the action of $\Gamma^{\prime}$ on
$\mathcal{D}^{n}$ is \emph{free }(no $\gamma\in\Gamma^{\prime}$ has a fixed
point), so the quotient map $\mathcal{D}^{n}\rightarrow\Gamma^{\prime
}\backslash\mathcal{D}^{n}$ is a covering projection with image a closed
aspherical manifold.\smallskip
\end{example}

Later in these notes, when we prove that $\mathcal{D}^{n}\not \approx
\mathbb{R}
^{n}$, an observation by Ancel and Siebenmann will come in handy. By
discarding all of the beautiful symmetry inherent in the Davis construction,
their observation provides a remarkably simple topological picture of
$\mathcal{D}^{n}$. Toward understanding that picture, let $P^{n}$ and $Q^{n}$
be oriented manifolds with connected boundaries, and let $B,B^{\prime}$ be
$\left(  n-1\right)  $-balls in $\partial P^{n}$ and $\partial Q^{n}$,
respectively. A \emph{boundary connected sum} $P^{n}\overset{\partial
}{\#}Q^{n}$ is obtained by identifying $B$ with $B^{\prime}$ via an
orientation reversing homeomorphism. (By using an orientation reversing gluing
map, we may give $P^{n}\overset{\partial}{\#}Q^{n}$ an orientation that agrees
with both original orientations.).

\begin{theorem}
\cite{AS} \label{Theorem: Ancel-Siebenmann}A Davis manifold $\mathcal{D}^{n}$
constructed from copies of an oriented compact contractible manifold $C^{n}$
is homeomorphic to the interior of an infinite boundary connected sum:
\[
C_{0}^{n}\overset{\partial}{\#}\left(  -C_{1}^{n}\right)  \overset{\partial
}{\#}\left(  C_{2}^{n}\right)  \overset{\partial}{\#}\left(  -C_{3}%
^{n}\right)  \overset{\partial}{\#}\cdots
\]
where each $C_{2i}^{n}$ is a copy of $C^{n}$ and each $-C_{2i+1}^{n}$ is a
copy of $-C^{n}$.
\end{theorem}

\begin{remark}
\emph{The reader is warned that an infinite boundary connected sum is not
topologically well-defined. For example, one could arrange that the result be
2-ended instead of 1-ended. See Figure \ref{Figure: Double bubbles}.
\begin{figure}[ptb]%
\centering
\includegraphics[
height=1.8321in,
width=2.8007in
]%
{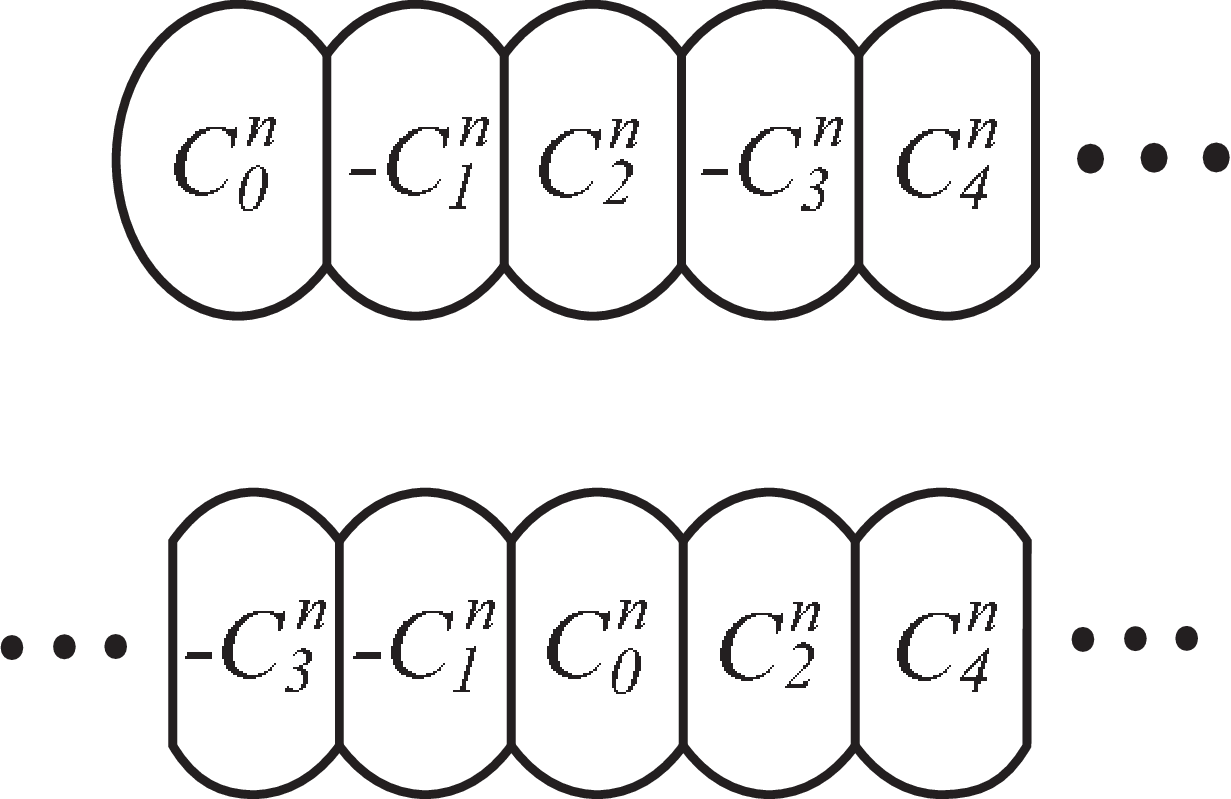}%
\caption{1- and 2-ended boundary connected sums}%
\label{Figure: Double bubbles}%
\end{figure}
Remarkably, the }interior\emph{ of such a sum is well-defined. The proof of
that fact is relatively straight-forward; it contains the essence of Theorem
\ref{Theorem: Ancel-Siebenmann}}.
\end{remark}

\begin{exercise}
Sketch a proof that the 1-ended and 2-ended versions of $C_{0}^{n}%
\overset{\partial}{\#}\left(  -C_{1}^{n}\right)  \allowbreak\overset{\partial
}{\#}\left(  C_{2}^{n}\right)  \allowbreak\overset{\partial}{\#}\left(
-C_{3}^{n}\right)  \allowbreak\overset{\partial}{\#}\cdots$, indicated by
Figure \ref{Figure: Double bubbles} have homeomorphic interiors.
\end{exercise}

\begin{example}
[Asymmetric Davis manifolds]%
\label{An asymmetric variation on the Davis manifolds}To create a larger
collection of exotic contractible open $n$-manifolds (without concern for
whether they are universal covers), the infinite boundary connect sum
construction can be applied to a collection $\left\{  C_{j}^{n}\right\}
_{j=0}^{\infty}$ of non-homeomorphic compact contractible $n$-manifolds. Here
orientations are less relevant, so mention is omitted. Since there are
infinitely many distinct compact contractible $n$-manifolds, this strategy
produces uncountably many examples, which we refer to informally as
\textbf{asymmetric Davis manifolds}. Distinguishing one from another will be a
good test for our soon-to-be-developed tools. Recent applications of these
objects can be found in \cite{Bel} and in the dissertation of P. Sparks.
\end{example}

\begin{exercise}
Show that the interior of an infinite boundary connected sum of compact
contractible $n$-manifolds is contractible.
\end{exercise}

A natural question is motivated by the above discussion:\medskip

\begin{quotation}
\emph{Among the contractible open manifolds described above, which can or
cannot be universal covers of closed }$n$\emph{-manifolds?}\medskip
\end{quotation}

\noindent We will return to this question in
\S \ref{Subsection: Another look at contractible open manifolds}. For now we
settle for a fun observation by McMillan and Thickstun \cite{McTh}.

\begin{theorem}
For each $n\geq3$, there exist exotic contractible open $n$-manifolds that are
not universal covers of any closed $n$-manifold.

\begin{proof}
There are uncountably many exotic open $n$-manifolds and, by \cite{CK}, only
countably many closed $n$-manifolds.
\end{proof}
\end{theorem}

\subsection{Fundamental groups at infinity for the classic
examples\label{Subsection: Examples of fundamental groups at infinity}}

With an ample supply of examples to work with, we begin defining an algebraic
invariant useful for distinguishing one contractible open manifold from
another. Technical issues will arise, but to keep focus on the big picture, we
delay confronting those until later. Once completed, the new invariant will be
more widely applicable, but for now we concentrate on contractible open manifolds.

Let $W^{n}$ be a contractible open manifold with $n\geq2$. Express $W^{n}$ as
$\cup_{i=0}^{\infty}K_{i}$ where each $K_{i}$ is a connected codimension 0
submanifold and $K_{i}\subseteq\operatorname*{int}K_{i+1}$ for each $i$. With
some additional care, arrange that each $K_{i}$ has connected complement.
(Here one uses the fact that $W^{n}$ is contractible and $n\geq2$. See
Exercise \ref{Exercise: contractible open mflds are 1-ended}.) The
corresponding \emph{neighborhoods of infinity} are the sets $U_{i}%
=\overline{W^{n}-K_{i}}$.

For each $i$, let $p_{i}\in U_{i}$ and consider the inverse sequence of groups:%

\begin{equation}
\pi_{1}\left(  U_{0},p_{0}\right)  \overset{\lambda_{1}}{\longleftarrow}%
\pi_{1}\left(  U_{1},p_{1}\right)  \overset{\lambda_{2}}{\longleftarrow}%
\pi_{1}\left(  U_{2},p_{2}\right)  \overset{\lambda_{3}}{\longleftarrow}%
\cdots\tag{3.1}\label{line: initial fundamental group at infinity}%
\end{equation}
We would \emph{like} to think of the $\lambda_{i}$ as being induced by
inclusion, but since $\cap_{i=0}^{\infty}U_{i}=\varnothing$, a single choice
of base point is impossible. Instead, for each $i$ choose a path $\alpha_{i}$
in $U_{i}$ connecting $p_{i}$ to $p_{i+1}$; then declare $\lambda_{i}$ to be
the composition
\[
\pi_{1}\left(  U_{i-1},p_{i-1}\right)  \overset{\widehat{\alpha}%
_{i-1}}{\longleftarrow}\pi_{1}\left(  U_{i-1},p_{i}\right)  \leftarrow\pi
_{1}\left(  U_{i},p_{i}\right)
\]
where the first map is induced by inclusion and $\widehat{\alpha}_{i-1}$ is
the \textquotedblleft change of base point isomorphism\textquotedblright. By
assembling the $\alpha_{i}$ end-to-end, we can define a map $r:[0,\infty
)\rightarrow X$, called the \emph{base ray}. The entire inverse sequence
(\ref{line: initial fundamental group at infinity}) is taken as a
representation of the \emph{fundamental group at infinity (based at}
$r$\emph{)} of $W^{n}$. Those who prefer a single group can take an inverse
limit (defined in \S \ref{Subsection: defining pro-isomorphism}) to obtain the
\emph{\v{C}ech fundamental group at infinity} (based at $r$). Unfortunately,
that inverse limit typically contains far less information than the inverse
sequence itself---more on that later.

Two primary technical issues are already evident:

\begin{itemize}
\item \textbf{well-definedness: }most obviously, the groups found in
(\ref{line: initial fundamental group at infinity}) depend upon the chosen
neighborhoods of infinity, and

\item \textbf{dependence upon base ray: }the \textquotedblleft bonding
homomorphisms\textquotedblright\ in
(\ref{line: initial fundamental group at infinity}) depend upon the base ray.
\end{itemize}

\noindent We will return to these issues soon; for now we forge ahead and
apply the basic idea to some examples.

\stepcounter{theorem}

\begin{example}
[Fundamental group at infinity for $%
\mathbb{R}
^{n}$]\label{Example: Fundamental group at infinity for R^n}Express $%
\mathbb{R}
^{n}$ as $\cup_{i=0}^{\infty}i\mathbb{B}^{n}$ where $i\mathbb{B}^{n}$ is the
closed ball of radius $i$. Then, $U_{0}=%
\mathbb{R}
^{n}$ and for $i>0$, $U_{i}=\overline{%
\mathbb{R}
^{n}-\mathbb{B}_{i}^{n}}$ is homeomorphic to $\mathbb{S}^{n-1}\times\lbrack
i,\infty)$. If we let $r$ be a true ray emanating from the origin and
$p_{i}=r\cap\left(  \mathbb{S}^{n-1}\times\left\{  i\right\}  \right)  $ we
get a representation of the fundamental group at infinity as
\begin{equation}
1\leftarrow1\leftarrow1\leftarrow1\leftarrow\cdots\tag{3.2}%
\label{Sequence: trivial inverse sequence}%
\end{equation}
when $n\geq3$, and when $n=2$, we get (with a slight abuse of notation)
\begin{equation}
1\leftarrow%
\mathbb{Z}
\overset{\operatorname*{id}}{\longleftarrow}%
\mathbb{Z}
\overset{\operatorname*{id}}{\longleftarrow}\mathbb{%
\mathbb{Z}
}\overset{\operatorname*{id}}{\longleftarrow}\cdots\tag{3.3}%
\end{equation}
Modulo the technical issues, we have a modest application of the fundamental
group at infinity---it distinguishes the plane from higher-dimensional
Euclidean spaces.
\end{example}

\stepcounter{theorem}

\stepcounter{theorem}

\begin{example}
[Fundamental group at infinity for open Newman manifolds]%
\label{Example: Fundamental group at infinty for open Newman manifolds}Let
$C^{n}$ be a compact contractible $n$-manifold and $G=\pi_{1}\left(  \partial
C^{n}\right)  $. By deleting $\partial C^{n}$ from a collar neighborhood of
$\partial C^{n}$ in $C^{n}$ we obtain an \emph{open collar }neighborhood of
infinity $U_{0}\approx\partial C^{n}\times\lbrack0,\infty)$ in the open Newman
manifold $\operatorname*{int}C^{n}$. For each $i\geq1$, let $U_{i}$ be the
subcollar corresponding to $\partial C^{n}\times\lbrack i,\infty)$ and let $r$
to be the ray $\left\{  p\right\}  \times\lbrack0,\infty)$, with
$p_{i}=p\times\left\{  i\right\}  $. We get a representation of the
fundamental group at infinity
\[
G\overset{\operatorname*{id}}{\longleftarrow}G\overset{\operatorname*{id}%
}{\longleftarrow}G\overset{\operatorname*{id}}{\longleftarrow}\cdots
\]
The (still-to-be-quantified) difference between this and
(\ref{Sequence: trivial inverse sequence}) verifies that $\operatorname*{int}%
C^{n}$ is not homeomorphic to $%
\mathbb{R}
^{n}$.
\end{example}

\begin{example}
[Fundamental group at infinity for Davis manifolds]%
\label{Example: Fundamental group at infinity for Davis manifolds}To aid in
obtaining a representation of the fundamental group at infinity of a Davis
manifold $\mathcal{D}^{n}$, we use Theorem \ref{Theorem: Ancel-Siebenmann} to
view $\mathcal{D}^{n}$ as the interior of $C_{0}\overset{\partial}{\#}\left(
-C_{1}\right)  \overset{\partial}{\#}\left(  C_{2}\right)  \overset{\partial
}{\#}\left(  -C_{3}\right)  \overset{\partial}{\#}\cdots$, where each $C_{i}$
is a copy of a fixed compact contractible n-manifold $C$. (Superscripts
omitted to avoid excessive notation.)

Borrow the setup from Example
\ref{Example: Fundamental group at infinty for open Newman manifolds} to
express $\operatorname*{int}C$ as $\cup_{i=0}^{\infty}K_{i}$ where each
$K_{i}\equiv\overline{\operatorname*{int}C-U_{i}}$ is homeomorphic to $C^{n}$.
We may exhaust $\mathcal{D}^{n}$ by compact contractible manifolds
$L_{i}\approx C_{0}\overset{\partial}{\#}(-C_{1})\overset{\partial}{\#}%
\cdots\overset{\partial}{\#}(\pm C_{i})$ created by \textquotedblleft tubing
together\textquotedblright\ $K_{i}^{0}\cup(-K_{i}^{1})\cup\cdots\cup(\pm
K_{i}^{i})$, where the tubes are copies of $\mathbb{B}^{n-1}\times
\lbrack-1,1]$ and $K_{i}^{j}$ is the copy of $K_{i}$ in $\pm C_{j}$ See Figure
\ref{Figure: Exhaustion by compact contractibles}.%
\begin{figure}[ptb]%
\centering
\includegraphics[
height=1.0395in,
width=3.3269in
]%
{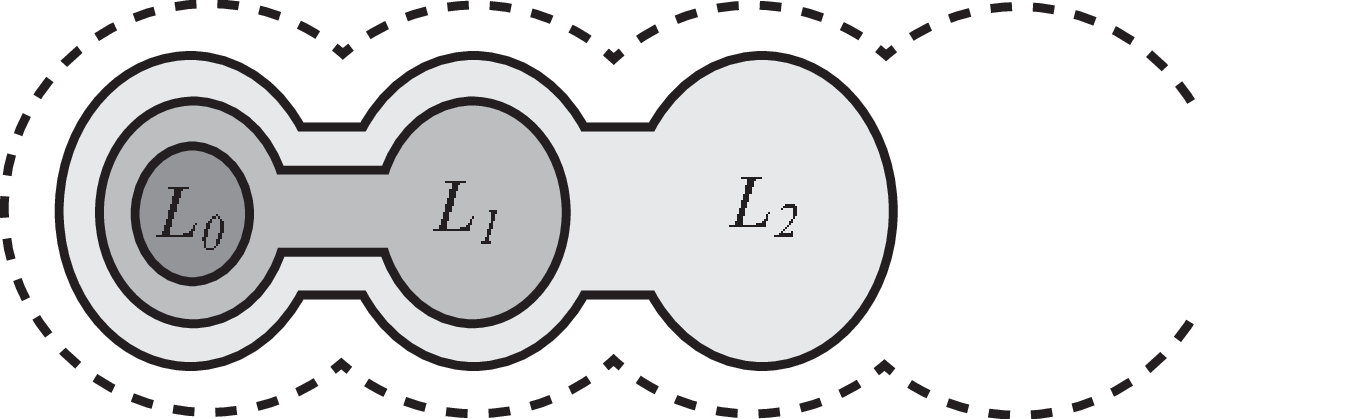}%
\caption{An exhaustion of $\mathcal{D}^{n}$ by compact contractible
manifiolds}%
\label{Figure: Exhaustion by compact contractibles}%
\end{figure}
It is easy to see that a corresponding neighborhood of infinity $V_{i}%
=\overline{\mathcal{D}^{n}-L_{i}}$ has fundamental group $G_{0}\ast G_{1}%
\ast\cdots\ast G_{i}$ where each $G_{i}$ is a copy of $G$; moreover, the
homomorphism of $G_{0}\ast G_{1}\ast\cdots\ast G_{i}\ast G_{i+1}$ to
$G_{0}\ast G_{1}\ast\cdots\ast G_{i}$ induced by $V_{i+1}\hookrightarrow
V_{i}$ acts as the identity on $G_{0}\ast G_{1}\ast\cdots\ast G_{i}$ and sends
$G_{i+1}$ to $1$. With appropriate choices of base points and ray, we arrive
at a representation of the fundamental group at infinity of $\mathcal{D}^{n}$
of the form%
\begin{equation}
G_{0}\twoheadleftarrow G_{0}\ast G_{1}\twoheadleftarrow G_{0}\ast G_{1}\ast
G_{2}\twoheadleftarrow G_{0}\ast G_{1}\ast G_{2}\ast G_{3}\twoheadleftarrow
\cdots.\tag{3.4}%
\label{line: fundamental group at infinity for a Davis manifold}%
\end{equation}

\end{example}

\stepcounter{theorem}

\begin{example}
[Fundamental group at infinity for asymmetric Davis manifolds]By proceeding as
in Example \ref{Example: Fundamental group at infinity for Davis manifolds},
but not requiring $C_{j}\approx C_{k}$ for $j\neq k$, we obtain manifolds with
fundamental groups at infinity represented by inverse sequences like
(\ref{line: fundamental group at infinity for a Davis manifold}), except that
the various $G_{i}$ need not be the same. By choosing different sequences of
compact contractible manifolds, we can arrive at an uncountable collection of
inverse sequences. Some work is still necessary in order to claim an
uncountable collection of topologically distinct manifolds.
\end{example}

\begin{example}
[Fundamental group at infinity for the Whitehead manifold]%
\label{Example: Fundamental group at infinity for the Whitehead manifold}%
Referring to Example \ref{Example: Definition of Whitehead manifold} and
Figure \ref{Figure: Whitehead}, for each $i\geq0$, let $A_{i}=\overline
{T_{i}-T_{i+1}}$. Then $A_{i}$ is a compact 3-manifold, with a pair of torus
boundary components $\partial T_{i}$ and $\partial T_{i+1}$. Standard
techniques from 3-manifold topology allow one to show that $G=\pi_{1}\left(
A_{i}\right)  $ is nonabelian and that each boundary component is
incompressible in $A_{i}$, i.e., $\pi_{1}\left(  \partial T_{i}\right)  $ and
$\pi_{1}\left(  \partial T_{i+1}\right)  $ inject into $G$. If we let $A_{-1}$
be the solid torus $\overline{\mathbb{S}^{3}-T_{0}}$, then
\[
\mathcal{W}^{3}=A_{-1}\cup A_{0}\cup A_{1}\cup A_{2}\cup\cdots
\]
where $A_{i}\cap A_{i+1}=T_{i+1}$ for each $i$. Set $U_{i}=A_{i}\cup
A_{i+1}\cup A_{i+2}\cup\cdots$, for each $i\geq0$, to obtain a nested sequence
of homeomorphic neighborhoods of infinity, each having fundamental group
isomorphic to an infinite free product with amalgamation
\[
\pi_{1}\left(  U_{i}\right)  =G_{i}\ast_{\Lambda}G_{i+1}\ast_{\Lambda}%
G_{i+2}\ast_{\Lambda}\cdots
\]
where $\Lambda\cong%
\mathbb{Z}
\oplus%
\mathbb{Z}
$. Assembling these into an inverse sequence (temporarily ignoring base ray
issues) gives a representation of the fundamental group at infinity%
\[
G_{0}\ast_{\Lambda}G_{1}\ast_{\Lambda}G_{2}\ast_{\Lambda}G_{4}\ast_{\Lambda
}\cdots\leftarrowtail G_{1}\ast_{\Lambda}G_{2}\ast_{\Lambda}G_{3}\ast
_{\Lambda}\cdots\leftarrowtail G_{2}\ast_{\Lambda}G_{3}\ast_{\Lambda}%
\cdots\leftarrowtail\cdots
\]
Combinatorial group theory provides a useful observation: each bonding
homomorphism is injective and none is surjective.
\end{example}

We will return to the calculations from this section after enough mathematical
rigor has been added to make them fully applicable.

\section{Basic notions in the study of noncompact spaces}

An important short-term goal is to confront the issue of well-definedness and
to clarify the role of the base ray in our above approach to the fundamental
group at infinity. Until that is done, the calculations in the previous
section should be viewed with some skepticism. Since we will eventually
broaden our scope to spaces far more general than contractible open manifolds,
we first take some time to lay out a variety general facts and definitions of
use in the study of noncompact spaces.

\subsection{Neighborhoods of infinity and ends of spaces}

A subset $U$ of a space $X$ is a \emph{neighborhood of infinity} if
$\overline{X-U}$ is compact; a subset of $X$ is \emph{unbounded} if its
closure is noncompact. (\emph{Note:} This differs from the metric notion of
\textquotedblleft unboundedness\textquotedblright, which is dependent upon the
metric.) We say that $X$ \emph{has }$k$ \emph{ends}, if $k$ is a least upper
bound on the number of unbounded components in a neighborhood of infinity. If
no such $k$ exists, we call $X$ \emph{infinite-ended}.

\begin{example}
The real line has $2$ ends while, for all $n\geq2$, $%
\mathbb{R}
^{n}$ is 1-ended. A space is compact if and only if it is $0$-ended. A common
example of an infinite-ended space is the universal cover of $\mathbb{S}%
^{1}\vee\mathbb{S}^{1}$.
\end{example}

\begin{exercise}
\label{Exercise: Ends of spaces admitting actions}Show that an ANR $X$ that
admits a proper action by an infinite group $G$, necessarily has $1,2,$ or
infinitely many ends. (This is a key ingredient in an important theorem from
geometric group theory. See \S \ref{Section: Exploring the ends of groups}.)
\end{exercise}

\begin{exercise}
\label{Exercise: contractible open mflds are 1-ended}Show that a contractible
open $n$-manifold of dimension $\geq2$ is always 1-ended. \emph{Hint:}\textbf{
}Ordinary singular or simplicial homology will suffice.
\end{exercise}

An \emph{exhaustion of }$X$ \emph{by compacta} is a nested sequence
$K_{0}\subseteq K_{1}\subseteq K_{2}\subseteq$ $\ $\ of compact subsets whose
union is $X$; in this case the corresponding collection of neighborhoods of
infinity $U_{i}=X-K_{i}$ is \emph{cofinal}, i.e., $\cap_{i=0}^{\infty}%
U_{i}=\varnothing$.\footnote{Sometimes closed neighborhood of infinity are
preferable; then we let $U_{i}=\overline{X-K_{i}}$. In many cases the choice
is just a matter of personal preference.} A compactum $K_{i}$ is
\emph{efficient} if it is connected and the corresponding $U_{i}$ has only
unbounded components. An exhaustion of $X$ by efficient compacta with each
$K_{i}\subseteq\operatorname*{int}K_{i+1}$ is called an \emph{efficient
exhaustion}. The following is an elementary, but nontrivial, exercise in
general topology.

\begin{exercise}
Show that every connected ANR $X$ admits an efficient exhaustion by compacta.
\emph{Note:}\textbf{ }For this exercise, one can replace the ANR hypothesis
with the weaker assumption of locally compact and locally path connected.
\end{exercise}

Let $\left\{  K_{i}\right\}  _{i=0}^{\infty}$ be an efficient exhaustion of
$X$ by compacta and, for each $i$, let $U_{i}=X-K_{i}$. Let $\mathcal{E}%
\emph{nds}\left(  X\right)  $ be the set of all sequences $\left(  V_{0}%
,V_{1},V_{2},\cdots\right)  $ where $V_{i}$ is a component of $U_{i}$ and
$V_{0}\supseteq V_{1}\supseteq\cdots$. Give $\overline{X}=X\cup\mathcal{E}%
\emph{nds}\left(  X\right)  $ the topology generated by the basis consisting
of all open subsets of $X$ and all sets $\overline{V}_{i}$ where
\[
\overline{V}_{i}=V_{i}\cup\left\{  \left(  W_{0},W_{1},\cdots\right)
\in\mathcal{E}\emph{nds}\left(  X\right)  \mid W_{i}=V_{i}\right\}  \text{.}%
\]
Then $\overline{X}$ is separable, compact, and metrizable; it is known as the
\emph{Freudenthal com\-pact\-ific\-at\-ion }of $X$.

\begin{exercise}
Verify the assertions made in the final sentence of the above paragraph. Then
show that any two efficient exhaustions of $X$ by compacta result in
com\-pact\-ific\-at\-ions that are canonically homeomorphic.
\end{exercise}

\begin{exercise}
Show that the cardinality of $\mathcal{E}nds\left(  X\right)  $ agrees with
the \textquotedblleft number of ends of $X$\textquotedblright\ defined at the
beginning of this section.\medskip
\end{exercise}

A \emph{closed [open] neighborhood of infinity} in $X$ is one that is closed
[open] as a subset of $X$. If $X$ is an ANR, we often prefer neighborhoods of
infinity to themselves be ANRs. This is automatic for open, but not for closed
neighborhoods of infinity. Call a neighborhood of infinity \emph{sharp }if it
is closed and also an ANR. Call a space $X$ \emph{sharp at infinity} if it
contains arbitrarily small sharp neighborhoods of infinity, i.e., if every
neighborhood of infinity in $X$ contains one that is sharp.

\begin{example}
Manifolds, locally finite polyhedra, and finite-dimensional locally finite CW
complexes are sharp at infinity---they contain arbitrarily small closed
neighborhoods of infinity that are themselves manifolds with boundary, locally
finite polyhedra, and locally finite CW complexes, respectively. In a similar
manner, Hilbert cube manifolds are sharp at infinity by an application of
Theorem \ref{Theorem: Triangulability of HCMs}. The existence of non-sharp
ANRs can be deduced from \cite{Bo0} and \cite{Mol}.
\end{example}

\begin{example}
\label{Example: Proper CAT(0) spaces are sharp at infinity}Every proper CAT(0)
space $X$ is sharp at infinity---but this is not entirely obvious. The most
natural closed neighborhood of infinity, $N_{p,r}=X-B\!\left(  p;r\right)  $,
is an ANR if and only if the metric sphere $S\!\left(  p;r\right)  $ is an
ANR. Surprisingly, it is not known whether this is always the case. However,
we can fatten $N_{p,r}$ to an ANR by applying Exercise
\ref{Exercise: union of balls in a CAT(0) space}.
\end{example}

\begin{problem}
In a proper CAT(0) space $X$, is each $S\!\left(  p;r\right)  $ an ANR? Does
there exist some $p_{0}\in X$ and a sequence of arbitrarily large $r_{i}$ for
which each $S\!\left(  p_{0};r_{i}\right)  $ is an ANR? Does it help to assume
that $X$ is finite-dimensional or that $X$ is a manifold?\medskip
\end{problem}

An especially nice variety of sharp neighborhood of infinity is available in
$n$-manifolds and Hilbert cube manifolds. A closed neighborhood of infinity
$N\subseteq M^{n}$ in an $n$-manifold with compact boundary is \emph{clean} if
it is a codimension 0 submanifold disjoint from $\partial M^{n}$ and $\partial
N=\operatorname*{Bd}_{M^{n}}N$ has a product neighborhood ($\approx\partial
N\times\left[  -1,1\right]  $) in $M^{n}$. In a Hilbert cube manifold $X$,
where there is no intrinsic notion of boundary (recall that $\mathcal{Q}$
itself is homogeneous!), we simply require that $\operatorname*{Bd}_{X}N$ be a
Hilbert cube manifold with a product neighborhood in $X$. In an $n$-manifold
with noncompact boundary a natural, but slightly more complicated, definition
is possible; but it is not needed in these notes.

\subsection{Proper maps and proper homotopy type}

A map $f:X\rightarrow Y$ is \emph{proper\footnote{Yes, this is our third
distinct mathematical use of the word \emph{proper}!}} if $f^{-1}\left(
C\right)  $ is compact for all compact $C\subseteq Y$.

\begin{exercise}
Show that a map $f:X\rightarrow Y$ between locally compact metric spaces is
proper if and only if the obvious extension to their 1-point
com\-pact\-ific\-at\-ions is continuous.
\end{exercise}

Maps $f_{0},f_{1}:X\rightarrow Y$ are \emph{properly homotopic} is there is a
proper map $H:X\times\left[  0,1\right]  \rightarrow Y$, with $H_{0}=f_{0}$
and $H_{1}=f_{1}$. We call $H$ a \emph{proper homotopy} between $f_{0}$ and
$f_{1}$ and write $f_{0}\overset{p}{\simeq}$ $f_{1}$. We say that
$f:X\rightarrow Y$ is a \emph{proper homotopy equivalence} if there exists
$g:Y\rightarrow X$ such that $gf\overset{p}{\simeq}\operatorname*{id}_{X}$ and
$fg\overset{p}{\simeq}Y$. In that case we say $X$ and $Y$ are \emph{proper
homotopy equivalent }and write $X\overset{p}{\simeq}Y$.

\begin{remark}
\label{Remark: importance of properness}\emph{It is immediate that
homeomorphisms are both proper maps and proper homotopy equivalences, but many
pairs of spaces that are homotopy equivalent in the traditional sense are not}
proper \emph{homotopy equivalent. For example, whereas all contractible open
manifolds (indeed, all contractible spaces) are homotopy equivalent, they are
frequently distinguished by their proper homotopy types.}

\emph{It would be impossible to overstate the importance of \textquotedblleft
properness\textquotedblright\ in the study of noncompact spaces. Indeed, it is
useful to think in terms of the} proper categories \emph{where the objects are
spaces (or certain subclasses of spaces) and the morphisms are proper maps or
proper homotopy classes of maps. In the latter case, the isomorphisms are
precisely the proper homotopy equivalences. Most of the invariants defined in
these notes (such as the fundamental group at infinity) can be viewed as
functors on the proper homotopy category of appropriate spaces.}
\end{remark}

The following offers a sampling of the usefulness of proper maps in
understanding noncompact spaces.

\begin{proposition}
\label{Prop: Induced maps on ends}Let $f:X\rightarrow Y$ be a proper map
between ANRs. Then

\begin{enumerate}
\item $f$ induces a canonical function $f^{\ast}:\mathcal{E}\!nds\left(
X\right)  \rightarrow\mathcal{E}\!nds\left(  Y\right)  $ that may be used to
extend $f$ to a map $\overline{f}:\overline{X}\rightarrow\overline{Y}$ between
Freudenthal com\-pact\-ific\-at\-ions,

\item if $f_{0},f_{1}:X\rightarrow Y$ are properly homotopic, then
$f_{0}^{\ast}=f_{1}^{\ast}$, and

\item if $f:X\rightarrow Y$ is a proper homotopy equivalence, then $f^{\ast}$
is a bijection.
\end{enumerate}

\begin{proof}
Begin with efficient exhaustions $\left\{  K_{i}\right\}  $ and $\left\{
L_{i}\right\}  $ of $X$ and $Y$, respectively. The following simple
observations make the uniqueness and well-definedness of $f^{\ast}$ straight-forward:

\begin{itemize}
\item By properness, for each $i$, there is a $k_{i}$ such that $f\left(
X-K_{k_{i}}\right)  \subseteq Y-L_{i}$,

\item By connectedness, a given component $U_{i}$ of $X-K_{k_{i}}$ is sent
into a unique component $V_{i}$ of $Y-L_{i}$,

\item By nestedness, each entry $W_{j}$ of $\left(  W_{0},W_{1},\cdots\right)
\in\mathcal{E}\emph{nds}\left(  X\right)  $ determines all entries of lower
index; hence every subsequence of entries determines that element.
\end{itemize}
\end{proof}
\end{proposition}

\begin{exercise}
Fill in the remaining details in the proof of Proposition
\ref{Prop: Induced maps on ends}.
\end{exercise}

The following observation is a key sources of proper maps and proper homotopy equivalences.

\begin{proposition}
\label{Proposition: proper lifts}Let $f:X\rightarrow Y$ be a proper map
between connected ANRs inducing an isomorphism on fundamental groups. Then the
lift $\widetilde{f}:\widetilde{X}\rightarrow\widetilde{Y}$ to universal covers
is a proper map. If $f:X\rightarrow Y$ is a proper homotopy equivalence, then
$\widetilde{f}:\widetilde{X}\rightarrow\widetilde{Y}$ is a proper homotopy equivalence.
\end{proposition}

\begin{corollary}
\label{Corollary: proper homotopy equivalence of covering spaces}If
$f:X\rightarrow Y$ is a homotopy equivalence between compact connected ANRs,
then $\widetilde{f}:\widetilde{X}\rightarrow\widetilde{Y}$ is a proper
homotopy equivalence.
\end{corollary}

We prove a simpler Lemma that leads directly to Corollary
\ref{Corollary: proper homotopy equivalence of covering spaces} and contains
the ideas needed for Proposition \ref{Proposition: proper lifts}. A different
approach and more general results can be found in \cite[\S 10.1]{Ge2}.

\begin{lemma}
\label{Lemma: lift of map between ANRs}If $k:A\rightarrow B$ is a map between
compact connected ANRs inducing an isomorphism on fundamental groups, then the
lift $\widetilde{k}:\widetilde{A}\rightarrow\widetilde{B}$ between universal
covers is proper.
\end{lemma}

\begin{proof}
[Proof of Lemma \ref{Lemma: lift of map between ANRs}]Let $G$ denote $\pi
_{1}\left(  A\right)  \cong\pi_{1}\left(  B\right)  $. Then $G$ acts by
covering transformations (properly, cocompactly and freely) on $\widetilde{A}$
and $\widetilde{B}$ so that $\widetilde{k}$ is $G$-equivariant. Let
$K\subseteq\widetilde{A}$ and $L\subseteq\widetilde{B}$ be compacta such that
$GK=\widetilde{A}$ and $GL=\widetilde{B}$; without loss of generality, arrange
that $G\cdot\operatorname*{int}\left(  L\right)  =\widetilde{B}$ and
$\widetilde{k}\left(  K\right)  \subseteq L$. The assertion follows easily if
$\widetilde{k}^{-1}\left(  L\right)  $ is compact. Suppose otherwise. Then
there exists a sequence $\left\{  g_{i}\right\}  _{i=1}^{\infty}$ of distinct
element of $G$ for which $g_{i}K\cap\widetilde{k}^{-1}\left(  L\right)
\neq\varnothing$. But then each $g_{i}L$ intersects $L$, contradicting properness.
\end{proof}

\begin{exercise}
Fill in the remaining details for a proof for Proposition
\ref{Proposition: proper lifts}.
\end{exercise}

\subsection{Proper rays}

Henceforth, we refer to any proper map $r:[a,\infty)\rightarrow X$ as a
\emph{proper ray} in $X$. In particular, we do not require a proper ray to be
\textquotedblleft straight\textquotedblright\ or even an embedding. A
\emph{reparametrization} $r^{\prime}$ of a proper ray $r$ is obtained
precomposing $r$ with a homeomorphism $h:[b,\infty)\rightarrow\lbrack
a,\infty)$. Note that a reparametrization of a proper ray is proper.

\begin{exercise}
\label{Exercise: Fixing a proper ray}Show that the base ray $r:[0,\infty
)\rightarrow X$ described in
\S \ref{Subsection: Examples of fundamental groups at infinity} is proper.
Conversely, let $s:[0,\infty)\rightarrow X$ be a proper ray, $\left\{
K_{i}\right\}  _{i=0}^{\infty}$ an efficient exhaustion of $X$ by compacta,
and for each $i$, $U_{i}=X-K_{i}$. Show that, by omitting an initial segment
$[0,a)$ and then reparametrizing $\left.  s\right\vert _{[a,\infty)}$, we may
obtain a corresponding proper ray $r:[0,\infty)\rightarrow X$ with $r\left(
\left[  i,i+1\right]  \right)  \subseteq U_{i}$ for each $i$. In this way, any
proper ray in $X$ can be used as a base ray for a representation of the
fundamental group at infinity.
\end{exercise}

Declare proper rays $r,s:[0,\infty)\rightarrow X$ to be \emph{strongly
equivalent} if they are properly homotopic, and \emph{weakly equivalent }if
there is a proper homotopy $K:%
\mathbb{N}
\times\lbrack0,1]\rightarrow X$ between $\left.  r\right\vert _{%
\mathbb{N}
}$ and $\left.  s\right\vert _{%
\mathbb{N}
}$. Equivalently, $r$ and $s$ are weakly equivalent if there is a proper map
$h$ of the \emph{infinite ladder} $L_{[0,\infty)}=([0,\infty)\times\left\{
0,1\right\}  )\cup(%
\mathbb{N}
\times\lbrack0,1])$ into $X$, with $\left.  h\right\vert _{[0,\infty)\times
0}=r$ and $\left.  h\right\vert _{[0,\infty)\times1}=s$. Properness ensures
that rungs near the end of $L_{[0,\infty)}$ map toward the end of $X$. When
the squares in the ladder can be filled in with a proper collection of 2-disks
in $X$, a weak equivalence can be promoted to a strong equivalence.

For the set of all proper rays in $X$ with domain $[0,\infty)$, let
$\mathcal{E}\left(  X\right)  $ be the set of weak equivalence classes and
$\mathcal{SE}\left(  X\right)  $ the set of strong equivalence classes. There
is an obvious surjection $\Phi:\mathcal{SE}\left(  X\right)  \rightarrow
\mathcal{E}\left(  X\right)  $. We say that $X$ is \emph{connected at
infinity} if $\left\vert \mathcal{E}\left(  X\right)  \right\vert =1$ and
\emph{strongly connected at infinity} if $\left\vert \mathcal{SE}\left(
X\right)  \right\vert =1$.

\begin{exercise}
Show that, for ANRs, there is a one-to-one correspondence between
$\mathcal{E}\left(  X\right)  $ and $\mathcal{E}\emph{nds}\left(  X\right)  $.
(Hence, proper rays provide an alternative, and more geometric, method for
defining the ends of a space.)
\end{exercise}

\begin{exercise}
Show that, for the infinite ladder $L_{[0,\infty)}$, $\Phi:\mathcal{SE}\left(
L_{[0,\infty)}\right)  \rightarrow\mathcal{E}\left(  L_{[0,\infty)}\right)  $
is not injective. In fact $\mathcal{SE}\left(  L_{[0,\infty)}\right)  $ is
uncountable. (This is the prototypical example where $\mathcal{SE}\left(
X\right)  $ differs from $\mathcal{E}\left(  X\right)  .$)
\end{exercise}

\subsection{Finite domination and homotopy type}

In addition to properness, there are notions related to homotopies and
homotopy types that are of particular importance in the study of noncompact
spaces. We introduce some of those here.

A space $Y$ has \emph{finite homotopy type} if it is homotopy equivalent to a
finite CW complex; it is \emph{finitely dominated }if there is a finite
complex $K$ and maps $u:Y\rightarrow K$ and $d:K\rightarrow Y$ such that
$d\circ u\simeq\operatorname*{id}_{Y}$. In this case, the map $d$ is called a
\emph{domination} and we say that $K$ dominates $Y.$

\begin{proposition}
\label{Prop: properties of finitely dominated spaces}Suppose $Y$ is finitely
dominated with maps $u:Y\rightarrow K$ and $d:K\rightarrow Y$ satisfying the
definition. Then

\begin{enumerate}
\item $H_{k}\left(  Y;%
\mathbb{Z}
\right)  $ is finitely generated for all $k$,

\item $\pi_{1}(Y,y_{0})$ is finitely presentable, and

\item if $Y^{\prime}$ is homotopy equivalent to $Y$, then $Y^{\prime}$ is
finitely dominated.
\end{enumerate}

\begin{proof}
Since $d$ induces surjections on all homology and homotopy groups, the finite
generation of $H_{k}\left(  Y;%
\mathbb{Z}
\right)  $ and $\pi_{1}(Y,y_{0})$ are immediate. The finite presentability of
the latter requires some elementary combinatorial group theory; an argument
(based on \cite{Wa65}) can be found in \cite[Lemma 2]{Gu1}. The final item is
left as an exercise.
\end{proof}
\end{proposition}

\begin{exercise}
Show that if $Y^{\prime}$ is homotopy equivalent to $Y$ and $Y$ is finitely
dominated, then $Y^{\prime}$ is finitely dominated.
\end{exercise}

The next proposition adds some intuitive meaning to finite domination.

\begin{proposition}
\label{Prop: finitely dominated = pulling homotopy}An ANR $Y$ is finitely
dominated if and only if there exists a self-homotopy that \textquotedblleft
pulls $Y$ into a compact subset\textquotedblright, i.e., $H:Y\times
\lbrack0,1]\rightarrow Y$ such that $H_{0}=\operatorname*{id}_{Y}$ and
$\overline{H_{1}\left(  Y\right)  }$ is compact.

\begin{proof}
If $u:Y\rightarrow K$ and $d:K\rightarrow Y$ satisfy the definition of finite
domination, then the homotopy between $\operatorname*{id}_{Y}$ and $d\circ u$
pulls $Y$ into $d\left(  K\right)  $.

For the converse, begin by assuming that $Y$ is a locally finite polyhedron.
If $H:Y\times\lbrack0,1]\rightarrow X$ such that $H_{0}=\operatorname*{id}%
_{X}$ and $\overline{H_{1}\left(  Y\right)  }$ is compact, then any compact
polyhedral neighborhood $K$ of $\overline{H_{1}\left(  Y\right)  }$ dominates
$Y$, with $u=H_{1}$ and $d$ the inclusion.

For the general case, we use some Hilbert cube manifold magic. By Theorem
\ref{Theorem: Edwards HCM Theorem}, $Y\times\mathcal{Q}$ is a Hilbert cube
manifold, so by Theorem \ref{Theorem: Triangulability of HCMs}, $Y\times
\mathcal{Q}\approx P\times\mathcal{Q}$, where $P$ is a locally finite
polyhedron. The homotopy that pulls $Y$ into a compact set can be used to pull
$P\times\mathcal{Q}$ into a compact subset of the form $K\times\mathcal{Q}$,
where $K$ is a compact polyhedron. It follows easily that $K$ dominates
$P\times\mathcal{Q}$. An application of Proposition
\ref{Prop: properties of finitely dominated spaces} completes the proof.
\end{proof}
\end{proposition}

At this point, the natural question becomes: \emph{Does there exist a finitely
dominated space }$Y$\emph{ that does not have finite homotopy type?} A version
of this question was initially posed by Milnor in 1959 and answered
affirmatively by Wall.

\begin{theorem}
[Wall's finiteness obstruction, \cite{Wa65}]For each finitely dominated space
$Y$, there is a well-defined obstruction $\sigma\left(  Y\right)  $, lying in
the reduced projective class group $\widetilde{K}_{0}\left(
\mathbb{Z}
\left[  \pi_{1}\left(  Y\right)  \right]  \right)  $, which vanishes if and
only if $Y$ has finite homotopy type. Moreover, all elements of $\widetilde{K}%
_{0}\left(
\mathbb{Z}
\left[  \pi_{1}\left(  Y\right)  \right]  \right)  $ can be realized as
finiteness obstructions of a finitely dominated CW complex.
\end{theorem}

A development of Wall's obstruction is interesting and entirely
understandable, but outside the scope of these notes. The interested reader is
referred to Wall's original paper or the exposition in \cite{Fe}. For late
use, we note that $\widetilde{K}_{0}$ determines a functor from $\mathcal{G}%
$\emph{roups }to $\mathcal{A}$\emph{belian groups}; in particular, if
$\lambda:G\rightarrow H$ is a group homomorphism, then there is an induced
homomorphism $\lambda_{\ast}:\widetilde{K}_{0}\left(
\mathbb{Z}
\left[  G\right]  \right)  \rightarrow\widetilde{K}_{0}\left(
\mathbb{Z}
\left[  H\right]  \right)  $ between the corresponding projective class groups.

\begin{example}
Every compact ENR $A$ is easily seen to be finitely dominated. Indeed, if
$U\subseteq%
\mathbb{R}
^{n}$ is a neighborhood of $A$ and $r:U\rightarrow A$ a retraction, let
$K\subseteq U$ be a polyhedral neighborhood of $A$, $d:K\rightarrow A$ the
restriction, and $u$ the inclusion.

Although this is a nice example, it is made obsolete by a major result of West
(see Proposition \ref{Proposition: ANR facts}), showing that every compact ANR
has finite homotopy type.
\end{example}

\subsection{Inward tameness}

Modulo a slight change in terminology, we follow \cite{CS} by defining an ANR
$X$ to be \emph{inward tame} if, for each neighborhood of infinity $N$ there
exists a smaller neighborhood of infinity $N^{\prime}$ so that, up to
homotopy, the inclusion $N^{\prime}\overset{j}{\hookrightarrow}N$ factors
through a finite complex $K$. In other words, there exist maps $f:N^{\prime
}\rightarrow K$ and $g:K\rightarrow N$ such that $gf\simeq j$.

\begin{exercise}
\label{Exercise: inward tame is a proper homotopy invariant}Show that if
$X\overset{p}{\simeq}Y$ and $X$ is inward tame, then $Y$ is inward tame.
\end{exercise}

For the remainder of this section, our goals are as follows:\medskip

\noindent\textbf{a)} to obtain a more intrinsic and intuitive characterization
of inward tameness, and\smallskip

\noindent\textbf{b)} to clarify the (apparent) relationship between inward
tameness and finite dominations.\medskip

The following is our answer to Goal a).

\begin{lemma}
\label{Lemma: inward tame implies pulling inward}An ANR $X$ is inward tame if
and only if, for every closed neighborhood of infinity $N$ in $X$, there is a
homotopy $S:N\times\left[  0,1\right]  \rightarrow N$ with $S_{0}%
=\operatorname*{id}_{N}$ and $\overline{S_{1}\left(  N\right)  }$ compact (a
homotopy pulling $N$ into a compact subset).\ 

\begin{proof}
For the forward implication, let $N^{\prime}$ be a closed neighborhood of
infinity contained in $\operatorname*{int}N$ so that $N^{\prime}%
\hookrightarrow\operatorname*{int}N$ factors through a compact polyhedron $K$.
Then there is a homotopy $H:N^{\prime}\times\left[  0,1\right]  \rightarrow
\operatorname*{int}N$ with $H_{0}$ the inclusion and $\overline{H_{1}\left(
N^{\prime}\right)  }\subseteq g\left(  K\right)  $. Choose an open
neighborhood $U$ of $N^{\prime}$ with $\overline{U}\cap\operatorname*{Bd}%
_{X}N=\varnothing$, then let $A=\operatorname*{int}N-U$ and $J$ be the
identity homotopy on $A$. Since $\operatorname*{int}N$ is an ANR, Borsuk's
Homotopy Extension Property (see Prop. \ref{Proposition: ANR facts}) allows us
to extend $H\cup J$ to a homotopy $S:\operatorname*{int}N\times\left[
0,1\right]  \rightarrow\operatorname*{int}N$ with $S_{0}=\operatorname*{id}%
_{\operatorname*{int}N}$. This in turn may be extended via the identity over
$\operatorname*{Bd}_{X}N$ to obtain a homotopy $S$ that pulls $N$ into a
compact subset of itself.

We will return for the converse after addressing Goal b).
\end{proof}
\end{lemma}

Recall that an ANR $X$ is sharp at infinity if it contains arbitrarily small
closed ANR neighborhoods of infinity.

\begin{lemma}
\label{Lemma: tame = nbds of infinity are f.d.}A space $X$ that is sharp at
infinity is inward tame if and only if each of its sharp neighborhoods of
infinity is finitely dominated.

\begin{proof}
Assume $X$ is sharp at infinity and inward tame. By Lemma
\ref{Lemma: inward tame implies pulling inward} each closed neighborhood of
infinity can be pulled into a compact subset, so by Proposition
\ref{Prop: finitely dominated = pulling homotopy}, those which are ANRs are
finitely dominated. The converse is immediate by the definitions.
\end{proof}
\end{lemma}

\begin{proof}
[Proof (completion of Lemma \ref{Lemma: inward tame implies pulling inward}%
)]Suppose that, for each closed neighborhood of infinity $N$ in $X$, there is
a homotopy pulling $N$ into a compact subset. Then the same is true for
$X\times\mathcal{Q}$. But $X\times\mathcal{Q}$ is sharp since it is a Hilbert
cube manifold, so by Proposition
\ref{Prop: finitely dominated = pulling homotopy}, each ANR neighborhood of
infinity in $X\times\mathcal{Q}$ is finitely dominated. By Lemma
\ref{Lemma: tame = nbds of infinity are f.d.} and Exercise
\ref{Exercise: inward tame is a proper homotopy invariant}, $X$ is inward tame.
\end{proof}

We tidy up by combining the above Lemmas into a single Proposition, and adding
some mild extensions. For convenience we restrict attention to spaces that are
sharp at infinity.

\begin{proposition}
For a space $X$ that is sharp at infinity, the following are equivalent.

\begin{enumerate}
\item $X$ is inward tame,

\item for every closed neighborhood of infinity $N$, there is a homotopy
$H:N\times\left[  0,1\right]  \rightarrow N$ with $H_{0}=\operatorname*{id}%
_{N}$ and $\overline{H_{1}\left(  N\right)  }$ compact,

\item there exist arbitrarily small closed neighborhood of infinity $N$, for
which there is a homotopy $H:N\times\left[  0,1\right]  \rightarrow N$ with
$H_{0}=\operatorname*{id}_{N}$ and $\overline{H_{1}\left(  N\right)  }$ compact,

\item every sharp neighborhood of infinity is finitely dominated,

\item there exist arbitrarily small sharp neighborhoods of infinity that are
finitely dominated.
\end{enumerate}
\end{proposition}

\begin{proof}
The equivalence of (b) and (c) is by a homotopy extension argument like that
found in Lemma \ref{Lemma: inward tame implies pulling inward}. The
equivalence of (d) and (e) is similar, but easier.
\end{proof}

\begin{remark}
\emph{The \textquotedblleft inward\textquotedblright\ in inward tame is
motivated by conditions (b) and (c) where the homotopies are viewed as pulling
the end of }$X$\emph{ inward} \emph{toward the center of }$X$\emph{. Based on
the definition and conditions (d) and (e), one may also think of inward
tameness as \textquotedblleft finitely dominated at infinity\textquotedblright%
. We call }$X$ absolutely inward tame \emph{if it contains arbitrarily small
closed ANR neighborhoods of infinity with finite homotopy type.}
\end{remark}

\begin{example}
The infinite ladder $L_{[0,\infty)}$ is not inward tame, since its ANR
neighborhoods of infinity have infinitely generated fundamental groups.
Similarly, the infinite genus 1-ended orientable surface in Figure
\ref{Figure: Infinite genus surface}
\begin{figure}[ptb]%
\centering
\includegraphics[
height=0.6227in,
width=2.951in
]%
{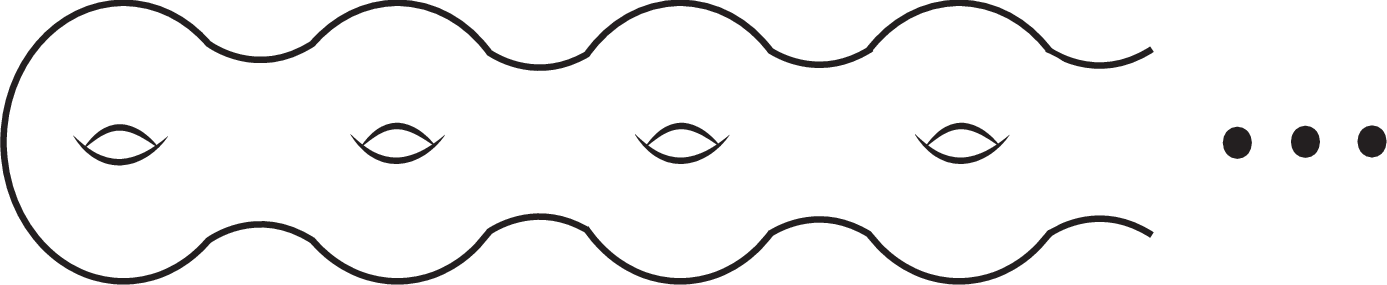}%
\caption{1-ended infinite genus surface}%
\label{Figure: Infinite genus surface}%
\end{figure}
is not inward tame.
\end{example}

\begin{example}
Although the Whitehead manifold $\mathcal{W}^{3}$ itself has finite homotopy
type, it is not inward tame, since the neighborhoods of infinity $U_{i}$
discussed in Example
\ref{Example: Fundamental group at infinity for the Whitehead manifold} do not
have finitely generated fundamental groups (proof would require some work).
The Davis manifolds, on the other hand, are absolutely inward tame. More on
these observations in \S \ref{Subsection: Generalizing Siebenman}.
\end{example}

\begin{exercise}
Justify the above assertion about the Davis manifolds.
\end{exercise}

\begin{example}
Every proper CAT(0) space $X$ is absolutely inward tame. For inward tameness,
let $N_{p,r}$ be the complement of an open ball $B\left(  p;r\right)  $ and
use geodesics to strong deformation retract $N_{p,r}$ onto the metric sphere
$S\!\left(  p;r\right)  $. If $S\!\left(  p;r\right)  $ is an ANR, then it
(and thus $N_{p,r}$) have finite homotopy type by Proposition
\ref{Proposition: ANR facts}. Since this is not known to be the case, more
work is required. For each sharp neighborhood of infinity $N$ (recall Example
\ref{Example: Proper CAT(0) spaces are sharp at infinity}), choose $r$ so that
$\overline{X-N}\subseteq B\left(  p;r\right)  $ and let $A=\overline
{N-N_{p,r}}$. Then $N$ strong deformation retracts onto $A$, which is a
compact ANR.\medskip
\end{example}

Before closing this section, we caution the reader that differing notions of
\textquotedblleft tameness\textquotedblright\ are scattered throughout the
literature. \cite{Si} called a 1-ended open manifold \emph{tame} if it
satisfies our definition for inward tame and also has \textquotedblleft
stable\textquotedblright\ fundamental group at infinity (a concept to be
discussed shortly). In \cite{CS}, the definition of tame was reformulated to
match our current-day definition of inward tame. Later still, \cite{Quin2} and
\cite{HR} put forth another version of \textquotedblleft
tame\textquotedblright\ in which homotopies push neighborhoods of infinity
toward the end of the space---sometimes referring to that version as
\emph{forward tame} and the \cite{CS} version as \emph{reverse tame}. In an
effort to avoid confusion, this author introduced the term \emph{inward tame,}
while referring to the Quinn-Hughes-Ranicki version as \emph{outward tame.}

Within the realm of 3-manifold topology, a \emph{tame end }is often defined to
be one for which there exists a product neighborhood of infinity
$N\approx\partial N\times\lbrack0,\infty)$. Remarkably, by \cite{Tu} combined
with the 3-dimensional Poincar\'{e} conjecture---in the special case of
3-manifolds---this property, inward tameness, and outward tameness are all equivalent.

Despite its mildly confusing history, the concept of inward tameness (and its
variants) is fundamental to the study of noncompact spaces. Throughout the
reminder of these notes, its importance will become more and more clear. In
\S \ref{Subsection: Shape of the end of a space}, we will give meaning to the
slogan: \textquotedblleft an inward tame space is one that acts like a
compactum at infinity\textquotedblright.

\section{Algebraic invariants of the ends of a space: the precise
definitions\label{Section: Algebraic invariants-precise definitions}}

In \S \ref{Subsection: Examples of fundamental groups at infinity} we
introduced the fundamental group at infinity rather informally. In this
section we provide the details necessary to place that invariant on firm
mathematical ground. In the process we begin to uncover subtleties that make
this invariant even more interesting than one might initially expect.

As we progress, it will become apparent that the fundamental group at infinity
(more precisely \textquotedblleft$\operatorname*{pro}$-$\pi_{1}$%
\textquotedblright) is just one of many \textquotedblleft end
invariants\textquotedblright. By the end of the section, we will have
introduced others, including $\operatorname*{pro}$-$\pi_{k}$ and
$\operatorname*{pro}$-$H_{k}$ for all $k\geq0.$

\subsection{An equivalence relation on the set of inverse
sequences\label{Subsection: defining pro-isomorphism}}

The \emph{inverse limit} of an inverse sequence%
\[
G_{0}\overset{\mu_{1}}{\longleftarrow}G_{1}\overset{\mu_{2}}{\longleftarrow
}G_{2}\overset{\mu_{3}}{\longleftarrow}G_{3}\overset{\mu_{4}}{\longleftarrow
}\cdots
\]
of groups is defined by%
\[
\underleftarrow{\lim}\left\{  G_{i},\mu_{i}\right\}  =\left\{  \left(
g_{0},g_{1},g_{2},\cdots\right)  \mid\mu_{i}\left(  g_{i}\right)
=g_{i-1}\text{ for all }i\geq1\right\}  .
\]
Although useful at times, passing to an inverse limit often results in a loss
of information. Instead, one usually opts to keep the entire sequence---or,
more accurately, the essential elements of that sequence. To get a feeling for
what is meant by \textquotedblleft essential elements\textquotedblright, let
us look at some things that can go wrong.

In Example \ref{Example: Fundamental group at infinity for R^n}, we obtained
the following representation of the fundamental group of infinity for $%
\mathbb{R}
^{3}$.%
\begin{equation}
1\leftarrow1\longleftarrow1\longleftarrow\cdots.\tag{3.5}%
\label{Inverse sequence: Constant 1}%
\end{equation}
That was done by exhausting $%
\mathbb{R}
^{3}$ with a sequence $\left\{  i\mathbb{B}^{3}\right\}  $ of closed $i$-balls
and letting $U_{i}=\overline{%
\mathbb{R}
^{3}-i\mathbb{B}^{3}}\approx\mathbb{S}^{2}\times\lbrack i,\infty)$. If
instead, $%
\mathbb{R}
^{3}$ is exhausted with a sequence $\left\{  T_{i}\right\}  $ of solid tori
where each $T_{j}$ lies in $T_{j+1}$ as shown in Figure
\ref{Figure: Null torus}%
\begin{figure}[ptb]%
\centering
\includegraphics[
height=1.4947in,
width=2.5001in
]%
{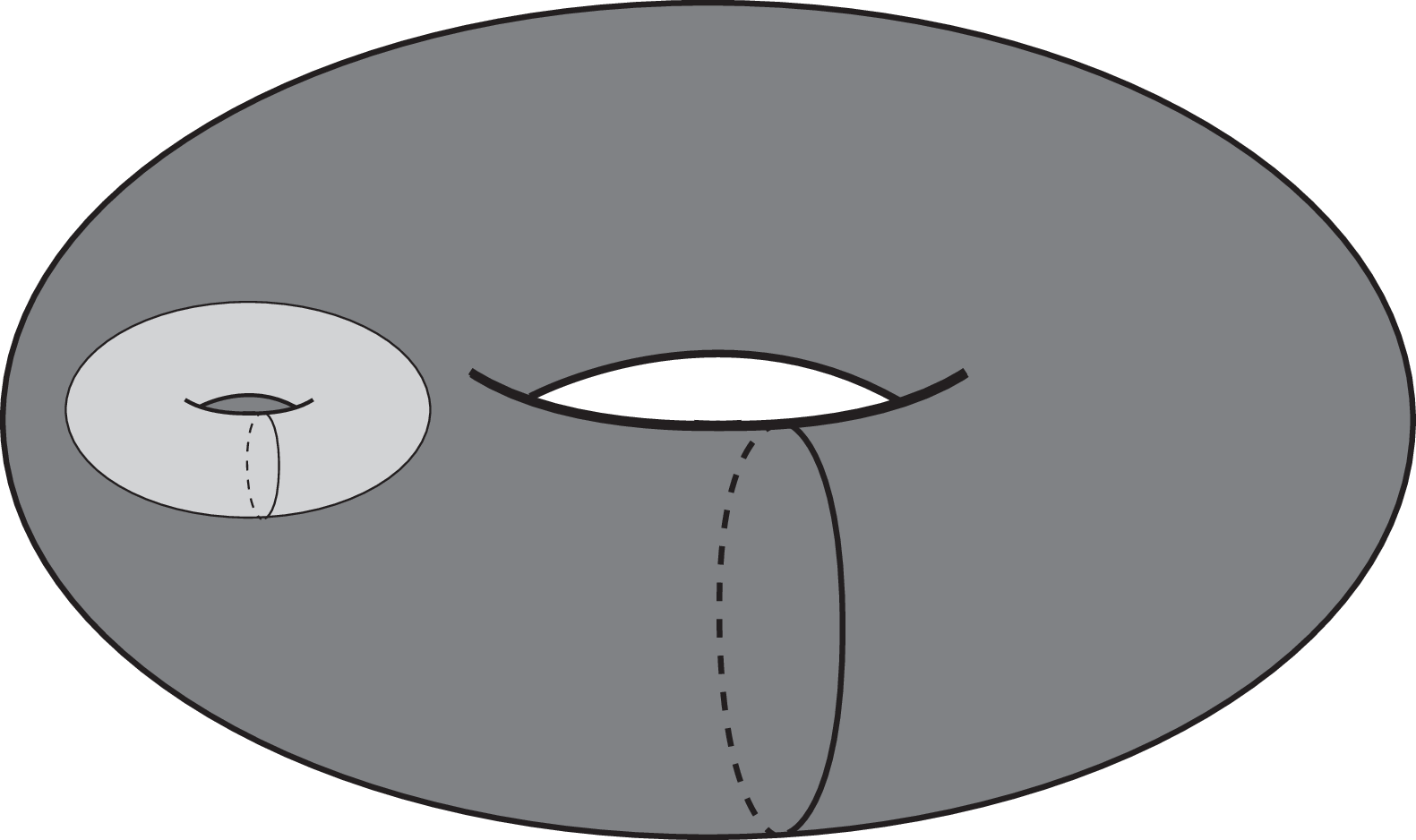}%
\caption{Exhausting $\mathbb{R} ^{3}$ with solid tori.}%
\label{Figure: Null torus}%
\end{figure}
and $V_{i}=\overline{%
\mathbb{R}
^{3}-T_{i}}$, the resulting representation of the fundamental group of
infinity is%
\[%
\mathbb{Z}
\overset{0}{\leftarrow}%
\mathbb{Z}
\overset{0}{\leftarrow}%
\mathbb{Z}
\overset{0}{\leftarrow}\cdots.
\]
By choosing more complicated exhausting sequences (e.g., exhaustions by higher
genus knotted handlebodies), representations with even more complicated groups
can be obtained. It can also be arranged that the bonding homomorphisms are
not always trivial. Yet each of these sequences purports to describe the same
thing. Although it seems clear that (\ref{Inverse sequence: Constant 1}) is
the preferred representative for the end of $%
\mathbb{R}
^{3}$, in the case of an arbitrary 1-ended space, there may be no obvious
\textquotedblleft best choice\textquotedblright. The problem is resolved by
placing an equivalence relation on the set of all inverse sequences of groups.
Within an equivalence class, certain representatives may be preferable to
others, but each contains the essential information.

For an inverse sequence $\left\{  G_{i},\phi_{i}\right\}  $, there is an
obvious meaning for \emph{subsequence}%
\[
G_{k_{0}}\overset{\phi_{k_{0},k_{1}}}{\longleftarrow}G_{k_{1}}\overset{\phi
_{k_{1},k_{2}}}{\longleftarrow}G_{k_{2}}\overset{\phi_{k_{2},k_{3}%
}}{\longleftarrow}\cdots
\]
where the bonding homomorphisms $\phi_{k_{i},k_{i+1}}$ are compositions of the
$\phi_{i}$. Declare inverse sequences $\left\{  G_{i},\phi_{i}\right\}  $ and
$\left\{  H_{i},\psi_{i}\right\}  $ to be \emph{pro-isomorphic} if they
contain subsequences that fit into a commuting \textquotedblleft ladder
diagram\textquotedblright%
\[
\begin{diagram} G_{i_{0}} & & \lTo^{\lambda_{i_{0},i_{1}}} & & G_{i_{1}} & & \lTo^{\lambda_{i_{1},i_{2}}} & & G_{i_{2}} & & \lTo^{\lambda_{i_{2},i_{3}}}& & G_{i_{3}}& \cdots\\ & \luTo & & \ldTo & & \luTo & & \ldTo & & \luTo & & \ldTo &\\ & & H_{j_{0}} & & \lTo^{\mu_{j_{0},j_{1}}} & & H_{j_{1}} & & \lTo^{\mu_{j_{1},j_{2}}}& & H_{j_{2}} & & \lTo^{\mu_{j_{2},j_{3}}} & & \cdots \end{diagram}
\]
More broadly, define \emph{pro-isomorphism} to be the equivalence relation on
the collection of all inverse sequences of groups generated by that
rule\footnote{The prefix \textquotedblleft pro\textquotedblright\ is derived
from \textquotedblleft projective\textquotedblright. Some authors refer to
inverse sequences and inverse limits as projective sequences and projective
limits, respectively.}.

It is immediate that an inverse sequence is pro-isomorphic to each of its
subsequences; but sequences can appear very different and still be pro-isomorphic.

\stepcounter{theorem}

\begin{exercise}
Convince yourself that the various inverse sequences mentioned above for
describing the fundamental group at infinity of $%
\mathbb{R}
^{3}$ are pro-isomorphic.
\end{exercise}

\begin{exercise}
\label{Exercise: Inverse limit well-defined on pro-groups}Show that a pair of
pro-isomorphic inverse sequences of groups have isomorphic inverse limits.
\emph{Hint: }Begin by observing a canonical isomorphism between the inverse
limit of a sequence and that of any of its subsequences.
\end{exercise}

The next exercise provides a counterexample to the converse of Exercise
\ref{Exercise: Inverse limit well-defined on pro-groups}. It justifies our
earlier assertion that passing to an inverse limit often results in loss of information.

\begin{exercise}
\label{Exercise: nontrivial inverse sequence with trivial inverse limit}Show
that the inverse sequence $%
\mathbb{Z}
\overset{\times2}{\longleftarrow}%
\mathbb{Z}
\overset{\times2}{\longleftarrow}%
\mathbb{Z}
\overset{\times2}{\longleftarrow}\cdots$ is not pro-isomorphic to the trivial
inverse sequence $1\leftarrow1\longleftarrow1\longleftarrow\cdots$, but both
inverse limits are trivial.
\end{exercise}

\begin{exercise}
A more slick (if less intuitive) way to define pro-isomorphism is to declare
it to be the equivalence relation generated by making sequences equivalent to
their subsequences. Show that the two approaches are equivalent.
\end{exercise}

\begin{remark}
\emph{With a little more work, we could define }morphisms\emph{ between
inverse sequences of groups and arrive at a category }pro-$\mathcal{G}$roups,
\emph{where the objects are inverse sequences of groups, in which two objects
are pro-isomorphic if and only if they are isomorphic in that category.}%
\footnote{We are not being entirely forthright here. In the literature,
\emph{pro}-\emph{Groups} usually refers to a larger category consisting of
\textquotedblleft inverse systems\textquotedblright\ of groups indexed by
arbitrary partially ordered sets. We have described a subcategory,
\emph{Tow}-\emph{Groups}, made up of those objects indexed by the natural
numbers---also known as \textquotedblleft towers\textquotedblright.}

\emph{Similarly, for any category }$\mathcal{C}$\emph{ one can build a
category }$\operatorname*{pro}$-$\,\mathcal{C}$\emph{ in which the objects are
inverse sequences of objects and morphisms from }$\mathcal{C}$\emph{ and for
which the resulting relationship of pro-isomorphism is similar to the one
defined above. All of this is interesting and useful, but more than we need
here. For a comprehensive treatment of this topic, see} \cite{Ge2}.
\end{remark}

\subsection{Topological definitions and justification of the pro-isomorphism
relation}

A quick look at the topological setting that leads to multiple inverse
sequences representing the same fundamental group at infinity provides
convincing justification for the definition of pro-isomorphic.

Let $U_{0}\hookleftarrow U_{1}\hookleftarrow U_{2}\hookleftarrow\cdots$ and
$V_{0}\hookleftarrow V_{1}\hookleftarrow V_{2}\hookleftarrow\cdots$ be two
cofinal sequences of connected neighborhoods of infinity for a 1-ended space
$X$. By going out sufficiently far in the second sequence, one arrives at a
$V_{k_{0}}$ contained in $U_{0}$. Similarly, going out sufficiently far in the
initial sequence produces a $U_{j_{1}}\subseteq V_{k_{0}}$. (For convenience,
let $j_{0}=0$.) Alternating back and forth produces a ladder diagram of
inclusions%
\begin{equation}
\begin{diagram} V_{j_{0}} & & \lInto & & V_{j_{1}} & & \lInto & & V_{j_{2}} & & \lInto & & \cdots \\ & \luInto & & \ldInto & & \luInto & & \ldInto & & \luInto \\ & & U_{k_{0}} & & \lInto & & U_{k_{1}} & & \lInto & & U_{k_{2}} & & \cdots \end{diagram}.\tag{3.6}%
\label{diagram: pro-equivalence of neighborhood systems}%
\end{equation}
Applying the fundamental group functor to that diagram (ignoring base points
for the moment) results in a diagram
\begin{equation}
\begin{diagram} \pi_{1}(V_{j_{0}}) & & \lTo & & \pi_{1}(V_{j_{1}}) & & \lTo & & \pi_{1}(V_{j_{2}}) & & \lTo & & \cdots \\ & \luTo & & \ldTo & & \luTo & & \ldTo & & \luTo \\ & & \pi_{1}(U_{k_{0}}) & & \lTo & & \pi_{1}(U_{k_{1}}) & & \lTo & & \pi_{1}(U_{k_{2}}) & & \cdots \end{diagram}\tag{3.7}%
\label{diagram: pro-pi1 pro-equivalence}%
\end{equation}
showing that
\[
\pi_{1}\left(  U_{0}\right)  \overset{\lambda_{1}}{\longleftarrow}\pi
_{1}\left(  U_{1}\right)  \overset{\lambda_{2}}{\longleftarrow}\pi_{1}\left(
U_{2}\right)  \overset{\lambda_{3}}{\longleftarrow}\cdots
\]
and
\[
\pi_{1}\left(  V_{0}\right)  \overset{\mu_{1}}{\longleftarrow}\pi_{1}\left(
V_{1}\right)  \overset{\mu_{2}}{\longleftarrow}\pi_{1}\left(  V_{2}\right)
\overset{\mu_{3}}{\longleftarrow}\cdots
\]
are pro-isomorphic.

A close look at base points and base rays is still to come, but recognizing
their necessity, we make the following precise definition. For a pair $\left(
X,r\right)  $ where $r$ is a proper ray in $X$, let $\operatorname*{pro}$%
-$\pi_{1}\left(  \varepsilon(X),r\right)  $ denote the pro-isomorphism class
of inverse sequences of groups which contains representatives of the form
(\ref{line: initial fundamental group at infinity}), where $\left\{
U_{i}\right\}  _{i=0}^{\infty}$ is a cofinal sequence of neighborhoods of
infinity, and $r$ has been modified (in the manner described in Exercise
\ref{Exercise: Fixing a proper ray}) so that $r\left(  [i,\infty)\right)
\subseteq U_{i}$ for each $i\geq0$. From now on, when we refer to the
\emph{fundamental group at infinity (based at }$r$\emph{)} of a space $X$, we
mean $\operatorname*{pro}$-$\pi_{1}\left(  \varepsilon(X),r\right)  $.

With the help of Exercise
\ref{Exercise: Inverse limit well-defined on pro-groups}, we also define the
\emph{\v{C}ech fundamental group of the end of }$X$ \emph{(based at }%
$r$\emph{)}, to be the inverse limit of $\operatorname*{pro}$-$\pi_{1}\left(
\varepsilon(X),r\right)  $. It is denoted by $\check{\pi}_{1}\left(
\varepsilon(X),r\right)  $.

\stepcounter{theorem}

\stepcounter{theorem}

\begin{exercise}
Fill in the details related to base points and base rays needed for the
existence of diagram (\ref{diagram: pro-pi1 pro-equivalence}).
\end{exercise}

\begin{remark}
\emph{Now that }$\operatorname*{pro}$\emph{-}$\pi_{1}\left(  \varepsilon
(X),r\right)  $ \emph{is well-defined and (hopefully) well-understood for
1-ended }$X$\emph{, it is time to point out that everything done thus far
works for multi-ended }$X$\emph{. In those situations, the role of }$r$\emph{
is more pronounced. In the process of selecting base points for a sequence of
neighborhoods of infinity }$\left\{  U_{i}\right\}  $\emph{, }$r$\emph{
determines the component of each }$U_{i}$\emph{ that contributes to
}$\operatorname*{pro}$\emph{-}$\pi_{1}\left(  \varepsilon(X),r\right)
$\emph{. So, if }$r$\emph{ and }$s$\emph{ point to different ends of }%
$X$\emph{, }$\operatorname*{pro}$\emph{-}$\pi_{1}\left(  \varepsilon
(X),r\right)  $\emph{ and }$\operatorname*{pro}$\emph{-}$\pi_{1}\left(
\varepsilon(X),s\right)  $\emph{ reflect information about entirely different
portions of }$X$\emph{. This observation is just the beginning; a thorough
examination of the role of base rays is begun in
\S \ref{Section: On the role of the base ray}.}
\end{remark}

\subsection{Other algebraic invariants of the end of a space}

By now it has likely occurred to the reader that $\pi_{1}$ is not the only
functor that can be applied to an inverse sequence of neighborhoods of
infinity. For any $k\geq1$ and proper ray $r$, define $\operatorname*{pro}%
$-$\pi_{k}\left(  \varepsilon(X),r\right)  $ in the analogous manner. By
taking inverse limits we get the \emph{\v{C}ech homotopy groups} $\check{\pi
}_{k}\left(  \varepsilon\left(  X\right)  ,r\right)  $ of the end of $X$
determined by $r$. Similarly, we may define $\operatorname*{pro}$-$\pi
_{0}\left(  \varepsilon(X),r\right)  $ and $\check{\pi}_{0}\left(
\varepsilon\left(  X\right)  ,r\right)  $; the latter is just a set (more
precisely a \emph{pointed} set, i.e., a set with a distinguished base point),
and the former an equivalence class of inverse sequences of (pointed) sets.

By applying the homology functor we obtain $\operatorname*{pro}$-$H_{k}\left(
\varepsilon(X);R\right)  $ and $\check{H}_{k}\left(  \varepsilon\left(
X\right)  ;R\right)  $ for each non-negative integer $k$ and arbitrary
coefficient ring $R$, the latter being called the \emph{\v{C}ech homology of
the end of }$X$. In this context, no base ray is needed!

If instead we apply the cohomology functor, a significant change occurs. The
contravariant nature of $H^{k}$ produces \emph{direct sequences}%
\[
H^{k}\left(  U_{0};R\right)  \overset{\lambda_{1}}{\longrightarrow}%
H^{k}\left(  U_{1};R\right)  \overset{\lambda_{2}}{\longrightarrow}%
H^{k}\left(  U_{2};R\right)  \overset{\lambda_{3}}{\longrightarrow}\cdots
\]
of cohomology groups. An algebraic treatment of such sequences, paralleling
\S \ref{Subsection: defining pro-isomorphism}, and a standard definition of
\emph{direct limit}, allow us to define $\operatorname*{ind}$-$H^{\ast}\left(
\varepsilon(X);R\right)  $ and $\check{H}^{\ast}\left(  \varepsilon\left(
X\right)  ;R\right)  $.

\begin{exercise}
Show that for ANRs there is a one-to-one correspondence between $\mathcal{E}%
\emph{nds}\left(  X\right)  $ and $\check{\pi}_{0}\left(  \varepsilon\left(
X\right)  ,r\right)  $.
\end{exercise}

\subsection{End invariants and the proper homotopy category}

In Remark \ref{Remark: importance of properness}, we commented on the
importance of proper maps and proper homotopy equivalences in the study of
noncompact spaces. We are now ready to back up that assertion. The following
Proposition could be made even stronger with a discussion of morphisms in the
category of pro-$\mathcal{G}$\emph{roups}, but for our purposes, it will suffice.

\begin{proposition}
\label{Prop: proper homotopy invariance of end invariants}Let $f:X\rightarrow
Y$ be a proper homotopy equivalence and $r$ a proper ray in $X$. Then

\begin{enumerate}
\item $\operatorname*{pro}$-$H_{k}\left(  \varepsilon(X);R\right)  $ is
pro-isomorphic to $\operatorname*{pro}$-$H_{k}\left(  \varepsilon(Y);R\right)
$ for all $k$ and every coefficient ring $R$,

\item $\operatorname*{pro}$-$\pi_{0}\left(  \varepsilon(X,r)\right)  $ is
pro-isomorphic to $\operatorname*{pro}$-$\pi_{0}\left(  \varepsilon(Y,f\circ
r)\right)  $ as inverse sequences of pointed sets, and

\item $\operatorname*{pro}$-$\pi_{k}\left(  \varepsilon(X),r\right)  $ is
pro-isomorphic to $\operatorname*{pro}$-$\pi_{k}\left(  \varepsilon(Y),f\circ
r\right)  $ for all $k\geq1.$
\end{enumerate}

\begin{corollary}
A proper homotopy equivalence $f:X\rightarrow Y$ induces isomorphisms between
$\check{H}_{k}\left(  \varepsilon(X);R\right)  $ and $\check{H}_{k}\left(
\varepsilon(Y);R\right)  $ for all $k$ and every coefficient ring $R$. It
induces a bijection between $\check{\pi}_{0}\left(  \varepsilon(X),r\right)  $
and $\check{\pi}_{0}\left(  \varepsilon(Y),f\circ r\right)  $ and isomorphisms
between $\check{\pi}_{k}\left(  \varepsilon(X),r\right)  $ and $\check{\pi
}_{k}\left(  \varepsilon(Y),f\circ r\right)  $ for all $k\geq1$.
\end{corollary}
\end{proposition}

\begin{proof}
[Sketch of the proof of Proposition
\ref{Prop: proper homotopy invariance of end invariants}]Let $g:Y\rightarrow
X$ be a proper inverse for $f$ and let $H$ and $K$ be proper homotopies
between $g\circ f$ and $\operatorname*{id}_{X}$ and $f\circ g$ and
$\operatorname*{id}_{Y}$, respectively. By using the properness of $H$ and $K$
and a back-and-forth strategy similar to the one employed in obtaining diagram
(\ref{diagram: pro-equivalence of neighborhood systems}), we obtain systems of
neighborhoods of infinity $\left\{  U_{i}\right\}  $ in $X$ and $\left\{
V_{i}\right\}  $ in $Y$ that fit into a ladder diagram%
\begin{equation}
\begin{diagram} V_{j_{0}} & & \lInto & & V_{j_{1}} & & \lInto & & V_{j_{2}} & & \lInto & & \cdots \\ & \luTo & & \ldTo & & \luTo & & \ldTo & & \luTo \\ & & U_{k_{0}} & & \lInto & & U_{k_{1}} & & \lInto & & U_{k_{2}} & & \cdots \end{diagram}.\tag{3.8}%
\label{Diagram: ladder diagran for a p.h.e.}%
\end{equation}
Unlike the earlier case, the up and down arrows are not inclusions, but rather
restrictions of $f$ and $g$. Furthermore, the diagram does not commute on the
nose; instead, it commutes up to homotopy. But that is enough to obtain a
commuting ladder diagram of homology groups, thus verifying (a). The same is
true for (b), but on the level of sets. Assertion (c) is similar, but a little
additional care must be taken to account for the base rays.
\end{proof}

\subsection{Inverse mapping telescopes and a topological realization
theorem\label{Subsection: Inverse mapping telescopes}}

It is natural to ask which inverse sequences (more precisely, pro-isomorphism
classes) can occur as $\operatorname*{pro}$-$\pi_{1}\left(  \varepsilon
(X),r\right)  $ for a space $X$. Here we show that, even if restricted to very
nice spaces, the answer is \textquotedblleft nearly all of
them\textquotedblright. Later we will see that, in certain important contexts
the answer becomes much different. But for now we create a simple machine for
producing wide range of examples.

Let
\begin{equation}
(K_{0},p_{0})\overset{f_{1}}{\longleftarrow}(K_{1},p_{1})\overset{f_{2}%
}{\longleftarrow}(K_{2},p_{2})\overset{f_{3}}{\longleftarrow}\cdots
\tag{3.9}\label{Inverse sequence: pointed complexes}%
\end{equation}
be an inverse sequence of pointed finite CW complexes and cellular maps. For
each $i\geq1$, let $M_{i}$ be a copy of the \emph{mapping cylinder} of $f_{i}%
$; more specifically%
\[
M_{i}=(K_{i}\times\lbrack i-1,i])\sqcup(K_{i-1}\times\left\{  i-1\right\}
)/\sim_{i}%
\]
where $\sim_{i}$ is the equivalence relation generated by the rule: $\left(
k,i-1\right)  \sim_{i}(f_{i}\left(  k\right)  ,i-1)$ for each $k\in K_{i}$.
Then $M_{i}$ contains a canonical copy $K_{i-1}\times\left\{  i-1\right\}  $
of $K_{i-1}$ and a canonical copy $K_{i}\times\left\{  i\right\}  $ of $K_{i}%
$; and $M_{i-1}\cap M_{i}=K_{i-1}\times\left\{  i-1\right\}  $. The infinite
union $\operatorname*{Tel}\left(  \left\{  K_{i},f_{i}\right\}  \right)
=\bigcup{}_{i=1}^{\infty}M_{i}$, with the obvious topology is called the
\emph{mapping telescope }of (\ref{Inverse sequence: pointed complexes}). See
Figure \ref{Figure: Mapping telescope}.%
\begin{figure}[ptb]%
\centering
\includegraphics[
height=1.1377in,
width=2.2549in
]%
{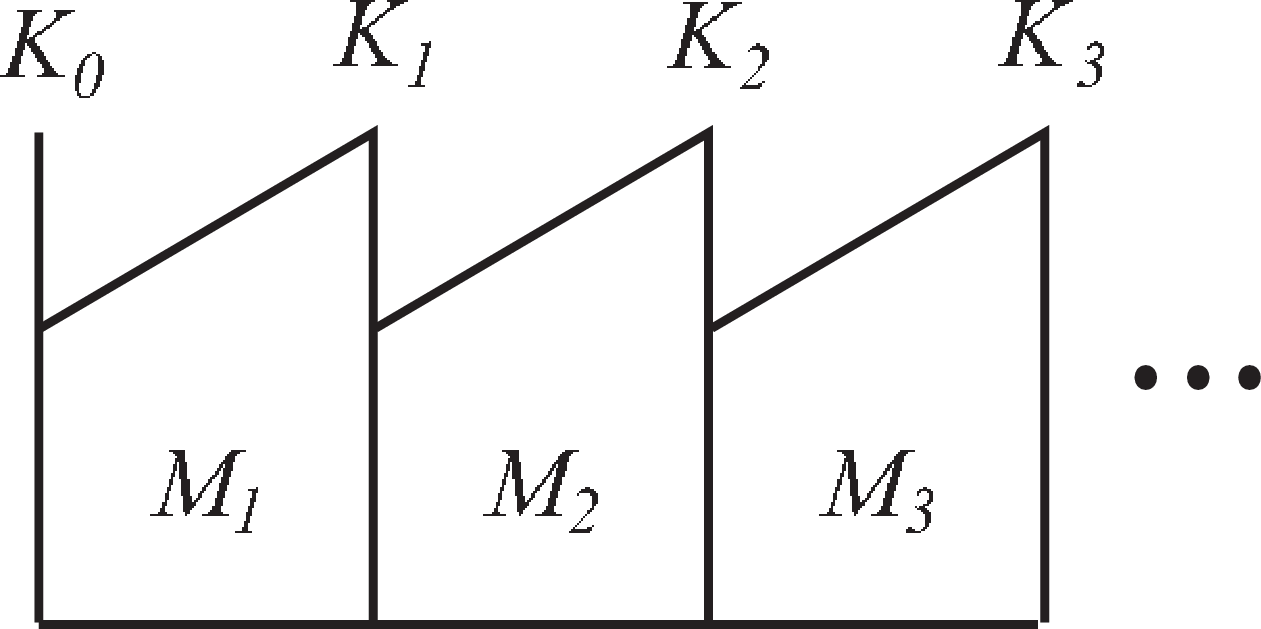}%
\caption{The mapping telescope $\operatorname*{Tel}\left(  \left\{
K_{i},f_{i}\right\}  \right)  $}%
\label{Figure: Mapping telescope}%
\end{figure}

For each $x\in K_{i}$, the (embedded) copy of the interval $\left\{
x\right\}  \times\lbrack i-1,i]$ in $M_{i}$ is called a \emph{mapping cylinder
line}. The following observations are straightforward.

\begin{itemize}
\item $\operatorname*{Tel}\left(  \left\{  K_{i},f_{i}\right\}  \right)  $ may
be viewed as the union of infinite and dead end \textquotedblleft telescope
rays\textquotedblright, each of which begins in $K_{0}\times\left\{
0\right\}  $ and intersects a given $M_{i}$ in a mapping cylinder line or not
at all. The dead end rays and empty intersections occur only when a point
$k\in K_{j}$ is not in the image of $f_{j+1}$; whereas, the infinite telescope
rays are proper and in one-to-one correspondence with $\underleftarrow{\lim
}\left\{  K_{i},f_{i}\right\}  $,

\item by choosing a canonical set of strong deformation retractions of the
above rays to their initial points, one obtains a strong deformation
retraction of $\operatorname*{Tel}\left(  \left\{  K_{i},f_{i}\right\}
\right)  $ to $K_{0}\times\left\{  0\right\}  $.

\item letting $U_{k}=\bigcup{}_{i=k+1}^{\infty}M_{i}$ provides a cofinal
sequence of neighborhoods of infinity. By a small variation on the previous
observation each $K_{i}\times\left\{  i\right\}  \hookrightarrow U_{i}$ is a
homotopy equivalence. (So $\operatorname*{Tel}\left(  \left\{  K_{i}%
,f_{i}\right\}  \right)  $ is absolutely inward tame.)

\item letting $r$ be the proper ray consisting of the cylinder lines
connecting each $p_{i}$ to $p_{i-1}$, we obtain a representation of
$\operatorname*{pro}$-$\pi_{1}\left(  \varepsilon(X),r\right)  $ which is
pro-isomorphic to the sequence
\[
\pi_{1}\left(  K_{0},p_{0}\right)  \overset{f_{1\#}}{\longleftarrow}\pi
_{1}\left(  K_{1},p_{1}\right)  \overset{f_{2\#}}{\longleftarrow}\pi
_{1}\left(  K_{2},p_{2}\right)  \overset{f_{3\#}}{\longleftarrow}\cdots
\]

\item in the same manner, representations of $\operatorname*{pro}$-$\pi
_{k}\left(  \varepsilon(X),r\right)  $ and $\operatorname*{pro}$-$H_{k}\left(
\varepsilon(X),%
\mathbb{Z}
\right)  $ can be obtained by applying the appropriate functor to sequence
(\ref{Inverse sequence: pointed complexes}).
\end{itemize}

\stepcounter{theorem}

\stepcounter{theorem}

\begin{proposition}
\label{Prop: mapping telescope realization fro pro-p1}For every inverse
sequence $G_{0}\overset{\mu_{1}}{\longleftarrow}G_{1}\overset{\mu
_{2}}{\longleftarrow}G_{2}\overset{\mu_{3}}{\longleftarrow}\cdots$ of finitely
presented groups, there exists a 1-ended, absolutely inward tame, locally
finite CW complex $X$ and a proper ray $r$ such that $\operatorname*{pro}%
$-$\pi_{1}\left(  \varepsilon(X),r\right)  $ is represented by that sequence.
If desired, $X$ can be chosen to be contractible.

\begin{proof}
For each $i$, let $K_{i}$ be a presentation $2$-complex for $G_{i}$ and let
$f_{i}:K_{i}\rightarrow K_{i-1}$ be a cellular map that induces $\mu_{i}$.
Then let $X=\operatorname*{Tel}\left(  \left\{  K_{i},f_{i}\right\}  \right)
$.

In order to make $X$ contractible, one simply adds a trivial space
$K_{-1}=\left\{  p_{-1}\right\}  $ to the left end of the sequence of complexes.
\end{proof}
\end{proposition}

\begin{example}
\label{Example: realizing Z<-x2-Z<-x2-Z}An easy application of Proposition
\ref{Prop: mapping telescope realization fro pro-p1} produces a
$\operatorname*{pro}$-$\pi_{1}\left(  \varepsilon(X),r\right)  $ equal to the
inverse sequence $%
\mathbb{Z}
\overset{\times2}{\longleftarrow}%
\mathbb{Z}
\overset{\times2}{\longleftarrow}%
\mathbb{Z}
\overset{\times2}{\longleftarrow}\cdots$ discussed in Exercise
(\ref{Exercise: nontrivial inverse sequence with trivial inverse limit}). For
each $i$, let $\mathbb{S}_{i}^{1}$ be a copy of the unit circle and
$f_{i}:\mathbb{S}_{i}^{1}\rightarrow\mathbb{S}_{i-1}^{1}$ the standard degree
2 map. Then $X=\operatorname*{Tel}\left(  \left\{  \mathbb{S}_{i}^{1}%
,f_{i}\right\}  \right)  $ is 1-ended and has the desired fundamental group at infinity.
\end{example}

\begin{proposition}
For every inverse sequence $G_{0}\overset{\mu_{1}}{\longleftarrow}%
G_{1}\overset{\mu_{2}}{\longleftarrow}G_{2}\overset{\mu_{3}}{\longleftarrow
}\cdots$ of finitely presented groups and $n\geq6$, there exists a 1-ended
open $n$-manifold $M^{n}$ such that $\operatorname*{pro}$-$\pi_{1}\left(
M^{n},r\right)  $ is represented by that sequence. If a (noncompact) boundary
is permitted, and $n\geq7$, then $M^{n}$ can be chosen to be contractible.

\begin{proof}
Let $X=\operatorname*{Tel}\left(  \left\{  K_{i},f_{i}\right\}  \right)  $ as
constructed in the previous Proposition. With some extra care, arrange for $X$
to be a simplicial $3$-complex, and choose a proper PL embedding into $%
\mathbb{R}
^{n+1}$. Let $N^{n+1}$ be a regular neighborhood of that embedding. It is easy
to see that $\operatorname*{pro}$-$\pi_{1}(N^{n+1},r)$ is identical to
$\operatorname*{pro}$-$\pi_{1}\left(  \varepsilon(X),r\right)  $, so if
boundary is permitted, we are finished. If not, let $M^{n}=\partial N^{n+1}$.
By general position, the base ray $r$ may be slipped off $X$ and then isotoped
to a ray $r^{\prime}$ in $M^{n}$. Also by general position, loops and disks in
$N^{n+1}$ may be slipped off $X$ and then pushed into $M^{n}$. In doing so,
one sees that $\operatorname*{pro}$-$\pi_{1}\left(  M^{n},r^{\prime}\right)  $
is pro-isomorphic to $\operatorname*{pro}$-$\pi_{1}\left(  N^{n+1},r\right)  $.
\end{proof}
\end{proposition}

In the study of compact manifolds, results like Poincar\'{e} duality place
significant restrictions on the topology of closed manifolds. A similar
phenomenon occurs in the study of noncompact manifolds. In that setting, it is
the open manifolds (and to a similar extent, manifolds with compact boundary)
that are the more rigidly restricted. If an open manifold is required to
satisfy additional niceness conditions, such as contractibility, finite
homotopy type, or inward tameness, even more rigidity comes into play. This is
at the heart of the study of noncompact manifolds, where a goal is to obtain
strong conclusions about the structure of a manifold from modest hypotheses.

\begin{exercise}
\label{Exercise: inward tame implies finite-ended}Show that an inward tame
manifold $M^{n}$ with compact boundary cannot have infinitely many ends.
(\emph{Hint: }Homology with $%
\mathbb{Z}
_{2}$-coefficients simplifies the algebra and eliminates issues related to
orientability.) Show that this result fails if we omit the tameness hypothesis
or if $M^{n}$ is permitted to have noncompact boundary.
\end{exercise}

\begin{exercise}
Show that the inverse sequence realized in Example
\ref{Example: realizing Z<-x2-Z<-x2-Z} cannot occur as $\operatorname*{pro}%
$-$\pi_{1}\left(  \varepsilon(M^{n}),r\right)  $ for a contractible open
manifold. \emph{Hint:}\textbf{ }A look ahead to
\S \ref{Subsection: Another look at contractible open manifolds} may be helpful.
\end{exercise}

The trick used in the proof of Proposition
\ref{Prop: mapping telescope realization fro pro-p1} for obtaining a
\emph{contractible} mapping telescope with the same end behavior as one that
is homotopically nontrivial is often useful. Given an inverse sequence
$\left\{  K_{i}\right\}  $ of finite CW complexes, the \emph{augmented inverse
sequence} $\left\{  K_{i},f_{i}\right\}  ^{\bullet}$ is obtained by inserting
a singleton space at the beginning of $\left\{  K_{i},f_{i}\right\}  $; the
corresponding \emph{contractible mapping telescope }$\operatorname*{CTel}%
\left(  \left\{  K_{i},f_{i}\right\}  \right)  $ is contractible, but
identical to $\operatorname*{Tel}\left(  \left\{  K_{i},f_{i}\right\}
\right)  $ at infinity.

\subsection{On the role of the base
ray\label{Section: On the role of the base ray}}

We now begin the detailed discussion of the role of base rays in the
fundamental group at infinity---a topic more subtle and more interesting than
one might expect.

As hinted earlier, small changes in base ray, such as reparametrization or
deletion of an initial segment, do not alter $\operatorname*{pro}$-$\pi
_{1}\left(  \varepsilon\left(  X\right)  ,r\right)  $; this follows from a
more general result to be presented shortly. On the other hand, large changes
can obviously have an impact. For example, if $X$ is multi-ended and $r$ and
$s$ point to different ends, then $\operatorname*{pro}$-$\pi_{1}\left(
\varepsilon\left(  X\right)  ,r\right)  $ and $\operatorname*{pro}$-$\pi
_{1}\left(  \varepsilon\left(  X\right)  ,s\right)  $ provide information
about different portions of $X$ ---much as the traditional fundamental group
of a non-path-connected space provides different information when the base
point is moved from one component to another. When $r$ and $s$ point to the
same end of $X$, it is reasonable to expect $\operatorname*{pro}$-$\pi
_{1}\left(  \varepsilon\left(  X\right)  ,r\right)  $ and $\operatorname*{pro}%
$-$\pi_{1}\left(  \varepsilon\left(  X\right)  ,s\right)  $ to be
$\operatorname*{pro}$-isomorphic---but this is not the case either! At the
heart of the matter is the difference between the set of ends $\mathcal{E}%
\left(  X\right)  $ and the set of strong ends $\mathcal{SE}\left(  X\right)
$. The following requires some effort, but the proof is completely elementary.

\begin{proposition}
\label{Prop: proper rays}If proper rays $r$ and $s$ in $X$ are strongly
equivalent, i.e., properly homotopic, then $\operatorname*{pro}$-$\pi
_{1}\left(  \varepsilon\left(  X\right)  ,r\right)  $ and $\operatorname*{pro}%
$-$\pi_{1}\left(  \varepsilon\left(  X\right)  ,s\right)  $ are pro-isomorphic.

\begin{corollary}
If $X$ is strongly connected at infinity, i.e., $\left\vert \mathcal{SE}%
\left(  X\right)  \right\vert =1$, then $\operatorname*{pro}$-$\pi_{1}\left(
\varepsilon\left(  X\right)  \right)  $ is a well-defined invariant of $X$.
\end{corollary}
\end{proposition}

\begin{exercise}
Prove Proposition \ref{Prop: proper rays}.
\end{exercise}

\begin{remark}
\emph{There are useful analogies between the role played by base points in the
fundamental group and that played by base rays in the fundamental group at
infinity:}

\begin{itemize}
\item \emph{The fundamental group is a functor from the category of} pointed
spaces, i.e.\emph{, pairs }$\left(  Y,p\right)  $\emph{, where }$p\in
Y$\emph{, to the category of groups. In a similar manner, the fundamental
group at infinity is a functor from the proper category of pairs }$\left(
X,r\right)  $\emph{, where }$r$\emph{ is a proper ray in }$X$\emph{, to the
category }$\operatorname*{pro}$\emph{-$\mathcal{G}$roups.}

\item \emph{If there is a path }$\alpha$\emph{ in }$Y$\emph{ from }$p$\emph{
to }$q$\emph{ in }$Y$\emph{, there is a corresponding isomorphism
}$\widehat{\alpha}:\pi_{1}\left(  Y,p\right)  \rightarrow\pi_{1}\left(
Y,q\right)  $\emph{. If there is a proper homotopy in }$X$\emph{ between
proper rays }$r$\emph{ and }$s$\emph{, then there is a corresponding
pro-isomorphism between }$\operatorname*{pro}$\emph{-}$\pi_{1}\left(
\varepsilon\left(  X\right)  ,r\right)  $\emph{ and }$\operatorname*{pro}%
$\emph{-}$\pi_{1}\left(  \varepsilon\left(  X\right)  ,s\right)  $\emph{.}

\item \emph{Even for connected }$Y$\emph{ there may be no relationship between
}$\pi_{1}\left(  Y,p\right)  $\emph{ and }$\pi_{1}\left(  Y,q\right)  $\emph{
when there is no path connecting }$p$\emph{ to }$q$\emph{. Similarly, for a
1-ended space }$X$\emph{, }$\operatorname*{pro}$\emph{-}$\pi_{1}\left(
\varepsilon\left(  X\right)  ,r\right)  $\emph{ and }$\operatorname*{pro}%
$\emph{-}$\pi_{1}\left(  \varepsilon\left(  X\right)  ,s\right)  $\emph{ may
be very different if there is no proper homotopy from }$r$\emph{ to }%
$s$\emph{.\medskip}
\end{itemize}
\end{remark}

We wish to describe a 1-ended $Y$ with proper rays $r$ and $s$ for which
$\operatorname*{pro}$-$\pi_{1}\left(  \varepsilon\left(  X\right)  ,r\right)
$ and $\operatorname*{pro}$-$\pi_{1}\left(  \varepsilon\left(  X\right)
,s\right)  $ are not pro-isomorphic. We begin with an intermediate space.

\begin{example}
[Another space with $\mathcal{SE}\left(  X\right)  \neq\mathcal{E}\left(
X\right)  $]\label{Example: SE(X) not E(X)}Let $X=\operatorname*{CTel}\left(
\left\{  \mathbb{S}_{i}^{1},f_{i}\right\}  \right)  $ where each
$\mathbb{S}_{i}^{1}$ is a copy of the unit circle and $f_{i}:\mathbb{S}%
_{i}^{1}\rightarrow\mathbb{S}_{i-1}^{1}$ is the standard degree 2 map (see
Example \ref{Example: realizing Z<-x2-Z<-x2-Z}). If $p_{i}$ is the canonical
base point for $\mathbb{S}_{i}^{1}$ and $f_{i}\left(  p_{i}\right)  =p_{i-1}$
for all $i$, we may construct a \textquotedblleft straight\textquotedblright%
\ proper ray $r$ by concatenating the mapping cylinder lines $\alpha_{i}$
connecting $p_{i}$ and $p_{i-1}$. Construct a second proper ray $s$ by
splicing between each $\alpha_{i}$ and $\alpha_{i+1}$ a loop $\beta_{i}$ that
goes once in the positive direction around $\mathbb{S}_{i}^{1}$; in other
words, $s=\alpha_{0}\cdot\beta_{0}\cdot\alpha_{1}\cdot\beta_{1}\cdot\alpha
_{2}\cdot\cdots$. With some effort, it can be shown that $r$ and $s$ are not
properly homotopic. That observation is also a corollary of the next example.
\end{example}

\begin{example}
For each $i$, let $K_{i}$ be a wedge of two circles and let $g_{i}%
:K_{i}\rightarrow K_{i-1}$ send one of those circles onto itself by the
identity and the other onto itself via the standard degree 2 map. Let
$Y=\operatorname*{CTel}\left(  \left\{  K_{i},g_{i}\right\}  \right)  $. This
space may be viewed as the union of $X$ from Example
\ref{Example: SE(X) not E(X)} and an infinite cylinder $\mathbb{S}^{1}%
\times\lbrack0,\infty)$, coned off at the left end, with the union identifying
the ray $r$ with a standard ray in the product. By viewing $X$ as a subset of
$Y$, view $r$ and $s$ as proper rays in $Y.$

Choose neighborhoods of infinity $U_{i}$ as described in
\S \ \ref{Subsection: Inverse mapping telescopes}. Each has fundamental group
that is free of rank 2. If we let $F_{i}$ be the free group of rank 2 with
formal generators $a^{2i}$ and $b$ then, $\operatorname*{pro}$-$\pi_{1}\left(
\varepsilon\left(  Y\right)  ,r\right)  $ may be represented by
\[
\left\langle a,b\right\rangle \hookleftarrow\left\langle a^{2},b\right\rangle
\hookleftarrow\left\langle a^{4},b\right\rangle \hookleftarrow\cdots.
\]
Similarly $\operatorname*{pro}$-$\pi_{1}\left(  \varepsilon\left(  Y\right)
,s\right)  $ may be represented by
\[
\left\langle a,b\right\rangle \overset{\lambda_{1}}{\longleftarrow
}\left\langle a^{2},b\right\rangle \overset{\lambda_{2}}{\longleftarrow
}\left\langle a^{4},b\right\rangle \overset{\lambda_{3}}{\longleftarrow}%
\cdots.
\]
where $\lambda_{i}\left(  a^{2i}\right)  =a^{2i}$ and $\lambda_{i}\left(
b\right)  =a^{2i}ba^{-2i}$. Taking inverse limits, produces $\check{\pi}%
_{1}\left(  \varepsilon\left(  Y\right)  ,r\right)  \allowbreak=\allowbreak
\left\langle b\right\rangle \allowbreak\cong\allowbreak%
\mathbb{Z}
$ and $\check{\pi}_{1}\left(  \varepsilon\left(  Y\right)  ,s\right)
\allowbreak=\allowbreak1$. Hence $\operatorname*{pro}$-$\pi_{1}\left(
\varepsilon\left(  Y\right)  ,r\right)  $ and $\operatorname*{pro}$-$\pi
_{1}\left(  \varepsilon\left(  Y\right)  ,s\right)  $ are not pro-isomorphic.
\end{example}

\begin{exercise}
Verify the assertions made in each of the two previous Examples.
\end{exercise}

The fact that a 1-ended space can have multiple fundamental groups at infinity
might lead one to doubt the value of that invariant. Over the next several
sections we provide evidence to counter that impression. For example, we will
investigate some properties of pro-$\pi_{1}$ that persist under change of base
ray. Furthermore, we will see that in some of the most important situations,
there is (verifiably in many cases and conjecturally in others) just one
proper homotopy class of base ray---causing the ambiguity to vanish. As an
example, the following important question is open.\medskip

\begin{conjecture}
[The Manifold Semistability Conjecture--version 1]%
\label{Question: semistability for universal covers}The universal cover of a
closed aspherical manifold of dimension greater than $1$ is always strongly
connected at infinity?\medskip
\end{conjecture}

We stated the above problem as a conjecture because it is a special case of
the following better-known conjecture. For now the reader can guess at the
necessary definitions. The meaning will be fully explained in
\S \ref{Section: Exploring the ends of groups}. The naming of these
conjectures will be explained over the next couple of pages.\medskip

\begin{conjecture}
[The Semistability Conjecture-version 1]Every finitely presented 1-ended group
is strongly connected at infinity.
\end{conjecture}

\subsection{Flavors of inverse sequences of groups}

When dealing with pro-isomorphism classes of inverse sequences of groups,
general properties are often more significant than the sequences themselves.
In this section we discuss several such properties.\smallskip

Let $G_{0}\overset{\mu_{1}}{\longleftarrow}G_{1}\overset{\mu_{2}%
}{\longleftarrow}G_{2}\overset{\mu_{3}}{\longleftarrow}G_{2}\overset{\mu
_{4}}{\longleftarrow}\cdots$ be an inverse sequence of groups. We say that
$\left\{  G_{i},\mu_{i}\right\}  $ is

\begin{itemize}
\item \emph{pro-trivial} if it is pro-isomorphic to the trivial inverse
sequence $1\leftarrow1\leftarrow1\leftarrow1\leftarrow\cdots,$

\item \emph{stable} if it is pro-isomorphic to an inverse sequence $\left\{
H_{i},\lambda_{i}\right\}  $ where each $\lambda_{i}$ is an isomorphism, or
equivalently, a constant inverse sequence $\left\{  H,\operatorname*{id}%
_{H}\right\}  $,

\item \emph{semistable }(or \emph{Mittag-Leffler, }or\emph{\ pro-epimorphic})
if it is pro-isomorphic to an $\left\{  H_{i},\lambda_{i}\right\}  $, where
each $\lambda_{i}$ is an epimorphism, and

\item \emph{pro-monomorphic} if it is pro-isomorphic to an $\left\{
H_{i},\lambda_{i}\right\}  $, where each $\lambda_{i}$ is a
monomorphism.\medskip
\end{itemize}

\noindent The following easy exercise will help the reader develop intuition
for the above definitions, and for the notion of pro-isomorphism itself.

\begin{exercise}
\label{Exercise: not pro-mono and not pro-epi}Show that an inverse sequence of
non-injective epimorphisms cannot be pro-monomorphic, and that an inverse
sequence of non-surjective monomorphisms cannot be semistable.
\end{exercise}

\begin{exercise}
\label{Exercise: stable inverse sequences}Show that if $\left\{  G_{i},\mu
_{i}\right\}  $ is stable and thus pro-isomorphic to some $\left\{
H,\operatorname*{id}_{H}\right\}  $, then $H$ is well-defined up to
isomorphism. In that case $H\cong\underleftarrow{\lim}\left\{  G_{i},\mu
_{i}\right\}  $.
\end{exercise}

A troubling aspect of the above definitions is that the concepts appear to be
extrinsic, requiring a second unseen sequence, rather than being intrinsic to
the given sequence. A standard result corrects that misperception.

\begin{proposition}
\label{Prop: passing to images}An inverse sequence of groups $\left\{
G_{i},\lambda_{i}\right\}  $ is stable if and only if it contains a
subsequence for which \textquotedblleft passing to images\textquotedblright%
\ results in an inverse sequence of isomorphisms, in other words: we may
obtain a diagram of the following form, where all unlabeled homomorphisms are
obtained by restriction or inclusion.%
\begin{equation}
\begin{diagram} G_{i_{0}} & & \lTo^{\lambda_{i_{0},i_{1}}} & & G_{i_{1}} & & \lTo^{\lambda_{i_{1},i_{2}}} & & G_{i_{2}} & & \lTo^{\lambda_{i_{2},i_{3}}}& & G_{i_{3}}& \cdots\\ & \luInto & & \ldTo & & \luInto & & \ldTo & & \luInto & & \ldTo & \\ & & \operatorname{Im}\left( \lambda_{i_{0},i_{1}}\right) & & \lTo^{\cong} & & \operatorname{Im}\left( \lambda_{i_{1},i_{2}}\right) & &\lTo^{\cong} & & \operatorname{Im}\left( \lambda_{i_{2},i_{3}}\right) & & \lTo^{\cong} & &\cdots & \\ \end{diagram}\tag{3.10}%
\label{Diagram: passing to images}%
\end{equation}

\noindent Analogous statements are true for the pro-epimorphic and
pro-monomorphic sequences; in those cases we require maps in the bottom row of
(\ref{Diagram: passing to images}) to be epimorphisms, and monomorphisms, respectively.
\end{proposition}

Proof of the above is an elementary exercise, as is the following:

\stepcounter{theorem}

\begin{proposition}
An inverse sequence is stable if and only if it is both pro-\allowbreak
epi\-mor\-ph\-ic and pro-\allowbreak mon\-o\-mor\-ph\-ic.
\end{proposition}

\begin{exercise}
Prove the previous two Propositions.
\end{exercise}

\subsection{Some topological interpretations of the previous
definitions\label{Subsection: topological interpretations}}

It is common practice to characterize simply connected spaces topologically
(without mentioning the word `group'), as path-connected spaces in which every
loop contracts to a point. In that spirit, we provide topological
characterizations of spaces whose fundamental groups at infinity possess some
of the algebraic properties discussed in the previous section.

\begin{proposition}
\label{Prop: top characerization 1-connected at infinity}For a 1-ended space
$X$ and a proper ray $r$, $\operatorname*{pro}$-$\pi_{1}\left(  \varepsilon
\left(  X\right)  ,r\right)  $ is

\begin{enumerate}
\item pro-trivial if and only if: for any compact $C\subseteq X$, there exists
a larger compact set $D$ such that every loop in $X-D$ contracts in $X-C$,

\item semistable if and only if: for any compact $C\subseteq X$, there exists
a larger compact set $D$ such that, for every still larger compact $E$, each
pointed loop $\alpha$ in $X-D$ based on $r$ can be homotoped into $X-E$ via a
homotopy into $X-C$ that slides the base point along $r$, and

\item pro-monomorphic if and only if there exists a compact $C\subseteq X$
such that, for every compact set $D$ containing $C$, there exists a compact
$E$ such that every loop in $X-E$ that contracts in $X-C$ contracts in $X-D$.
\end{enumerate}

\begin{proof}
This is a straightforward exercise made easier by applying Proposition
\ref{Prop: passing to images}.
\end{proof}
\end{proposition}

Note that the topological condition in part (a) of Proposition
\ref{Prop: top characerization 1-connected at infinity} makes no mention of a
base ray. So (for 1-ended spaces) the property of having pro-trivial
fundamental group at infinity is independent of base ray; such spaces are
called \emph{simply connected at infinity}. Similarly, the topological
condition in (c) is independent of base ray; 1-ended spaces with that property
are called \emph{pro-monomorphic at infinity }(or simply
\emph{pro-monomorphic}). And despite the (unavoidable) presence of a base ray
in the topological portion of (b), there does exist an elegant and useful
characterization of spaces with semistable pro-$\pi_{1}$.

\begin{proposition}
\label{Prop: strongly connected at infinity}A 1-ended space $X$ is strongly
connected at infinity if and only if there exists a proper ray $r$ for which
$\operatorname*{pro}$-$\pi_{1}\left(  \varepsilon(X),r\right)  $ is semistable.

\begin{proof}
[Sketch of proof]First we outline a proof of the reverse implication. Let $r$
be as in the hypothesis and let $s$ be another proper ray. By 1-endedness,
there is a proper map $h$ of the infinite ladder $L_{[0,\infty)}%
=([0,\infty)\times\left\{  0,1\right\}  )\cup(%
\mathbb{N}
\times\lbrack0,1])$ into $X$, with $\left.  h\right\vert _{[0,\infty)\times
0}=r$ and $\left.  h\right\vert _{[0,\infty)\times1}=s$. For convenience,
choose an exhaustion of $X$ by compacta $\varnothing=C_{0}\subseteq
C_{1}\subseteq C_{2}\subseteq\cdots$ with the property that the subladder
$L_{[i,\infty)}$ is sent into $U_{i}=\overline{X-C_{i}}$ for each $i\geq1$. As
a simplifying hypothesis, assume that all bonding homomorphisms in the
corresponding inverse sequence%
\begin{equation}
\pi_{1}\left(  X,p_{0}\right)  \overset{\lambda_{1}}{\longleftarrow}\pi
_{1}\left(  X-C_{1},p_{1}\right)  \overset{\lambda_{2}}{\longleftarrow}\pi
_{1}\left(  X-C_{2},p_{2}\right)  \overset{\lambda_{3}}{\longleftarrow}%
\cdots\tag{3.11}\label{above}%
\end{equation}
are surjective. (For a complete proof, one should instead apply Proposition
\ref{Prop: top characerization 1-connected at infinity} inductively.)

We would like to extend $h$ to a proper map of $[0,\infty)\times\left[
0,1\right]  $ into $X$. To that end, let $\square_{i}$ be the loop in $X$
corresponding to $r_{i+1}\cup e_{i+1}\cup s_{i+1}^{-1}\cup e_{i}^{-1}$ in
$L_{[0,\infty)}$. (Here $r_{i+1}=\left.  r\right\vert _{\left[  i,i+1\right]
}$ and $s_{i+1}=\left.  s\right\vert _{\left[  i,i+1\right]  }$;
$e_{j}=\left.  h\right\vert _{j\times\lbrack0,1]}$, the $j^{\text{th}}$
\textquotedblleft rung\textquotedblright\ of the ladder.)

If each $\square_{i}$ contracts in $X$ we can use those contractions to extend
$h$ to $[0,\infty)\times\left[  0,1\right]  $; if each $\square_{i}$ contracts
in $X-C_{i}$ the resulting extension is proper (as required). The idea of the
proof is to arrange those conditions. Begin inductively with $\square_{0}$. If
this loop does not contract in $X$, we make it so by rechoosing $e_{1}$ as
follows: choose a loop $\alpha_{1}$ based at $p_{1}$ so that $r_{1}\cdot
\alpha_{1}\cdot r_{1}^{-1}$ is equal to $\square_{0}$ in $\pi_{1}\left(
X,p_{0}\right)  $. Replace $e_{1}$ with the rung $\hat{e}_{1}=\alpha_{1}%
^{-1}\cdot e_{1}$. The newly modified $\square_{0}$ contracts in $X$, as
desired. Now move to the correspondingly modified $\square_{1}$ viewed as an
element of $\pi_{1}\left(  X-C_{1},p_{1}\right)  $. If it is nontrivial,
choose a loop $\alpha_{2}$ in $X-C_{2}$ based at $p_{2}$ such that
$\lambda_{2}\left(  \alpha_{2}\right)  =r_{2}\cdot\alpha_{2}\cdot r_{2}%
^{-1}=\square_{1}$. Replacing $e_{2}$ with $\hat{e}_{2}=\alpha_{2}^{-1}\cdot
e_{2}$ results in a further modified $\square_{1}$ that contracts in $X-C_{1}%
$. Continue this process inductively to obtain a proper homotopy
$H:[0,\infty)\times\left[  0,1\right]  \rightarrow X$ between $r$ and $s$.

For the reverse implication, assume that (\ref{above}) is not semistable. One
creates a proper ray $s$ not properly homotopic to $r$ by affixing to each
vertex $p_{i}$ of $r$ a loop in $\beta_{i}$ in $X-C_{i}$ that that does not
lie in the image of $\pi_{1}\left(  X-C_{i+1},p_{i+1}\right)  $. More
specifically
\[
s=r_{1}\cdot\alpha_{1}\cdot r_{2}\cdot\alpha_{2}\cdot r_{3}\cdot\alpha
_{3}\cdot\cdots.
\]

\end{proof}
\end{proposition}

As a result of Proposition \ref{Prop: strongly connected at infinity}, a
1-ended space $X$ may be called \emph{semistable at infinity} [respectively,
\emph{stable at infinity}] if $\operatorname*{pro}$-$\pi_{1}\left(
\varepsilon(X),r\right)  $ is semistable [respectively, stable] for some (and
hence any) proper ray $r$. Alternatively, a 1-ended space is sometimes
\emph{defined} to be semistable at infinity (or just \emph{semistable}) if all
proper rays in $X$ are properly homotopic. In those cases we often drop the
base ray and refer to the homotopy end invariants simply as
$\operatorname*{pro}$-$\pi_{1}\left(  \varepsilon\left(  X\right)  \right)  $
and $\check{\pi}_{1}\left(  \varepsilon\left(  X\right)  \right)  $.

Multi-ended spaces are sometimes called semistable if, whenever two proper
rays determine the same end, they are properly homotopic; or equivalently,
when $\Phi:\mathcal{SE}\left(  X\right)  \rightarrow\mathcal{E}\left(
X\right)  $ is bijective.

\begin{remark}
\emph{By using the sketched proof of Proposition
\ref{Prop: strongly connected at infinity} as a guide, it is not hard to see
why a 1-ended space }$X$\emph{ that is not semistable will necessarily have
}uncountable\emph{ }$\mathcal{SE}\left(  X\right)  $.\emph{ A method for
placing} $\mathcal{SE}\left(  X\right)  $ \emph{into an algebraic context
involves the} derived limit\emph{ or }`$\lim^{1}$ functor'. \emph{More
generally,} $\lim^{1}\left\{  G_{i},\mu_{i}\right\}  $ \emph{is an algebraic
construct that helps to recover the information lost when one passes from an
inverse sequence to its inverse limit. See \cite[\S 11.3]{Ge2}.}
\end{remark}

\section{Applications of end invariants to manifold topology}

Although a formal study of pro-homotopy and pro-homology of the ends of
noncompact space is not a standard part of the education of most manifold
topologists, there are numerous important results and open questions best
understood in that context. In this section we discuss several of those,
beginning with classical results and moving toward recent work and still-open questions.

\subsection{Another look at contractible open
manifolds\label{Subsection: Another look at contractible open manifolds}}

We now return to the study of contractible open manifolds begun in
\S \ref{Section: Motivating examples}. We will tie up some loose ends from
those earlier discussions---most of which focused on specific examples. We
also present some general results whose hypotheses involve nothing more than
the fundamental group at infinity.

\begin{theorem}
[Whitehead's Exotic Open $3$-manifold]There exists a contractible open
$3$-manifold not homeomorphic to $%
\mathbb{R}
^{3}$.

\begin{proof}
We wish to nail down a proof that the Whitehead contractible $3$-manifold
$\mathcal{W}^{3}$ described in
\S \ref{Subsection: classic examples contractible manifolds} is not
homeomorphic to $%
\mathbb{R}
^{3}$. We do that by showing $\mathcal{W}^{3}$ is not simply connected at
infinity. Using the representation of $\operatorname*{pro}$-$\pi_{1}\left(
\varepsilon(\mathcal{W}^{3}),r\right)  $ obtained in
\S \ref{Subsection: Examples of fundamental groups at infinity} and applying
the rigorous development from
\S \ref{Section: Algebraic invariants-precise definitions}, we can accomplish
that task with an application of Exercise
\ref{Exercise: not pro-mono and not pro-epi}.
\end{proof}
\end{theorem}

\begin{theorem}
The open Newman contractible $n$-manifolds are not homeomorphic to $%
\mathbb{R}
^{n}$. More generally, any compact contractible $n$-manifold with non-simply
connected boundary has interior that is not homeomorphic to $%
\mathbb{R}
^{n}$.

\begin{proof}
Combine our observations from Example
\ref{Example: Fundamental group at infinty for open Newman manifolds} with
Exercise \ref{Exercise: stable inverse sequences}---or simply observe that the
topological characterization of simply connected at infinity fails.
\end{proof}
\end{theorem}

The next result establishes simple connectivity at infinity as the definitive
property in determining whether a contractible open manifold is exotic. The
initial formulation is due to Stallings \cite{St}, who proved it for PL
manifolds of dimension $\geq5$; his argument is clean, elegant, and highly
recommended---but outside the scope of these notes. That result was extended
to all topological manifolds of dimension $\geq5$ by Luft \cite{Lu}. Extending
the result to dimensions $3$ and $4$ requires the Fields Medal winning work of
Perelman and Freedman \cite{Free}, respectively. The foundation for the
3-dimensional result was laid by C.H. Edwards in \cite{EdC}.

\begin{theorem}
[Stallings' Characterization of $%
\mathbb{R}
^{n}$]\label{Theorem: Stallings Theorem} A contractible open $n$-manifold
($n\geq3$) is homeomorphic to $%
\mathbb{R}
^{n}$ if and only if it is simply connected at infinity.
\end{theorem}

\begin{exercise}
\label{Exercise: contractible cross R}Prove the following corollary to Theorem
\ref{Theorem: Stallings Theorem}: If $W^{n}$ is a contractible open manifold,
then $W^{n}\times%
\mathbb{R}
\approx%
\mathbb{R}
^{n+1}$.
\end{exercise}

The next application of the fundamental group at infinity returns us to
another prior discussion.

\begin{theorem}
[Davis' Exotic Universal Covering Spaces]For $n\geq4$, there exist closed
aspherical $n$-manifolds whose universal covers are not homeomorphic to $%
\mathbb{R}
^{n}$.

\begin{proof}
Here we provide only the punch-line to this major theorem. As noted in
\S \ref{Subsection: classic examples contractible manifolds} Davis'
construction produces closed aspherical $n$-manifolds $M^{n}$ with universal
covers homeomorphic to the infinite open sums described in Example
\ref{Example: Construction of Davis manifolds} and Theorem
\ref{Theorem: Ancel-Siebenmann}. As observed in Example
\ref{Example: Fundamental group at infinity for Davis manifolds},
$\operatorname*{pro}$-$\pi_{1}\left(  \varepsilon(\widetilde{M}^{n}),r\right)
$ may be represented by
\begin{equation}
G\twoheadleftarrow G\ast G\twoheadleftarrow G\ast G\ast G\twoheadleftarrow
G\ast G\ast G\ast G\twoheadleftarrow\cdots,\tag{3.12}%
\label{Inverse sequence: Davis type}%
\end{equation}
a sequence that is semistable but not pro-monomorphic. An application of
Exercise \ref{Exercise: not pro-mono and not pro-epi} verifies that
$\widetilde{M}^{n}$ is not simply connected at infinity.
\end{proof}
\end{theorem}

After Davis showed that aspherical manifolds need not be covered by $%
\mathbb{R}
^{n}$, many questions remained. With the 3-dimensional version unresolved (at
the time), it was asked whether the Whitehead manifold could cover a closed
$3$-manifold. In higher dimensions, people wondered whether a Newman
contractible open manifold (or the interior of another compact contractible
manifold) could cover a closed manifold. Myers \cite{My} resolved the first
question in the negative, before Wright \cite{Wr} proved a remarkably general
result in which the fundamental group at infinity plays the central role.

\stepcounter{theorem}

\begin{theorem}
[Wright's Covering Space Theorem]Let $M^{n}$ be a contractible open
$n$-manifold with pro-\allowbreak mono\-mor\-phic fundamental group at
infinity. If $M^{n}$ admits a nontrivial action by covering transformations,
then $M^{n}\approx\mathbb{R}^{n}$.
\end{theorem}

\begin{corollary}
Neither the Whitehead manifold nor the interior of any compact contractible
manifold with non-simply connected boundary can cover a manifold nontrivially.
\end{corollary}

Wright's theorem refocuses attention on a question mentioned earlier.

\begin{conjecture}
[The Manifold Semistability Conjecture]%
\label{Question: semistability for manifolds-version 2}Must the universal
cover of every closed aspherical manifold have semistable fundamental group at infinity?
\end{conjecture}

\noindent More generally we can ask:\medskip

\noindent\textbf{Vague Question:} \emph{Must all universal covers of
aspherical manifolds be similar to the Davis examples?\medskip}

In discussions still to come, we will make this vague question more precise.
But, before moving on, we note that in 1991 Davis and Januszkiewicz \cite{DJ}
invented a new strategy for creating closed aspherical manifolds with exotic
universal covers. Although that strategy is very different from Davis'
original approach, the resulting exotic covers are remarkably similar. For
example, their fundamental groups at infinity are precisely of the form
(\ref{Inverse sequence: Davis type}).

\begin{exercise}
\label{Exercise: contractible open n-manifolds have pro-homology of R^n}%
Theorem \ref{Theorem: Stallings Theorem} suggests that the essence of a
contractible open manifold is contained in its fundamental group at infinity.
Show that every contractible open n-manifold $W^{n}$ has the same
\emph{homology} at infinity as $%
\mathbb{R}
^{n}$. In particular, show that for all $n\geq2$, $\operatorname*{pro}$%
-$H_{i}\left(  W^{n};%
\mathbb{Z}
\right)  $ is stably $%
\mathbb{Z}
$ if $i=0$ or $n-1$ and pro-trivial otherwise. \emph{Note:}\textbf{ }This
exercise may be viewed as a continuation of Exercise
\ref{Exercise: contractible open mflds are 1-ended}.
\end{exercise}

\subsection{Siebenmann's thesis\label{Subsection: Siebenmann's thesis}}

Theorem \ref{Theorem: Stallings Theorem} may be viewed as a classification of
those open manifolds that can be compactified to a closed $n$-ball by addition
of an $\left(  n-1\right)  $-sphere boundary. More generally, one may look to
characterize open manifolds that can be compactified to a manifold with
boundary by addition of a boundary $\left(  n-1\right)  $-manifold. Since the
boundary of a manifold $P^{n}$ always has a \emph{collar neighborhood}
$N\approx\partial P^{n}\times\lbrack0,1]$, an open manifold $M^{n}$ allows
such a com\-pact\-ific\-at\-ion if and only if it contains a neighborhood of
infinity homeomorphic to an \emph{open collar} $Q^{n-1}\times\lbrack0,1)$, for
some closed $\left(  n-1\right)  $-manifold $Q^{n-1}$. We refer to open
manifolds of this sort as being \emph{collarable}.

The following shows that, to characterize collarable open manifolds, it is not
enough to consider the fundamental group at infinity.

\begin{example}
Let $M^{n}$ be the result of a countably infinite collection of copies of
$\mathbb{S}^{2}\times\mathbb{S}^{n-2}$ connect-summed to $%
\mathbb{R}
^{n}$ along a sequence of $n$-balls tending to infinity (see Figure
\ref{Figure: Infinite handles}).%
\begin{figure}[ptb]%
\centering
\includegraphics[
height=1.6869in,
width=3.4968in
]%
{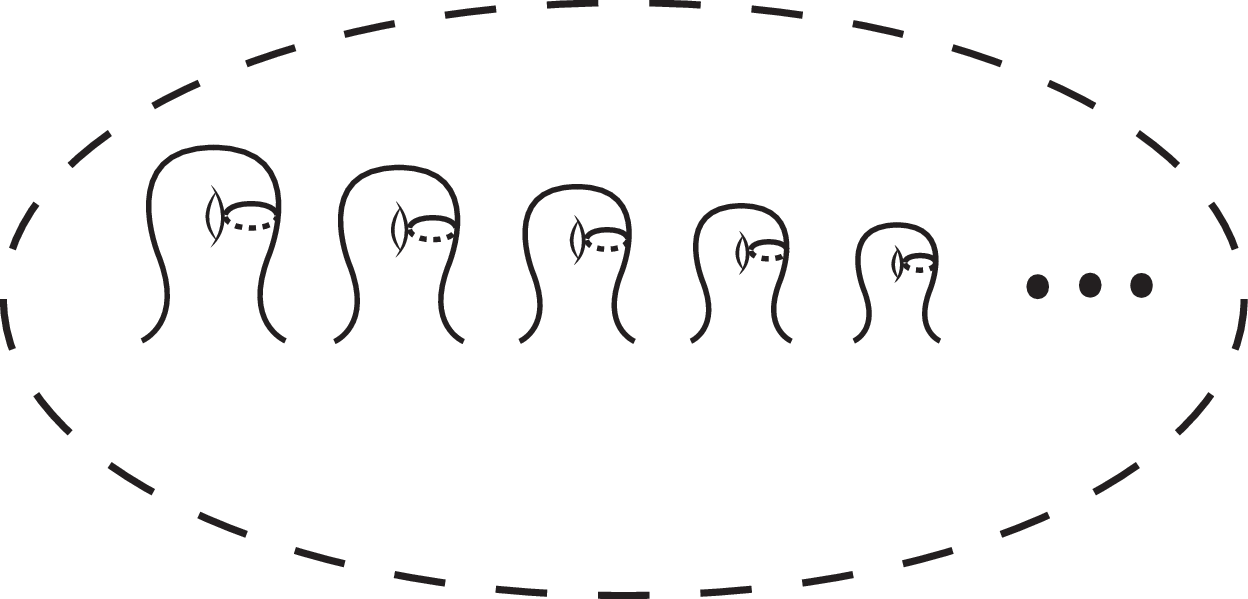}%
\caption{$\mathbb{R} ^{n}$ connect-summed with infinitely many $\mathbb{S}%
^{2}\times\mathbb{S}^{n-2}$}%
\label{Figure: Infinite handles}%
\end{figure}
Provided $n\geq4$, $M^{n}$ is simply connected at infinity. Moreover, since a
compact manifold with boundary has finite homotopy type, and since the
addition of a manifold boundary does not affect homotopy type, this $M^{n}$
admits no such com\-pact\-ific\-at\-ion.
\end{example}

For manifolds that are simply connected at infinity, the necessary additional
hypothesis is as simple as one could hope for.

\begin{theorem}
[See \cite{BLL}]Let $W^{n}$ be a 1-ended open $n$-manifold ($n\geq6$) that is
simply connected at infinity. Then $W^{n}$ is collarable if and only if
$H_{\ast}\left(  W;%
\mathbb{Z}
\right)  $ is finitely generated.
\end{theorem}

For manifolds not simply connected at infinity, the situation is more
complicated, but the characterization is still remarkably elegant. It is one
of the best-known and most frequently applied theorems in manifold topology.

\begin{theorem}
[Siebenmann's Collaring Theorem]\label{Theorem: Siebenmann's thesis}A 1-ended
n-manifold $W^{n}$ ($n\geq6$) with compact (possibly empty) boundary is
collarable if and only if

\begin{enumerate}
\item $W^{n}$ is inward tame,

\item $\operatorname*{pro}$-$\pi_{1}\left(  \varepsilon(W^{n})\right)  $ is
stable, and

\item $\sigma_{\infty}\left(  W^{n}\right)  \in\widetilde{K}_{0}\left(
\mathbb{Z}
\left[  \check{\pi}_{1}\left(  \varepsilon\left(  X\right)  \right)  \right]
\right)  $ is trivial.
\end{enumerate}
\end{theorem}

\begin{remark}
\emph{(1) Under the assumption of hypotheses (a) and (b), }$\sigma_{\infty
}\left(  W^{n}\right)  $ \emph{is }defined\emph{ to be the Wall finiteness
obstruction} $\sigma\left(  N\right)  $ \emph{of a single clean neighborhood
of infinity, chosen so that its fundamental group (under inclusion) matches}
$\check{\pi}_{1}\left(  \varepsilon\left(  X\right)  \right)  $. \emph{A more
general definition for }$\sigma_{\infty}\left(  W^{n}\right)  $ ---\emph{one
that can be used when} $\operatorname*{pro}$-$\pi_{1}\left(  \varepsilon
(W^{n})\right)  $ \emph{is not stable---will be introduced in
\S \ref{Subsection: Generalizing Siebenman}.\smallskip}

\noindent\emph{(2) Together, assumptions (a) and (c) are equivalent to
assuming that }$W^{n}$\emph{ is absolutely inward tame. That would allow for a
simpler statement of the Collaring Theorem; however, the power of the given
version is that it allows the finiteness obstruction to be measured on a
single (appropriately chosen) neighborhood of infinity. Furthermore, in a
number of important cases, }$\sigma_{\infty}\left(  W^{n}\right)  $\emph{ is
trivial for algebraic reasons. That is the case, for example, when }%
$\check{\pi}_{1}\left(  \varepsilon\left(  X\right)  \right)  $ \emph{is
trivial, free, or free abelian, by a fundamental result of algebraic K-theory
found in \cite{BHS}.\smallskip}

\noindent\emph{(3) Due to stability, no base ray needs to be mentioned in
Condition (b). Use of the \v{C}ech fundamental group in Condition (c) is just
a convenient way of specifying the single relevant group implied by Condition
(b) (see Exercise \ref{Exercise: stable inverse sequences}).\smallskip}

\noindent\emph{(4) Since an inward tame manifold with compact boundary is
necessarily finite-ended (see Exercise
\ref{Exercise: inward tame implies finite-ended}), the 1-ended hypothesis is
easily eliminated from the above by requiring each end to satisfy (b) and (c),
individually.}

\noindent\emph{(5) By applying \cite{Free}, Theorem
\ref{Theorem: Siebenmann's thesis} can be extended to dimension 5, provided
}$\check{\pi}_{1}\left(  \varepsilon\left(  X\right)  \right)  $ \emph{is a
\textquotedblleft good\textquotedblright\ group, in the sense of}
\emph{\cite{FQ}; whether the theorem holds for all 5-manifolds is an open
question.} \emph{Meanwhile, Kwasik and Schultz \cite{KS} have shown that
Theorem \ref{Theorem: Siebenmann's thesis} fails in dimension 4; partial
results in that dimension can be found in \cite[\S 11.9]{FQ}. By combining the
solution to the Poincar\'{e} Conjecture with work by Tucker \cite{Tu}, one
obtains a strong 3-Dimensional Collaring Theorem---only condition (a) is
necessary. For classical reasons, the same is true for }$n=2$\emph{. And for
}$n=1$\emph{, there are no issues.}
\end{remark}

The proof of Theorem \ref{Theorem: Siebenmann's thesis} is intricate in
detail, but simple in concept. Readers unfamiliar with h-cobordisms and
s-cobordisms, and their role in the topology of manifolds, should consult
\cite{RS}.

\begin{proof}
[Proof of Siebenmann's Theorem (outline)]Since a 1-ended collarable manifold
is easily seen to be absolutely inward tame with stable fundamental group at
infinity, conditions (a)-(c) are necessary. To prove sufficiency, begin with a
cofinal sequence $\left\{  N_{i}\right\}  _{i=0}^{\infty}$ of clean
neighborhoods of infinity with $N_{i+1}\subseteq\operatorname*{int}N_{i}$ for
all $i$. After some initial combinatorial group theory, a 2-dimensional disk
trading argument allows us to improve the neighborhoods of infinity so that,
for each $i$, $N_{i}$ and $\partial N_{i}$ have fundamental groups
corresponding to the stable fundamental group $\check{\pi}_{1}\left(
\varepsilon(W^{n})\right)  $. More precisely, each inclusion induces
isomorphisms $\pi_{1}\left(  \partial N_{i}\right)  \overset{\cong%
}{\rightarrow}\pi_{1}\left(  N_{i}\right)  $ and $\pi_{1}\left(
N_{i+1}\right)  \overset{\cong}{\rightarrow}\pi_{1}\left(  N_{i}\right)  $,
with each group being isomorphic to $\check{\pi}_{1}\left(  \varepsilon
(W^{n})\right)  $.

Under the assumption that one of these $N_{i}$ has trivial finiteness
obstruction, the \textquotedblleft Sum Theorem\textquotedblright\ for the Wall
obstruction (first proved in \cite{Si} for this purpose) together with the
above $\pi_{1}$-isomorphisms, implies that all $N_{i}$ have trivial finiteness
obstruction. From there, a carefully crafted sequence of modifications to
these neighborhoods of infinity---primarily handle manipulations---results in
a further improved sequence of neighborhoods of infinity with the property
that $\partial N_{i}\hookrightarrow N_{i}$ is a homotopy equivalence for each
$i$. The resulting cobordisms $\left(  A_{i},\partial N_{i},\partial
N_{i+1}\right)  $, where $A_{i}=\overline{N_{i}-N}_{i+1}$ are then
h-cobordisms (See Exercise
\ref{Exercise: cobordisms are [one-sided] h-cobordisms}).

A clever \textquotedblleft infinite swindle\textquotedblright\ allows one to
trivialize the Whitehead torsion of $\partial N_{i}\hookrightarrow A_{i}$ in
each h-cobordism by inductively borrowing the inverse h-cobordism $B_{i}$ from
a collar neighborhood of $\partial N_{i+1}$ in $A_{i+1}$ (after which the
\textquotedblleft new\textquotedblright\ $N_{i+1}$ is $\overline{N_{i+1}%
-B}_{i}$), until the s-cobordism theorem yields $A_{i}\approx\partial
N_{i}\times\lbrack i,i+1]$, for each \thinspace$i$. Gluing these products
together completes the proof.
\end{proof}

\begin{exercise}
\label{Exercise: cobordisms are [one-sided] h-cobordisms}Verify the
h-cobordism assertion in the above paragraph. In particular, let $N_{i}$ and
$N_{i+1}$ be clean neighborhoods of infinity with $\operatorname*{int}%
N_{i}\supseteq N_{i+1}$ satisfying the properties: \emph{(1)} $\partial
N_{i}\hookrightarrow N_{i}$ and $\partial N_{i+1}\hookrightarrow N_{i+1}$ are
homotopy equivalences and \emph{(2)}\textbf{ }$N_{i+1}\hookrightarrow N_{i}$
induces a $\pi_{1}$-isomorphism. For $A_{i}=\overline{N_{i}-N}_{i+1}$, show
that both $\partial N_{i}\hookrightarrow A_{i}$ and $\partial N_{i+1}%
\hookrightarrow A_{i}$ are homotopy equivalences.

Observe that in the absence of Condition (2), it is still possible to conclude
that $\left(  A_{i},\partial N_{i},\partial N_{i+1}\right)  $ is a
\textquotedblleft1-sided h-cobordism\textquotedblright, in particular,
$\partial N_{i}\hookrightarrow A_{i}$ is a homotopy equivalence.
\end{exercise}

In the spirit of the result in Exercise \ref{Exercise: contractible cross R},
the following may be obtained as an application of Theorem
\ref{Theorem: Siebenmann's thesis}.

\begin{theorem}
[\cite{Gu2}]\label{Theorem: Finite type x R}For an open manifold $M^{n}$
($n\geq5$), the \textquotedblleft stabilization\textquotedblright%
\ $M^{n}\times%
\mathbb{R}
$ is collarable if and only if $M^{n}$ has finite homotopy type.

\begin{proof}
[Sketch of proof]Since a collarable manifold has finite homotopy type, and
since $M^{n}\times%
\mathbb{R}
$ is homotopy equivalent to $M^{n}$, it is clear that $M^{n}$ must have finite
homotopy in order for $M^{n}\times%
\mathbb{R}
$ to be collarable. To prove sufficiency of that condition, we wish to verify
that the conditions Theorem \ref{Theorem: Siebenmann's thesis} are met by
$M^{n}\times%
\mathbb{R}
$.

Conditions (a) and (c) are relatively easy, and are left as an exercise (see
below). The key step is proving stability of $\operatorname*{pro}$-$\pi
_{1}\left(  \varepsilon(M^{n}\times%
\mathbb{R}
),r\right)  $. We will say just enough to convey the main idea---describing a
technique that has been useful in several other contexts. Making these
argument rigorous is primarily a matter of base points and base rays---a
nontrivial issue, but one that we ignore for now. (See \cite{Gu2} for the details.)

For simplicity, assume $M^{n}$ is 1-ended and $N_{0}\supseteq N_{1}\supseteq
N_{2}\supseteq\cdots$ is a cofinal sequence of clean connected neighborhoods
of infinity in $M^{n}$. If $R_{i}=\left(  M^{n}\times(-\infty,-i]\cup\lbrack
i,\infty)\right)  \allowbreak\cup\allowbreak\left(  N_{i}\times%
\mathbb{R}
\right)  $, then $\{R_{i}\}$ forms a cofinal sequence of clean connected
neighborhoods of infinity in $M^{n}\times%
\mathbb{R}
$. If $G=\pi_{1}\left(  M^{n}\right)  $ and $H_{i}=\operatorname*{Im}\left(
\pi_{1}\left(  N_{i}\right)  \allowbreak\rightarrow\allowbreak\pi_{1}\left(
M^{n}\right)  \right)  $ for each $i$, then $\pi_{1}\left(  R_{i}\right)
=G\ast_{H_{i}}G$ and $\operatorname*{pro}$-$\pi_{1}\left(  \varepsilon
(M^{n}\times%
\mathbb{R}
),r\right)  $ may be represented by
\[
G\ast_{H_{0}}G\twoheadleftarrow G\ast_{H_{1}}G\twoheadleftarrow G\ast_{H_{2}%
}G\twoheadleftarrow\cdots
\]
where the bonds are induced by the identities on $G$ factors. Notice that each
$H_{i+1}$ injects into $H_{i}$. To prove stability, it suffices to show that,
eventually, $H_{i+1}$ goes onto $H_{i}$. To that end, we argue that every loop
in $N_{i}$ can be homotoped into $N_{i+1}$ by a homotopy whose tracks may go
anywhere in $M^{n}$.\footnote{A complete proof would do this while keeping a
base point of the loop on a base ray $r$.} The loops of concern are those
lying in $N_{i}-N_{i+1}$; let $\alpha$ be such a loop, and assume it is an
embedded circle.

By the finite homotopy type of $M^{n}$ (in fact, finite domination is enough),
we may assume the existence of a homotopy $S$ that pulls $M^{n}$ into
$M^{n}-N_{i}$. Consider the map $J=\left.  S\right\vert _{\partial N_{i}%
\times\lbrack0,1]}$. Adjust $J$ so that it is transverse to the $1$-manifold
$\alpha$. Then $J^{-1}\left(  \alpha\right)  $ is a finite collection of
circles. With some extra effort we can see that at least one of those circles
goes homeomorphically onto $\alpha$. The strong deformation retraction of
$\partial N_{i}\times\lbrack0,1]$ onto $\partial N_{i}\times\left\{
0\right\}  $ composed with $J$ pushes $\alpha$ into $N_{i+1}$.
\end{proof}
\end{theorem}

\begin{exercise}
Show that for an open manifold $M^{n}$ with finite homotopy type, the special
neighborhoods of infinity $R_{i}\subseteq M^{n}\times%
\mathbb{R}
$, used in the above proof, have finite homotopy type. Therefore, $M^{n}\times%
\mathbb{R}
$ is absolutely inward tame.
\end{exercise}

\begin{exercise}
Show that if $M^{n}$ (as above) is finitely dominated, but does not have
finite homotopy type, then $M^{n}\times%
\mathbb{R}
$ satisfies Conditions (a) and (b) of Theorem
\ref{Theorem: Siebenmann's thesis}, but not Condition (c).
\end{exercise}

\subsection{Generalizing Siebenmann\label{Subsection: Generalizing Siebenman}}

Siebenmann's Collaring Theorem and a \textquotedblleft
controlled\textquotedblright\ version of it found in \cite{Quin1} have proven
remarkably useful in manifold topology; particularly in obtaining the sorts of
structure and embedding theorems that symbolize the tremendous activity in
high-dimensional manifold topology in the 1960's and 70's. But the discovery
of exotic universal covering spaces, along with a shift in research interests
(the Borel and Novikov Conjectures in particular and geometric group theory in
general) to topics where an understanding of universal covers is crucial,
suggests a need for results applicable to spaces with \emph{non-stable}
fundamental group at infinity. As an initial step, one may ask what can be
said about open manifolds satisfying some of Siebenmann's conditions---but not
$\pi_{1}$-stability. In \S \ref{Subsection: Inverse mapping telescopes} we
described a method for constructing locally finite polyhedra satisfying
Conditions (a) and (c) of Theorem \ref{Theorem: Stallings Theorem}, but having
almost arbitrary $\operatorname*{pro}$-$\pi_{1}$. By the same method, we could
build unusual behavior into $\operatorname*{pro}$-$H_{k}$. So it is a pleasant
surprise that, for manifolds with compact boundary, inward tameness by itself,
has significant implications.

\begin{theorem}
[{\cite[Th.1.2]{GT1}}]\label{Theorem: tame implies semistability}If a manifold
with compact (possibly empty) boundary is inward tame, then it has finitely
many ends, each of which has semistable fundamental group and stable homology
in all dimensions.
\end{theorem}

\begin{proof}
[Sketch of proof]Finite-endedness of inward tame manifolds with compact
boundary was obtained in Exercise
\ref{Exercise: inward tame implies finite-ended}. The $\pi_{1}$-semistablity
of each end is based on the transversality strategy described in Theorem
\ref{Theorem: Finite type x R}. Stability of the homology groups is similar,
but algebraic tools like duality are also needed.
\end{proof}

Siebenmann's proof of Theorem \ref{Theorem: Siebenmann's thesis} (as outlined
earlier),\ along with the strategy used by Chapman and Siebenmann in \cite{CS}
(to be discussed
\S \ref{Section: Existence and uniqueness of Z-compactifications}) make the
following approach seem all but inevitable: Define a manifold $N^{n}$ with
compact boundary to be a \emph{homotopy collar} if $\partial N^{n}%
\hookrightarrow N^{n}$ is a homotopy equivalence. A homotopy collar is called
a \emph{pseudo-collar} if it contains arbitrarily small homotopy collar
neighborhoods of infinity. A manifold that contains a pseudo-collar
neighborhood of infinity is called \emph{pseudo-collarable}.

Clearly, every collarable manifold is pseudo-collarable, but the Davis
manifolds are counterexamples to the converse (see Example
\ref{Example: Davis manifolds are pseudo-collarable}). Before turning our
attention to a pseudo-collarability characterization, modeled after Theorem
\ref{Theorem: Siebenmann's thesis}, we spend some time getting familiar with
pseudo-collars and their properties.

A cobordism $\left(  A,\partial_{-}A,\partial_{+}A\right)  $ is called a
\emph{one-sided h-cobordism} if $\partial_{-}A\hookrightarrow A$ is a homotopy
equivalence, but not necessarily so for $\partial_{+}A\hookrightarrow A$. The
key connection between these concepts is contained in Proposition
\ref{Proposition: pseudo-collar properties}. First we state a standard lemma.

\begin{lemma}
\label{Lemma: properties of one-sided h-cobordisms}Let $\left(  A,\partial
_{-}A,\partial_{+}A\right)  $ be a compact one-sided h-cobordism as described
above. Then the inclusion $\partial_{+}A\hookrightarrow A$ induces $%
\mathbb{Z}
$-homology isomorphisms (in fact, $%
\mathbb{Z}
\lbrack\pi_{1}\left(  A\right)  ]$-homology isomorphisms) in all dimensions;
in addition, $\pi_{1}\left(  \partial_{+}A\right)  \rightarrow\pi_{1}\left(
A\right)  $ is surjective with perfect kernel.
\end{lemma}

Lemma \ref{Lemma: properties of one-sided h-cobordisms} is obtained from
various forms of duality. For details, see \cite[Th. 2.5]{GT1}.

\begin{proposition}
[Structure of manifold pseudo-collars]%
\label{Proposition: pseudo-collar properties}Let $N^{n}$ be a pseudo-collar. Then

\begin{enumerate}
\item $N^{n}$ can be expressed as a union $A_{0}\cup A_{1}\cup A_{2}\cup
\cdots$ of one-sided h-cobordisms with $\partial_{-}A_{0}=\partial N$ and
$\partial_{+}A_{i}=\partial_{-}A_{i+1}=A_{i}\cap A_{i+1}$ for all $i\geq0$,

\item $N^{n}$ contains arbitrarily small pseudo-collar neighborhoods of infinity,

\item $N^{n}$ is absolutely inward tame,

\item $\operatorname*{pro}$-$H_{i}\left(  \varepsilon\left(  N^{n}\right)  ;%
\mathbb{Z}
\right)  $ is stable for all $i$,

\item $\operatorname*{pro}$-$\pi_{1}\left(  \varepsilon\left(  N^{n}\right)
\right)  $ may be represented by a sequence $G_{0}\overset{\mu_{1}%
}{\twoheadleftarrow}G_{1}\overset{\mu_{2}}{\twoheadleftarrow}G_{2}%
\overset{\mu_{3}}{\twoheadleftarrow}\cdots$ of surjections, where each $G_{i}$
is finitely presentable and each $\ker(\mu_{i})$ is perfect, and

\item there exists a proper map $\phi:N^{n}\rightarrow\lbrack0,\infty)$ with
$\phi^{-1}\left(  0\right)  =\partial N^{n}$ and $\phi^{-1}\left(  r\right)  $
a closed $\left(  n-1\right)  $-manifold with the same $%
\mathbb{Z}
$-homology as $\partial N^{n}$ for all $r$.
\end{enumerate}

\begin{proof}
Observations (a)-(c) are almost immediate, after which (d) and (e) can be
obtained by straightforward applications of Lemma
\ref{Lemma: properties of one-sided h-cobordisms}. Item (f) can be obtained by
applying the (highly nontrivial) main result from \cite{DT} to each cobordism
$\left(  A_{i},\partial_{-}A_{i},\partial_{+}A_{i}\right)  $.
\end{proof}

\begin{exercise}
Fill in the necessary details for observations (a)-(e).
\end{exercise}
\end{proposition}

Some examples are now in order.

\begin{example}
[The Whitehead manifold is not pseudo-collarable]First notice that
$\mathcal{W}^{3}$ does contain a homotopy collar neighborhood of infinity. Let
$D^{3}$ be a tame ball in $\mathcal{W}^{3}$ and let $N=\mathcal{W}%
^{3}-\operatorname*{int}D^{3}$. By excision and the Hurewicz and Whitehead
theorems, $N$ is a homotopy collar. (This argument works for all contractible
open manifolds.) But since $\mathcal{W}^{3}$ is neither inward tame nor
semistable, Proposition \ref{Proposition: pseudo-collar properties} assures
that $\mathcal{W}^{3}$ is not pseudo-collarable.
\end{example}

\begin{example}
[Davis manifolds are pseudo-collarable]%
\label{Example: Davis manifolds are pseudo-collarable}Non-collarable but
pseudo-collarable ends are found in some of our most important examples---the
Davis manifolds. It is easy to see that the neighborhood of infinity $N_{0}$
shown in Figure \ref{Figure: Exhaustion by compact contractibles} is a
homotopy collar, as is $N_{i}$ for each $i>0$.
\end{example}

Motivated by Proposition \ref{Proposition: pseudo-collar properties} and
previous definitions, call an inverse sequence of groups \emph{perfectly
semistable }if it is pro-isomorphic to an inverse sequence of finitely
presentable groups for which the bonding homomorphisms are all surjective with
perfect kernels. A complete characterization of pseudo-collarable
$n$-manifolds is provided by:

\begin{theorem}
[\cite{GT2}]\label{Theorem: characterization of pseudocollarable}A 1-ended
n-manifold $W^{n}$ ($n\geq6$) with compact (possibly empty) boundary is
pseudo-collarable if and only if

\begin{enumerate}
\item $W^{n}$ is inward tame,

\item $\operatorname*{pro}$-$\pi_{1}\left(  \varepsilon\left(  W^{n}\right)
\right)  $ is perfectly semistable, and

\item $\sigma_{\infty}\left(  W^{n}\right)  \in\underleftarrow{\lim}\left\{
\widetilde{K}_{0}(%
\mathbb{Z}
\left[  \pi_{1}\left(  N,p_{i}\right)  \right]  )\mid N\text{ a clean nbd. of
infinity}\right\}  $ is trivial.
\end{enumerate}
\end{theorem}

\begin{remark}
\emph{In (c), }$\sigma_{\infty}\left(  W^{n}\right)  $ \emph{may be defined
as} $\left(  \sigma\left(  N_{0}\right)  ,\sigma\left(  N_{1}\right)
,\sigma\left(  N_{2}\right)  ,\cdots\right)  $\emph{, the sequence of Wall
finiteness obstructions of an arbitrary nested cofinal sequence of clean
neighborhoods of infinity. By the functoriality of }$\widetilde{K}_{0}$\emph{,
this obstruction may be viewed as an element of the indicated inverse limit
group.} \emph{It is trivial if and only if each coordinate is trivial, i.e.,
each }$N_{i}$\emph{ has finite homotopy type. So just as in Theorem
\ref{Theorem: Siebenmann's thesis}, Conditions (a) and (c) together are
equivalent to }$W^{n}$\emph{ being absolutely inward tame.}
\end{remark}

By Theorem \ref{Theorem: tame implies semistability}, every inward tame open
manifold $W^{n}$ has semistable $\operatorname*{pro}$-$\pi_{1}$ and stable
$\operatorname*{pro}$-$H_{1}$. Together those observations guarantee a
representation of $\operatorname*{pro}$-$\pi_{1}\left(  \varepsilon
(W^{n})\right)  $ by an inverse sequence of surjective homomorphisms of
finitely presented groups with \textquotedblleft nearly
perfect\textquotedblright\ kernels (in a way made precise in \cite{GT3}). One
might hope that Condition (b) of Theorem
\ref{Theorem: characterization of pseudocollarable} is extraneous, but an
example constructed in \cite{GT1} dashes that hope.

\begin{theorem}
\label{Theorem: existence of nonpseudocollarable manifolds}In all dimensions
$\geq6$ there exist absolutely inward tame open manifolds that are not pseudo-collarable.
\end{theorem}

In light of Theorem \ref{Theorem: characterization of pseudocollarable}, it is
not surprising that Theorem
\ref{Theorem: existence of nonpseudocollarable manifolds} uses a significant
dose of group theory. In fact, unravelling the group theory at infinity seems
to be the key to understanding ends of inward tame manifolds. That topic is
the focus of ongoing work \cite{GT3}. As for our favorite open manifolds, the
following is wide-open.

\begin{Question}
Is the universal cover $\widetilde{M}^{n}$of a closed aspherical $n$-manifold
always pseudocollarable? Must it satisfy some of the hypotheses of Theorem
\ref{Theorem: characterization of pseudocollarable}? In particular, is
$\widetilde{M}^{n}$ always inward tame? (If so, an affirmative answer to
Conjecture \ref{Question: semistability for manifolds-version 2} would follow
from Theorem \ref{Theorem: tame implies semistability}.)
\end{Question}

We close this section with a reminder that the above results rely heavily on
manifold-specific tools. For general locally finite complexes, Proposition
\ref{Prop: mapping telescope realization fro pro-p1} serves as warning. Even
so, many ideas and questions discussed here have interesting analogs outside
manifold topology---in the field of geometric group theory. We now take a
break from manifold topology to explore that area.

\section{End invariants applied to group
theory\label{Section: Exploring the ends of groups}}

A standard method for applying topology to group theory is via
Eilenberg-MacLane spaces. For a group $G$, a $K\!\left(  G,1\right)  $
\emph{complex} (or \emph{Eilenberg-MacLane complex }for $G$ or a
\emph{classifying space} for $G$) is an aspherical CW complex with fundamental
group isomorphic to $G$. When the language of classifying spaces is used, a
$K\!\left(  G,1\right)  $ complex is often referred to as a $BG$ complex and
its universal cover as an $EG$ complex. Alternatively, an $EG$ complex is a
contractible CW complex on which $G$ acts properly and freely.

\begin{exercise}
\label{Exercise: aspherical = contractible universal cover}Show that a CW
complex $X$ is aspherical if and only if $\widetilde{X}$ is contractible.
\end{exercise}

It is a standard fact that, for every group $G$: \textbf{(a)} there exists a
$K\!\left(  G,1\right)  $ complex, and \textbf{(b)} any two $K\!\left(
G,1\right)  $ complexes are homotopy equivalent. Therefore, any homotopy
invariant of a $K\!\left(  G,1\right)  $ complex is an invariant of $G$. In
that way we define the \emph{(co)homology of }$G$\emph{ with constant
coefficients} in a ring $R$, denoted $H_{\ast}\left(  G;R\right)  $ and
$H^{\ast}\left(  G;R\right)  $, to be $H_{\ast}\left(  K\!\left(  G,1\right)
;R\right)  $ and $H^{\ast}\left(  K\!\left(  G,1\right)  ;R\right)  $, respectively.

At times it is useful to relax the requirement that a $BG$ or an $EG$ be a CW
complex. For example, an aspherical manifold or a locally CAT(0) space with
fundamental group $G$, but with no obvious cell structure might be a used as a
$BG$. Provided the space in question is an ANR, there is no harm in allowing
it, since all of the key facts from algebraic topology (for example, Exercise
\ref{Exercise: aspherical = contractible universal cover}) still apply.
Moreover, by Proposition \ref{Proposition: ANR facts}, ANRs are homotopy
equivalent to CW complexes, so, if necessary, an appropriate complex can be obtained.

\subsection{Groups of type $\mathbf{F}\label{Subsection: groups of type F}$}

We say that $G$ \emph{has} \emph{type }$F$ if $K\!\left(  G,1\right)  $
complexes have finite homotopy type or, equivalently, there exits a finite
$K\!\left(  G,1\right)  $ complex or a compact ANR $K\!\left(  G,1\right)  $
space. Note that if $K$ is a finite $K\!\left(  G,1\right)  $ complex, then
$\widetilde{K}$ is locally finite and the $G$-action is cocompact; then we
call $\widetilde{K}$ a cocompact $EG$ complex.

\begin{example}
All finitely generated free and free abelian groups have type $F$, as do the
fundamental groups of all closed surfaces, except for $%
\mathbb{R}
P^{2}$. In fact, the fundamental group of every closed aspherical manifold has
type $F$. No group that contains torsion can have type $F$ (see \cite[Prop.
7.2.12]{Ge2}), but every torsion-free $CAT\left(  0\right)  $ or $\delta
$-hyperbolic group has type $F$.
\end{example}

For groups of type $F$, there is an immediate connection between group theory
and topology at the ends of noncompact spaces. If $G$ is nontrivial and
$K_{G}$ is a finite $K\!\left(  G,1\right)  $ complex, $\widetilde{K}_{G}$ is
contractible, locally finite, and noncompact, and by Corollary
\ref{Corollary: proper homotopy equivalence of covering spaces}, all other
finite $K\!\left(  G,1\right)  $ complexes (or compact ANR classifying spaces)
have universal covers proper homotopy equivalent to $\widetilde{K}_{G}$. So
the end invariants of $\widetilde{K}_{G}$, which are well-defined up to proper
homotopy equivalence, may be attributed directly to $G$. For example, one may
discuss: the \emph{number of ends}\ of $G$; the \emph{homology and cohomology
at infinity} of $G$ (denoted by $\operatorname*{pro}$-$H_{\ast}(\varepsilon
\left(  G);R\right)  $, $\check{H}_{\ast}\left(  \varepsilon\left(  G\right)
;R\right)  $ and $\check{H}^{\ast}\left(  \varepsilon\left(  G\right)
;R\right)  $); and the homotopy behavior of the end(s) of $G$---properties
such as simple connectedness, stability, semistability, or pro-monomorpic at
infinity. In cases where $\widetilde{K}_{G}$ is 1-ended and semistable,
$\operatorname*{pro}$-$\pi_{\ast}\left(  \varepsilon\left(  G\right)  \right)
$ and $\check{\pi}_{\ast}\left(  \varepsilon\left(  G\right)  \right)  $ are
defined similarly. The need for semistability is, of course, due to base ray
issues. Although $\widetilde{K}_{G}$ is well-defined up to proper homotopy
type, there is no canonical choice base ray; in the presence of semistability
that issue goes away. We will return to that topic shortly.

\subsection{Groups of type $\mathbf{F}_{k}%
\label{Subsection: Groups of type F_k}$}

In fact, the existence of a finite $K\!\left(  G,1\right)  $ is excessive for
defining end invariants like $\operatorname*{pro}$-$H_{\ast}\left(
\varepsilon\left(  G);R\right)  \right)  $ and $\check{H}_{\ast}\left(
\varepsilon\left(  G\right)  ;R\right)  $. If $G$ admits a $K\!\left(
G,1\right)  $ complex $K$ with a finite $k$-skeleton (in which case we say $G$
\emph{has type }$F_{k}$), then all $j$-dimensional homology and homotopy end
properties of the (locally finite) $k$-skeleton $\widetilde{K}_{G}^{(k)}$ of
$\widetilde{K}_{G}$ can be directly attributed to $G$, provided $j<k$. The
proof of invariance is rather intuitive. If $L$ is any other $K\!\left(
G,1\right)  $ with finite $k$-skeleton, choose a cellular homotopy equivalence
$f:K\rightarrow L$ and a homotopy inverse $g:L\rightarrow K$. These lift to
homotopy equivalences $\tilde{f}:\widetilde{K}\rightarrow\widetilde{L}$ and
$\tilde{g}:\widetilde{L}\rightarrow\widetilde{K}$, which cannot be expected to
be proper. Nevertheless, the restrictions of $\tilde{g}\circ\tilde{f}$ and
$\tilde{f}\circ\tilde{g}$ to the $\left(  k-1\right)  $-skeletons of
$\widetilde{K}$ and $\widetilde{L}$ can be proven properly homotopic to
inclusions $\widetilde{K}^{\left(  k-1\right)  }\hookrightarrow\widetilde{K}%
^{\left(  k\right)  }$ and $\widetilde{L}^{\left(  k-1\right)  }%
\hookrightarrow\widetilde{L}^{\left(  k\right)  }$. This is enough for the
desired result.

As another example of the above, the number of ends, viewed as (the
cardinality of) $\check{\pi}_{0}\left(  \widetilde{K}_{G}^{(1)}\right)  $,\ is
a well-defined invariant of a finitely generated group, i.e., group of type
$F_{1}$.

\begin{exercise}
Alternatively, one may define the number of ends of a finitely generated $G$
to be the number of ends of a corresponding Cayley graph. Explain why this
definition is equivalent to the above.
\end{exercise}

\begin{remark}
\emph{There are key connections between} $\operatorname*{pro}$-$H_{\ast
}\left(  \varepsilon\left(  G);R\right)  \right)  $ \emph{and} $\check
{H}_{\ast}\left(  \varepsilon\left(  G\right)  ;R\right)  $ \emph{and the
cohomology of} $G$ \emph{with} $RG$ \emph{coefficients (as presented, for
example, in \cite{Br}). We have chosen not to delve into that topic in these
notes. The interested reader is encouraged to read Chapters 8 and 13 of
\cite{Ge2}.}
\end{remark}

\subsection{Ends of Groups}

In view of earlier comments, the following iconic result may be viewed as an
application of $\check{\pi}_{0}\left(  \varepsilon\left(  G\right)  \right)  $.

\begin{theorem}
[Freudenthal-Hopf-Stallings]\label{Theorem: Ends theorem for groups}Every
finitely generated group $G$ has 0,1,2, or infinitely many ends. Moreover

\begin{enumerate}
\item $G$ is $0$-ended if and only if it is finite,

\item $G$ is $2$-ended if and only if it contains an infinite cyclic group of
finite index, and

\item $G$ is infinite-ended if and only if

\begin{itemize}
\item $G=A\ast_{C}B$ (a free product with amalgamation), where $C$ is finite
and has index $\geq2$ in both $A$ and $B$ with at least one index being
$\geq3$, or

\item $G=A\ast_{\phi}$ (an HNN extension\footnote{Definitions of \emph{free
product with amalgamation }\ and \emph{HNN\ extension} can be found in
\cite{SW}, \cite{Ge2}, or any text on combinatorial group theory.}), where
$\phi$ is an isomorphism between finite subgroups of $A$ each having index
$\geq2$.\medskip
\end{itemize}
\end{enumerate}
\end{theorem}

\begin{proof}
[Proof (small portions)]The opening line of Theorem
\ref{Theorem: Ends theorem for groups} is essentially Exercise
\ref{Exercise: Ends of spaces admitting actions}; item (a) is trivial and item
(b) is a challenging exercise. Item (c) is substantial \cite{St2}, but
pleasantly topological. Complete treatments can be found in \cite{SW} or
\cite{Ge2}.
\end{proof}

\subsection{The Semistability Conjectures}

If $G$ is finitely presentable, i.e., $G$ has type $F_{2}$, and $K$ is a
corresponding presentation $2$-complex (or any finite 2-complex with
fundamental group $G$), then $K$ may be realized as the 2-skeleton of a
$K\!\left(  G,1\right)  $. That is accomplished by attaching 3-cells to $K$ to
kill $\pi_{2}\left(  K\right)  $ and proceeding inductively, attaching
$\left(  k+1\right)  $-cells to kill the $k^{\text{th}}$ homotopy group, for
all $k\geq3$. It follows that $\operatorname*{pro}$-$H_{1}\left(
\varepsilon\left(  \widetilde{K}\right)  ;R\right)  $ and $\check{H}%
_{1}\left(  \varepsilon\left(  \widetilde{K}\right)  ;R\right)  $ represent
the group invariants $\operatorname*{pro}$-$H_{1}\left(  \varepsilon\left(
G\right)  ;R\right)  $ and $\check{H}_{1}\left(  \varepsilon\left(  G\right)
;R\right)  $, as discussed in \S \ref{Subsection: Groups of type F_k}. And by
the same approach used there, when $G$ (in other words $\widetilde{K}$) is
1-ended, properties such as simple connectivity at infinity, stability,
semistability and pro-monomorphic at infinity can be measured in
$\widetilde{K}$ and attributed directly to $G$. In an effort to go further
with homotopy properties of the end of $G$, we are inexorably led back to the
open problem:

\begin{conjecture}
[Semistability Conjecture--with explanation]%
\label{Conjecture: Semistability-detailed version}Every 1-ended finitely
presented group $G$ is semistable. In other words, the universal cover
$\widetilde{K}$ of every finite complex with fundamental group $G$ is strongly
connected at infinity; equivalently, $\operatorname*{pro}$-$\pi_{1}\left(
\widetilde{K},r\right)  $ is semistable for some (hence all) proper rays $r$.
\end{conjecture}

The fundamental nature of the Semistability Conjecture is now be clear. We
would like to view $\operatorname*{pro}$-$\pi_{1}\left(  \varepsilon\left(
\widetilde{K}\right)  ;r\right)  $ and $\check{\pi}_{1}\left(  \varepsilon
\left(  \widetilde{K}\right)  ;r\right)  $ as group invariants
$\operatorname*{pro}$-$\pi_{1}\left(  \varepsilon\left(  G\right)  \right)  $
and $\check{\pi}_{1}\left(  \varepsilon\left(  G\right)  \right)  $.
Unfortunately, there is the potential for these to depend on base rays. A
positive resolution of the Semistability Conjecture would eliminate that
complication once and for all. The same applies to $\operatorname*{pro}$%
-$\pi_{j}\left(  \varepsilon\left(  \widetilde{K}\right)  ;r\right)  $and
$\check{\pi}_{j}\left(  \varepsilon\left(  \widetilde{K}\right)  ;r\right)  $
when $G$ is of type F$_{k}$ and $j<k$.

The extension of Conjecture \ref{Conjecture: Semistability-detailed version}
to groups with arbitrarily many ends makes sense---the conjecture is that
$\widetilde{K}$ is semistable (defined for multi-ended spaces near the end of
\S \ref{Subsection: topological interpretations}). But this situation is
simpler than one might expect: for 0-ended groups there is nothing to discuss,
and 2-ended groups are known to be simply connected at each end (see Exercise
\ref{Exercise: 2-ended groups simply connected at ends} below); moreover,
Mihalik \cite{Mi87} has shown that an affirmative answer for $1$-ended groups
would imply an affirmative answer for all infinite-ended groups.

\begin{exercise}
\label{Exercise: 2-ended groups simply connected at ends}Let $G$ be a group of
type $F_{k}$. Show that every finite index subgroup $H$ is of type F$_{k}$ and
the two groups share the same end invariants through dimension $k-1$. Use
Theorem \ref{Theorem: Ends theorem for groups} to conclude that every
$2$-ended group is simply connected at each end.
\end{exercise}

Evidence for the Semistability Conjecture is provided by a wide variety of
special cases; here is a sampling.

\begin{theorem}
\label{Theorem: Some groups that are semistable}A finitely presented group
satisfying any one of the following is semistable.

\begin{enumerate}
\item $G$ is the extension of an infinite group by an infinite group,

\item $G$ is a one-relator group,

\item $G=A\ast_{C}B$ where $A$ and $B$ are finitely presented and semistable
and $C$ is infinite,

\item $G=A\ast_{C}$ where $A$ is finitely generated and semistable and $C$ is infinite,

\item $G$ is $\delta$-hyperbolic,

\item $G$ is a Coxeter group,

\item $G$ is an Artin group.
\end{enumerate}
\end{theorem}

\noindent References include: \cite{Mi83}, \cite{MiTs}, \cite{MiTs2},
\cite{Swa}, and \cite{Mi96}.\medskip

There is a variation on the Semistability Conjecture that is also open.

\begin{conjecture}
[$H_{1}$-semistability Conjecture]For every 1-ended finitely presented group
$G$, $\operatorname*{pro}$-$H_{1}\left(  \varepsilon(G);%
\mathbb{Z}
\right)  $ is semistable.
\end{conjecture}

Since $\operatorname*{pro}$-$H_{1}\left(  \varepsilon\left(  G\right)  ;%
\mathbb{Z}
\right)  $ can be obtained by abelianization of any representative of
$\operatorname*{pro}$-$\pi_{1}\left(  \varepsilon\left(  \widetilde{K}\right)
,r\right)  $, for any presentation 2-complex $K$ and base ray $r$, it is clear
that the $H_{1}$-semistability Conjecture is weaker than the Semistability
Conjecture. Moreover, the $H_{1}$-version of our favorite special case of the
Semistability Conjecture---the case where $G$ is the fundamental group of an
aspherical manifold---is easily solved in the affirmative, by an application
of Exercise
\ref{Exercise: contractible open n-manifolds have pro-homology of R^n}. This
provides a ray of hope that the Manifold Semistability Conjecture is more
accessible that the general case.

\begin{remark}
\emph{The Semistability Conjectures presented in this section were initially
formulated by Ross Geoghegan in 1979. At the time, he simply called them
\textquotedblleft questions\textquotedblright, expecting the answers to be
negative. Their long-lasting resistance to solutions, combined with an
accumulation of affirmative answers to special cases, has gradually led them
to become known as conjectures.}
\end{remark}

\section{Shape Theory}

Shape theory may be viewed as a method for studying bad spaces using tools
created for the study of good spaces. Although more general approaches exist,
we follow the classical (and the most intuitive) route by developing shape
theory only for compacta. But now we are interested in arbitrary
compacta---not just ANRs. A few examples to be considered are shown in Figure
\ref{Figure: Shapes}.%
\begin{figure}[ptb]%
\centering
\includegraphics[
height=3.5404in,
width=4.755in
]%
{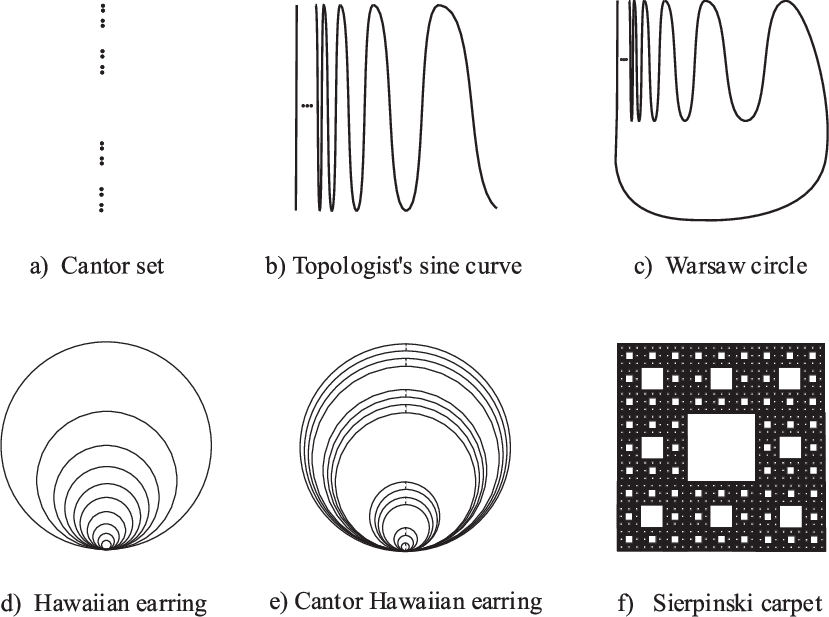}%
\caption{\ }%
\label{Figure: Shapes}%
\end{figure}

The abrupt shift from noncompact spaces with nice local properties to compacta
with bad local properties may seem odd, but there are good reasons for this
temporary shift in focus. First, the tools we have already developed for
analyzing the ends of manifolds and complexes are nearly identical to those
used in shape theory; understanding and appreciating the basics of shape
theory will now be quite easy. More importantly, certain aspects of the study
of ends are nearly impossible without shapes---if the theory did not already
exist, we would be forced to invent it.

For more comprehensive treatments of shape theory, the reader can consult
\cite{Bor2} or \cite{DySe}.

\subsection{Associated sequences, basic definitions, and
examples\label{Subsection: basic definitions of shape theory}}

In shape theory, the first step in studying a compactum $A$ is to choose an
\emph{associated inverse sequence} $K_{0}\overset{f_{1}}{\longleftarrow}%
K_{1}\overset{f_{2}}{\longleftarrow}K_{2}\overset{f_{3}}{\longleftarrow}%
\cdots$ of finite polyhedra and simplicial maps. There are several ways this
can be done. We describe a few of them.\medskip

\noindent\textbf{Method 1:} If $A$ is finite-dimensional, choose an embedding
$A\hookrightarrow%
\mathbb{R}
^{n}$, and let $K_{0}\supseteq K_{1}\supseteq K_{2}\supseteq\cdots$ be a
sequence of compact polyhedral neighborhoods intersecting in $A$. Since it is
impossible to choose triangulations under which all inclusion maps are
simplicial, choose progressively finer triangulations for the $K_{i}$ and let
the $f_{i}$ be simplicial approximations to the inclusion maps.\medskip

\noindent\textbf{Method 2:} Choose a sequence $\mathcal{U}_{0},\mathcal{U}%
_{1},\mathcal{U}_{2},\cdots$ of finite covers of $A$ by $\varepsilon_{i}%
$-balls, where $\varepsilon_{i}\rightarrow0$ and each $\mathcal{U}_{i+1}$
refines $\mathcal{U}_{i}$. Let $K_{i}$ be the nerve of $\mathcal{U}_{i}$ and
$f_{i}:K_{i}\rightarrow K_{i-1}$ a simplicial map that takes each vertex
$U\in\mathcal{U}_{i}$ to a vertex $V\in\mathcal{U}_{i-1}$ with $U\subseteq
V$.\medskip

\noindent\textbf{Method 3:} If $A$ can be expressed as the inverse limit of an
inverse sequence $K_{0}\overset{g_{1}}{\longleftarrow}K_{1}\overset{g_{2}%
}{\longleftarrow}K_{2}\overset{g_{3}}{\longleftarrow}\cdots$ of finite
polyhedra\footnote{By definition, $\underleftarrow{\lim}\left\{  K_{i}%
,f_{i}\right\}  $ is viewed as a subspace of the infinite product space
$\prod_{i=0}^{\infty}K_{i}$ and is topologized accordingly.}, then that
sequence itself may be associated to $A$, after each map is approximated by
one that is simplicial.\medskip

\begin{remark}
\emph{(a) \ At times, it will be convenient if each }$K_{i}$\emph{ in an
associated inverse sequence has a preferred vertex }$p_{i}$\emph{ with each
}$f_{i+1}$\emph{ taking }$p_{i+1}$\emph{ to }$p_{i}$\emph{. That can easily be
arranged; we refer to the result as a} pointed inverse sequence.\smallskip

\noindent\emph{(b) }\textbf{\ }\emph{Our requirement that the bonding maps in
associated inverse sequences be simplicial, will soon be seen as unnecessary.
But, for now, there is no harm in including that additional niceness
condition.}\smallskip

\noindent\emph{(c) \ When }$A$\emph{ is infinite-dimensional, a variation on
Method 1 is available. In that case, }$A$\emph{ is embedded in the Hilbert
cube and a sequence }$\left\{  N_{i}\right\}  $\emph{ of closed Hilbert cube
manifold neighborhoods of }$A$\emph{ is chosen. By Theorem
\ref{Theorem: Triangulability of HCMs}, each }$N_{i}$\emph{ has the homotopy
type of a finite polyhedron }$K_{i}$\emph{. From there, an associated inverse
sequence for }$A$\emph{ is readily obtained.}
\end{remark}

The choice of an associated inverse sequence for a compactum $A$ should be
compared to the process of choosing a cofinal sequence of neighborhoods of
infinity for a noncompact space $X$. In both situations, the terms in the
sequences can be viewed as progressively better approximations to the object
of interest, and in both situations, there is tremendous leeway in assembling
those approximating sequences. In both contexts, that flexibility raises
well-definedness issues. In the study of ends, we introduced an equivalence
relation based on ladder diagrams to obtain the appropriate level of
well-definedness. The same is true in shape theory.

\begin{proposition}
\label{Proposition: Equivalence of associated sequences of polyhedra}For a
fixed compactum $A$, let $\left\{  K_{i},f_{i}\right\}  $ and $\left\{
L_{i},g_{i}\right\}  $ be a pair of associated inverse sequences of finite
polyhedra. Then there exist subsequences, simplicial maps, and a corresponding
ladder diagram%
\[
\begin{diagram} K_{i_{0}} & & \lTo^{f_{i_{0},i_{1}}} & & K_{i_{1}} & & \lTo^{f_{i_{1},i_{2}}} & & K_{i_{2}} & & \lTo^{f_{i_{2},i_{3}}}& & K_{i_{3}}& \cdots\\ & \luTo & & \ldTo & & \luTo & & \ldTo & & \luTo & & \ldTo &\\ & & L_{j_{0}} & & \lTo^{g_{j_{0},j_{1}}} & & L_{j_{1}} & & \lTo^{g_{j_{1},j_{2}}}& & L_{j_{2}} & & \lTo^{g_{j_{2},j_{3}}} & & \cdots \end{diagram}
\]
in which each triangle of maps homotopy commutes. If desired, we may require
that those homotopies preserve base points.
\end{proposition}

\begin{exercise}
Prove some or all of Proposition
\ref{Proposition: Equivalence of associated sequences of polyhedra}. Start by
comparing any pair of sequences obtained using the same method, then note that
Method 1 is a special case of Method 3.
\end{exercise}

Define a pair of inverse sequences of finite polyhedra and simplicial maps to
be $\operatorname*{pro}$\emph{-homotopy equivalent} if they contain
subsequences that fit into a homotopy commuting ladder diagram, as described
in Proposition
\ref{Proposition: Equivalence of associated sequences of polyhedra}. Compacta
$A$ and $A^{\prime}$ are \emph{shape equivalent }if some (and thus every) pair
of associated inverse sequences of finite polyhedra are pro-homotopy
equivalent. In that case we write $\mathcal{S}\emph{h\!}\left(  A\right)
=\mathcal{S}\emph{h\!}\left(  A^{\prime}\right)  $ or sometimes $A^{\prime}%
\in\mathcal{S}\emph{h\!}\left(  A\right)  $.

\begin{remark}
\emph{If }$\left\{  K_{i},f_{i}\right\}  $\emph{ is an associated inverse
sequence for a compactum }$A$\emph{, it is not necessarily the case that
}$\underleftarrow{\lim}\left\{  K_{i},f_{i}\right\}  \approx A$\emph{. But it
is immediate from the definitions that the two spaces have the same shape.}
\end{remark}

\begin{exercise}
Show that the Topologist's Sine Curve has the shape of a point and the Warsaw
Circle has the shape of a circle (see Figure \ref{Figure: Shapes}). Note that
neither space is homotopy equivalent to its nicer shape version.
\end{exercise}

\begin{exercise}
Show that the Whitehead Continuum (see Example
\ref{Example: Definition of Whitehead manifold}) has the shape of a point.
Spaces with the shape of a point are often called \textbf{cell-like}.
\end{exercise}

\begin{exercise}
Show that the Sierpinski Carpet is shape equivalent to a Hawaiian Earring.
\end{exercise}

\begin{exercise}
Show that the Cantor Hawaiian Earring is shape equivalent to a standard
Hawaiian Earring. (An observation that once prompted the reaction:
\textquotedblleft I demand a recount!\textquotedblright)
\end{exercise}

When considering the shape of a compactum $A$, the space $A$ itself becomes
largely irrelevant after an associated inverse sequence has been chosen. In a
sense, shape theory is just the study of pro-homotopy classes of inverse
sequences of finite polyhedra. Nevertheless, there is a strong correspondence
between inverse sequences of finite polyhedra and compact metric spaces
themselves. If $A$ is the inverse limit of an inverse sequence $\left\{
K_{i},f_{i}\right\}  $ of finite polyhedra, then applying any of the three
methods mentioned earlier to the space $A$ yields an inverse sequence of
finite polyhedra pro-homotopy equivalent to the original $\left\{  K_{i}%
,f_{i}\right\}  $. In other words, passage to an inverse limit preserves all
relevant information. As we saw in Exercise
\ref{Exercise: nontrivial inverse sequence with trivial inverse limit}, that
is not the case with inverse sequences of groups. This phenomenon is even more
striking when studying ends of spaces. If $N_{0}\hookleftarrow N_{1}%
\hookleftarrow N_{2}\hookleftarrow\cdots$ is a cofinal sequence of
neighborhoods of infinity of a space $X$, the inverse limit of that sequence
is clearly the empty set. In some sense, the study of ends is a study of an
imaginary \textquotedblleft space at infinity\textquotedblright. By using
shape theory, we can sometimes make that space a reality.

\begin{exercise}
Prove that an inverse sequence of nonempty finite polyhedra (or more
generally, an inverse sequence of nonempty compacta) is never the empty set.
\end{exercise}

\begin{exercise}
\label{Exercise: Shapes of finite polyhedra}So far, our discussion of shape
has focused on exotic compacta; but nice spaces, such as finite polyhedra, are
also part of the theory. Show that finite polyhedra $K$ and $L$ are shape
equivalent if and only if they are homotopy equivalent. \emph{Hint: }Choosing
trivial associated inverse sequences $K\overset{\operatorname*{id}%
}{\longleftarrow}K\overset{\operatorname*{id}}{\longleftarrow}\cdots$ and
$L\overset{\operatorname*{id}}{\longleftarrow}L\overset{\operatorname*{id}%
}{\longleftarrow}\cdots$ makes the task easier. A more general observation of
this sort will be made shortly.
\end{exercise}

\subsection{The algebraic shape invariants}

In the spirit of the work already done on ends of spaces, we define a variety
of algebraic invariants for compacta. Given a compactum $A$ and any associated
inverse sequence $\left\{  K_{i},f_{i}\right\}  $, define $\operatorname*{pro}%
$-$H_{\ast}\left(  A;R\right)  $ to be the pro-isomorphism class of the
inverse sequence
\[
H_{\ast}\left(  K_{0};R\right)  \overset{f_{1\ast}}{\longleftarrow}H_{\ast
}\left(  K_{1};R\right)  \overset{f_{2\ast}}{\longleftarrow}H_{\ast}\left(
K_{2};R\right)  \overset{f_{3\ast}}{\longleftarrow}\cdots
\]
and $\check{H}_{\ast}\left(  A;R\right)  $ to be its inverse limit. By
reversing arrows and taking a direct limit, we also define
$\operatorname*{ind}$-$H^{\ast}\left(  A;R\right)  $ and $\check{H}^{\ast
}\left(  A;R\right)  $. The groups $\check{H}_{\ast}\left(  A;R\right)  $ and
$\check{H}^{\ast}\left(  A;R\right)  $ are know as the \emph{\v{C}ech homology
and cohomology groups} of $A$, respectively. If we begin with a pointed
inverse sequence $\left\{  \left(  K_{i},p_{i}\right)  ,f_{i}\right\}  $ we
obtain $\operatorname*{pro}$-$\pi_{\ast}\left(  A,p\right)  $ and $\check{\pi
}_{\ast}\left(  A,p\right)  $, where $p$ corresponds to $\left(  p_{0}%
,p_{1},p_{2},\cdots\right)  $. Call $\check{\pi}_{\ast}\left(  A,p\right)  $
the \emph{\v{C}ech homotopy groups} of $A$, or sometimes, the \emph{shape
groups }of $A.$

\v{C}ech cohomology is known to be better-behaved than \v{C}ech homology, in
that there is a full-blown \v{C}ech cohomology theory satisfying the
Eilenberg-Steenrod axioms. Although the \v{C}ech homology groups of $A$ do not
fit into such a nice theory, they are are still perfectly good topological
invariants of $A$. For reasons we have seen before, $\operatorname*{pro}%
$-$H_{\ast}\left(  A;R\right)  $ and $\operatorname*{pro}$-$\pi_{\ast}\left(
A,p\right)  $ tend to carry more information than the corresponding inverse limits.

\begin{exercise}
\label{Exercise: Cech homology of the Warsaw circle}Observe that, for the
Warsaw circle $W$, the first \v{C}ech homology and the first \v{C}ech homotopy
group are not the same as the first singular homology and traditional
fundamental group of $W$.
\end{exercise}

\begin{remark}
\emph{Another way to think about the phenomena that occur in Exercise
\ref{Exercise: Cech homology of the Warsaw circle} is that, for an inverse
sequence of spaces} $K_{0}\overset{g_{1}}{\longleftarrow}K_{1}\overset{g_{2}%
}{\longleftarrow}K_{2}\overset{g_{3}}{\longleftarrow}\cdots$\emph{, the
homology [homotopy] of the inverse limit is not necessarily the same as the
inverse limit of the homologies [homotopies]. It is the point of view of shape
theory that the latter inverse limits often do a better job of capturing the
true nature of the space.}
\end{remark}

\subsection{Relaxing the definitions}

Now that the framework for shape theory is in place, we make a few adjustments
to the definitions. These changes will not nullify anything done so far, but
at times they will make the application of shape theory significantly easier.

Previously we required bonding maps in associated inverse sequences to be
simplicial. That has some advantages; for example, $\operatorname*{pro}%
$-$H_{\ast}\left(  A;R\right)  $ and $\check{H}^{\ast}\left(  A;R\right)  $
can be defined using only simplicial homology. But in light of the definition
of pro-homotopy equivalence, it is clear that only the homotopy classes of the
bonding maps really matters. So, adjusting a naturally occurring bonding map
to make it simplicial is unnecessary. From now on, we only require bonding
maps to be continuous. In a similar vein, a finite polyhedron $K_{i\text{ }}%
$in an inverse sequence corresponding to $A$ can easily be replaced by a
finite CW complex. More generally, any compact ANR is acceptable as an entry
in that inverse sequence (Proposition \ref{Proposition: ANR facts} is relevant
here). Of course, once these changes are made, we must use cellular or
singular (as opposed to simplicial) homology for defining the algebraic shape
invariants of the previous section.

With the above relaxation of definitions in place, the following fundamental
facts becomes elementary.

\begin{proposition}
\label{Prop: For compact ANRs shape = homotopy}Let $A$ and $B$ be compact
ANRs. Then $\mathcal{S}\emph{h\!}\left(  A\right)  =\mathcal{S}\emph{h\!}%
\left(  B\right)  $ if and only if $A\simeq B$.
\end{proposition}

\begin{proof}
An argument like that used in Exercise
\ref{Exercise: Shapes of finite polyhedra} can now be applied here.
\end{proof}

\begin{proposition}
\label{Prop: For a compact ANR shape invariants are homotopy invariants}If $A$
is a compact ANR, then $\operatorname*{pro}$-$H_{\ast}\left(  A;R\right)  $
and $\operatorname*{pro}$-$\pi_{\ast}\left(  A,p\right)  $ are stable for all
$\ast$ with $\check{H}_{\ast}\left(  A;R\right)  $ and $\check{\pi}_{\ast
}\left(  A,p\right)  $ being isomorphic to the singular homology groups
$H_{\ast}\left(  A;R\right)  $\ and the traditional homotopy groups $\pi
_{\ast}\left(  A,p\right)  $, respectively.
\end{proposition}

\begin{proof}
Choose the trivial associated inverse sequence $A\overset{\operatorname*{id}%
}{\longleftarrow}A\overset{\operatorname*{id}}{\longleftarrow}\cdots$.
\end{proof}

\begin{corollary}
If $B$ is a compactum that is shape equivalent to a compact ANR $A$, then
$\operatorname*{pro}$-$H_{\ast}\left(  B;R\right)  $ and $\operatorname*{pro}%
$-$\pi_{\ast}\left(  B,p\right)  $ are stable for all $\ast$ with $\check
{H}_{\ast}\left(  B;R\right)  \cong H_{\ast}\left(  B;R\right)  $ and
$\check{\pi}_{\ast}\left(  B,p\right)  \cong\pi_{\ast}\left(  A,p\right)  $.
In particular, $\check{H}_{\ast}\left(  B;R\right)  $ is finitely generated,
for all $\ast$ and $\check{\pi}_{1}\left(  B,p\right)  $ is finitely presentable.
\end{corollary}

\begin{example}
Compacta a), d), e), and f) from Figure \ref{Figure: Shapes} do not have the
shapes of compact ANRs.\medskip
\end{example}

Taken together, Propositions \ref{Prop: For compact ANRs shape = homotopy} and
\ref{Prop: For a compact ANR shape invariants are homotopy invariants} form
the foundation of the true slogan: \emph{When restricted to compact ANRs,
shape theory reduces to (traditional) homotopy theory.} Making that slogan a
bona fide theorem would require a development of the notion of
\textquotedblleft shape morphism\textquotedblright\ and a comparison of those
morphisms to homotopy classes of maps. We have opted against providing that
level of detail in these notes. We will, however, close this section with a
few comments aimed at giving the reader a feel for how that can be done.

Let \emph{pro-}$\mathcal{H}\emph{omotopy}$\emph{ }denote the set of all
pro-homotopy classes of inverse sequences of compact ANRs and continuous maps.
If $\mathcal{S}$\emph{hapes} denotes the set of all shape classes of compact
metric spaces, then there is a natural bijection $\Theta:\,$\emph{pro-}%
$\mathcal{H}\emph{omotopy}\rightarrow\mathcal{S}$\emph{hapes} defined by
taking inverse limits; Methods 1-3 in
\S \ref{Subsection: basic definitions of shape theory} determine $\Theta^{-1}%
$. With some additional work, one can define morphisms in \emph{pro-}%
$\mathcal{H}\emph{omotopy}$\emph{ }as certain equivalence classes of sequences
of maps, thereby promoting \emph{pro-}$\mathcal{H}\emph{omotopy}$\emph{ }to a
full-fledged category. From there, one can use $\Theta$ to (indirectly) define
morphisms in $\mathcal{S}$\emph{hapes}, thereby obtaining the shape category.
In that case, it can be shown that each continuous function $f:A\rightarrow B$
between compacta determines a unique shape morphism (a fact that uses some ANR
theory); but unfortunately, not every shape morphism from $A$ to $B$ can be
realized by a continuous map. This is not as surprising as it first appears:
as an example, the reader should attempt to construct a map from
$\mathbb{S}^{1}$ to the Warsaw circle that deserves to be called a shape equivalence.

\begin{remark}
\emph{In order to present a thorough development of the} pro-Homotopy
\emph{and} $\mathcal{S}$hapes \emph{categories, more care would be required in
dealing with base points. In fact, we would end up building a pair of slightly
different categories for each---one incorporating base points and the other
without base points. The differences between those categories does not show up
at the level of objects (for example, compacta are shape equivalent if and
only if they are \textquotedblleft pointed shape equivalent\textquotedblright%
), but the categories differ in their morphisms. In the context of these
notes, we need not be concerned with that distinction.}
\end{remark}

\subsection{The shape of the end of an inward tame
space\label{Subsection: Shape of the end of a space}}

The relationship between shape theory and the topology of the ends of
noncompact spaces goes beyond a similarity between the tools used in their
studies. In this section we develop a precise relationship between shapes of
compacta and ends of inward tame ANRs. In so doing, the fundamental nature of
inward tameness is brought into focus.

Let $Y$ be a inward tame ANR. By repeated application of the definition of
inward tameness, there exist sequences of neighborhoods of infinity $\left\{
N_{i}\right\}  _{i=0}^{\infty}$, finite complexes $\left\{  K_{i}\right\}
_{i=1}^{\infty}$, and maps $f_{i}:N_{i}\rightarrow K_{i}$ and $g_{i}%
:K_{i}\rightarrow N_{i-1}$ with $g_{i}f_{i}\simeq\operatorname*{incl}\left(
N_{i}\hookrightarrow N_{i-1}\right)  $ for all $i$. By letting $h_{i}%
=f_{i-1}g_{i}$, these can be assembled into a homotopy commuting ladder
diagram
\[
\begin{diagram} N_{0} & & \lInto & & N_{1} & & \lInto & & N_{2} & & \lInto & & N_{3} & \cdots\\ & \luTo^{g_{1}}  & & \ldTo^{f_{1}} & & \luTo^{g_{2}}   & & \ldTo^{f_{2}} & & \luTo^{g_{3}}   & & \ldTo^{f_{3}} &\\ & & K_{1} & & \lTo^{h_{2}} & & K_{2} & & \lTo^{h_{3}}& & K_{3} & & \lTo^{h_{4}} & & \cdots \end{diagram}
\]

The pro-homotopy equivalence class of $K_{1}\overset{h_{2}}{\longleftarrow
}K_{2}\overset{h_{3}}{\longleftarrow}K_{3}\overset{h_{4}}{\longleftarrow
}\cdots$ is fully determined by $Y$. That is easily verified by a diagram of
the form (\ref{diagram: pro-equivalence of neighborhood systems}), along with
the transitivity of the pro-homotopy equivalence relation. Define the
\emph{shape of the end of }$Y$, denoted $\mathcal{S}\emph{h}\left(
\varepsilon\left(  Y\right)  \right)  $, to be the shape class of
$\underleftarrow{\lim}\left\{  K_{i},h_{i}\right\}  $. A compactum
$A\in\mathcal{S}\emph{h}\left(  \varepsilon\left(  Y\right)  \right)  $ can be
viewed as a physical representative of the illusive \textquotedblleft end of
$Y$\textquotedblright.

The following is immediate.

\begin{theorem}
\label{Theorem: Sh(e(X)) determines end invariants}Let $Y$ be an inward tame
ANR and $A\in\mathcal{S}\emph{h\!}\left(  \varepsilon\left(  Y\right)
\right)  $. Then

\begin{enumerate}
\item $\operatorname*{pro}$-$H_{i}\left(  \varepsilon\left(  Y\right)
;R\right)  =\operatorname*{pro}$-$H_{i}\left(  A;R\right)  $ and $\check
{H}_{i}\left(  \varepsilon\left(  Y\right)  ;R\right)  \cong\check{H}%
_{i}\left(  A;R\right)  $ for all $i$ and any coefficient ring $R$, and

\item if $Y$ is 1-ended and semistable then $\operatorname*{pro}$-$\pi
_{i}\left(  \varepsilon\left(  Y\right)  \right)  =\operatorname*{pro}$%
-$\pi_{i}\left(  A\right)  $ and $\check{\pi}_{i}\left(  \varepsilon\left(
Y\right)  \right)  \allowbreak\cong\allowbreak\check{\pi}_{i}\left(  A\right)
$ for all $i$.\medskip
\end{enumerate}
\end{theorem}

The existence of diagrams like (\ref{Diagram: ladder diagran for a p.h.e.})
shows that $\mathcal{S}\emph{h}\left(  \varepsilon\left(  Y\right)  \right)  $
is also an invariant of the proper homotopy class of $Y$. There is also a
partial converse to that statement---an assertion about the proper homotopy
type of $Y$ based only on the shape of its end. Since the topology at the end
of a space does not determine the global homotopy type of that space, a new
definition is required.

Spaces $X$ and $Y$ are \emph{homeomorphic at infinity} if there exists a
homeomorphism $h:N\rightarrow M$, where $N\subseteq X$ and $M\subseteq Y$ are
neighborhoods of infinity. They are \emph{proper homotopy equivalent at
infinity} if there exist pairs of neighborhoods of infinity $N^{\prime
}\subseteq N$ in $X$ and $M^{\prime}\subseteq M$ in $Y$ and proper maps
$f:N\rightarrow Y$ and $g:M\rightarrow X$, with $\left.  g\circ f\right\vert
_{N^{\prime}}\overset{p}{\simeq}\operatorname*{incl}\left(  N^{\prime
}\hookrightarrow X\right)  $ and $\left.  f\circ g\right\vert _{M^{\prime}%
}\overset{p}{\simeq}\operatorname*{incl}\left(  M^{\prime}\hookrightarrow
Y\right)  $.

\begin{theorem}
\label{Theorem: shape of ends vs. p.h.e. at infinity}Let $X$ and $Y$ be inward
tame ANRs. Then $\mathcal{S}\emph{h}\left(  \varepsilon\left(  X\right)
\right)  =\mathcal{S}\emph{h}\left(  \varepsilon\left(  Y\right)  \right)  $
if and only if $X$ and $Y$ are proper homotopy equivalent at infinity.

\begin{proof}
The reverse implication follows from the previous paragraphs, while the
forward direction is nontrivial. A proof can be obtained by combining results
from \cite{CS} and \cite{EH}.
\end{proof}
\end{theorem}

In certain circumstances, the \textquotedblleft at infinity\textquotedblright%
\ phrase can be removed from the above. For example, we have.

\begin{corollary}
\label{Corollary: shape of ends vs. p.h.e}Let $X$ and $Y$ be contractible
inward tame ANRs. Then $\mathcal{S}\emph{h\!}\left(  \varepsilon\left(
X\right)  \right)  =\mathcal{S}\emph{h\!}\left(  \varepsilon\left(  Y\right)
\right)  $ if and only if $X$ and $Y$ are proper homotopy equivalent.
\end{corollary}

\begin{exercise}
Use the Homotopy Extension Property to obtain Corollary
\ref{Corollary: shape of ends vs. p.h.e} from Theorem
\ref{Theorem: shape of ends vs. p.h.e. at infinity}.
\end{exercise}

\begin{example}
\label{Example: tame X p.h.e. to a mapping telescope}If $K_{0}\overset{f_{1}%
}{\longleftarrow}K_{1}\overset{f_{2}}{\longleftarrow}K_{2}\overset{f_{3}%
}{\longleftarrow}\cdots$ is a sequence of finite complexes and
$A=\underleftarrow{\lim}\left\{  K_{i},f_{i}\right\}  $, it is easy to see
that $A$ represents $\mathcal{S}\emph{h}(\varepsilon\left(
\operatorname*{Tel}\left(  \left\{  K_{i},f_{i}\right\}  \right)  \right)  )$.
By Theorem \ref{Theorem: shape of ends vs. p.h.e. at infinity}, any inward
tame ANR $X$ with $\mathcal{S}\emph{h}(\varepsilon\left(  X\right)  )=A$ is
proper homotopy equivalent at infinity to $\operatorname*{Tel}\left(  \left\{
K_{i},f_{i}\right\}  \right)  $. When issues of global homotopy type are
resolved, even stronger conclusions are possible; for example, if $X$ is
contractible, $X\overset{p}{\simeq}\operatorname*{CTel}\left(  \left\{
K_{i},f_{i}\right\}  \right)  $. In some sense, the inverse mapping telescope
is an uncomplicated model for the end behavior of an inward tame ANR.
\end{example}

\begin{remark}
\emph{In this section, we have intentionally not required spaces to be
1-ended. So, for example,} $\mathcal{S}\emph{h}(\varepsilon\left(
\mathbb{R}
\right)  )$ \emph{is representable by a 2-point space and the shape of the end
of a ternary tree is representable by a Cantor set. For more complex
multi-ended }$X$\emph{, individual components of }$A\in\mathcal{S}%
\emph{h}(\varepsilon\left(  X\right)  )$ \emph{may have nontrivial shapes, and
to each end of }$X$\emph{ there will be a component of }$A$\emph{ whose shape
reflects properties of that end.}
\end{remark}

\section{$\mathcal{Z}$-sets and $\mathcal{Z}$-com\-pact\-ific\-at\-ions}

While reading \S \ref{Subsection: Shape of the end of a space}, the following
question may have occurred to the reader: \emph{For inward tame }$X$\emph{
with }$\mathcal{S}\emph{h\!}\left(  \varepsilon\left(  X\right)  \right)
=A$,\emph{ is there a way to glue }$A$\emph{ to the end of }$X$\emph{ to
obtain a nice com\-pact\-ific\-at\-ion? }As stated, that question is a bit too
simple, but it provides reasonable motivation for the material in this section.

\subsection{Definitions and examples}

A closed subset $A$ of an ANR\ $X$ is a $\mathcal{Z}$\emph{-set} if any of the
following equivalent conditions is satisfied:

\begin{itemize}
\item For every $\varepsilon>0$ there is a map $f:X\rightarrow X-A$ that is
$\varepsilon$-close to the identity.

\item There exists a homotopy $H:X\times\left[  0,1\right]  \rightarrow X$
such that $H_{0}=\operatorname*{id}_{X}$ and $H_{t}\left(  X\right)  \subseteq
X-A$ for all $t>0$. (We say that $H$ \emph{instantly homotopes} $X$ \emph{off
of} $A$.)

\item For every open set $U$ in $X$, $U-A\hookrightarrow U$ is a homotopy equivalence.
\end{itemize}

\noindent The third condition explains some alternative terminology:
$\mathcal{Z}$-sets are sometimes called \emph{homotopy negligible }sets.

\begin{example}
The $\mathcal{Z}$-sets in a manifold $M^{n}$ are precisely the closed subsets
of $\partial M^{n}.$ In particular, $\partial M^{n}$ is a $\mathcal{Z}$-set in
$M^{n}$.
\end{example}

\begin{example}
\label{Example: Z-sets in Q}It is a standard fact that every compactum $A$ can
be embedded in the Hilbert cube $\mathcal{Q}$. It may be embedded as a
$\mathcal{Z}$-set as follows: embed $A$ in the \textquotedblleft
face\textquotedblright\ $\left\{  1\right\}  \times\prod_{i=2}^{\infty}\left[
-1,1\right]  \subseteq\prod_{i=1}^{\infty}\left[  -1,1\right]  =\mathcal{Q}$.
\end{example}

Example \ref{Example: Z-sets in Q} is the starting point for a remarkable
characterization of shape, sometimes used as an alternative definition. We
will not attempt to describe a proof.

\begin{theorem}
[Chapman's Complement Theorem, \cite{Ch}]Let $A$ and $B$ be compacta embedded
as $\mathcal{Z}$-sets in $\mathcal{Q}$. Then $Sh\left(  A\right)  =Sh\left(
B\right)  $ if and only if $\mathcal{Q-}A\approx\mathcal{Q}-B$.
\end{theorem}

A $\mathcal{Z}$\emph{-com\-pact\-ific\-at\-ion} of a space $Y$ is a
com\-pact\-ific\-at\-ion $\overline{Y}=Y\sqcup Z$ with the property that $Z$
is a $\mathcal{Z}$-set in $\overline{Y}$. In this case, $Z$ is called a
$\mathcal{Z}$\emph{-boundary }for $Y$. Implicit in this definition is the
requirement that $\overline{Y}$ be an ANR; and since an open subset of an ANR
is an ANR, $Y$ must be an ANR to be a candidate for $\mathcal{Z}%
$-com\-pact\-ific\-at\-ion. By a result from the ANR theory, any
com\-pact\-ific\-at\-ion $\overline{Y}$ of an ANR $Y$, for which $\overline
{Y}-Y$ satisfies any of the above bullet points, is necessarily an
ANR---hence, it is a $\mathcal{Z}$-com\-pact\-ific\-at\-ion. The point here is
that, when attempting to form a $\mathcal{Z}$-com\-pact\-ific\-at\-ion, one
must begin with an ANR $Y$. Then it is enough to find a compactification
satisfying one of the above equivalent conditions.

A nice property of a $\mathcal{Z}$-com\-pact\-ific\-at\-ion is that the
homotopy type of a space is left unchanged by the compactification; for
example, a $\mathcal{Z}$-com\-pact\-ific\-at\-ion of a contractible space is
contractible. The prototypical example is the com\-pact\-ific\-at\-ion of $%
\mathbb{R}
^{n}$ to an $n$-ball by addition of the $\left(  n-1\right)  $-sphere at
infinity; the prototypical non-example is the 1-point com\-pact\-ific\-at\-ion
of $%
\mathbb{R}
^{n}$. Finer relationships between $Y$, $\overline{Y}$, and $Z$ can be
understood via shape theory and the study of ends. Before moving in that
direction, we add to our collection of examples.

\begin{example}
\label{Example: Prototypical manifold Z-compactification}In manifold topology,
the most fundamental $\mathcal{Z}$-com\-pac\-tifi\-ca\-tion is the addition of
a manifold boundary to an open manifold, as discussed in
\S \ref{Subsection: Siebenmann's thesis}.
\end{example}

Not all $\mathcal{Z}$-com\-pact\-ific\-at\-ions of open manifolds are as
simple as the above.

\begin{example}
\label{Example: Decomposition Z-compactification}Let $C^{n}$ be a Newman
contractible $n$-manifold embedded in $\mathbb{S}^{n}$ (as it is by
construction). A non-standard $\mathcal{Z}$-com\-pact\-ific\-at\-ion of
$\operatorname*{int}\mathbb{B}^{n+1}$ can be obtained by crushing $C^{n}$ to a
point. In this case, the quotient $\mathbb{S}^{n}/C^{n}$ is a $\mathcal{Z}%
$-set in $\mathbb{B}^{n+1}/C^{n}$. Note that $\mathbb{S}^{n}/C^{n}$ is not a manifold!

For those who prefer lower-dimensional examples, a similar $\mathcal{Z}%
$-com\-pact\-ific\-at\-ion of $\operatorname*{int}\mathbb{B}^{4}$ can be
obtained by crushing out a wild arc or a Whitehead continuum in $\mathbb{S}%
^{3}$. In terms of dimension, that is as low as it gets. As a result of
Corollary \ref{Corollary: homology properties Z-compactified manifolds} (still
to come), for $n\leq2$, a $\mathcal{Z}$-boundary of $\mathbb{B}^{n+1}$ is
necessarily homeomorphic to $\mathbb{S}^{n}$.
\end{example}

\begin{example}
\label{Example: suspending a Newman manifold}Let $\Sigma C^{n}$ be the
suspension of a Newman compact contractible $n$-manifold. The suspension of
$\partial C^{n}$ is a $\mathcal{Z}$-set in $\Sigma C^{n}$, and its complement,
$\operatorname*{int}C^{n}\times\left(  -1,1\right)  $, is homeomorphic to $%
\mathbb{R}
^{n+1}$ by Exercise \ref{Exercise: contractible cross R}. So this is another
nonstandard $\mathcal{Z}$-com\-pact\-ific\-at\-ion of $%
\mathbb{R}
^{n+1}$.
\end{example}

\begin{exercise}
Verify the assertions made in Examples
\ref{Example: Decomposition Z-compactification} and
\ref{Example: suspending a Newman manifold}.
\end{exercise}

Often a manifold that cannot be compactified by addition of a manifold
boundary is, nevertheless, $\mathcal{Z}$-compactifiable---a fact that is key
to the usefulness of $\mathcal{Z}$-com\-pact\-ific\-at\-ions. Davis manifolds
are the ideal examples.

\begin{example}
\label{Example: Z-compactifying the Davis manifold}The 1-point
com\-pact\-ific\-at\-ion of the infinite boundary connected sum $C_{0}%
^{n}\overset{\partial}{\#}\left(  -C_{1}^{n}\right)  \overset{\partial
}{\#}\left(  C_{2}^{n}\right)  \overset{\partial}{\#}\left(  -C_{3}%
^{n}\right)  \overset{\partial}{\#}\cdots$ shown at the top of Figure
\ref{Figure: Double bubbles} is a $\mathcal{Z}$-com\-pact\-ific\-at\-ion. More
significantly the point at infinity together with the original manifold
boundary form a $\mathcal{Z}$-boundary for the corresponding Davis manifold
$\mathcal{D}^{n}$. It is interesting to note that $\mathcal{D}^{n}$ cannot
admit a $\mathcal{Z}$-com\-pact\-ific\-at\-ion with $\mathcal{Z}$-boundary a
manifold (or even an ANR) since $\operatorname*{pro}$-$\pi_{1}\left(
\varepsilon\left(  \mathcal{D}^{n}\right)  \right)  $ is not stable. This will
be explained soon.
\end{example}

\begin{example}
\label{Example: Adding the CAT(0) boundary}In geometric group theory, the
prototypical $\mathcal{Z}$-com\-pact\-ific\-at\-ion is the addition of the
visual boundary $\partial_{\infty}X$ to a proper CAT(0) space $X$. Indeed, if
$\partial_{\infty}X$ is viewed as the set of end points of all infinite
geodesic rays emanating from a fixed $p_{0}\in X$, a homotopy pushing inward
along those rays verifies the $\mathcal{Z}$-set property.
\end{example}

\begin{example}
\label{Example: Adding CAT(0) boundaries to Davis and DJ manifolds}In
\cite{ADG}, an equivariant CAT(0) metric is placed on many of the original
Davis manifolds. In \cite{DJ} an entirely different construction produces
locally CAT(0) closed aspherical manifolds, whose CAT(0) universal covers are
similar to Davis' earlier examples. These objects with their $\mathcal{Z}%
$-com\-pact\-ific\-at\-ions and $\mathcal{Z}$-boundaries provide interesting
common ground for manifold topology and geometric group theory.

At the expense of losing the isometric group actions, \cite{FG} places
CAT($-1$) metrics on the Davis manifolds in such a way that the visual sphere
at infinity is homogeneous and nowhere locally contractible. Their method can
also be used to place CAT($-1$) metrics on the asymmetric Davis manifolds from
Example \ref{An asymmetric variation on the Davis manifolds}.
\end{example}

\begin{example}
\label{Example: Z-compactifying mapping telescopes}If $K_{0}\overset{f_{1}%
}{\longleftarrow}K_{1}\overset{f_{2}}{\longleftarrow}K_{2}\overset{f_{3}%
}{\longleftarrow}\cdots$ is an inverse sequence of finite polyhedra (or finite
CW complexes or compact ANRs), then the inverse mapping telescope
$\operatorname*{Tel}\left(  \left\{  K_{i},f_{i}\right\}  \right)  $ can be
$\mathcal{Z}$-compactified by adding a copy of $\underleftarrow{\lim}\left\{
K_{i},f_{i}\right\}  $, a space that contains one point for each of the
infinite telescope rays described in
\S \ref{Subsection: Inverse mapping telescopes}.\ (Note the similarity of this
to Example \ref{Example: Adding the CAT(0) boundary}.)
\end{example}

In \S \ref{Section: Z-boundaries in geometric group theory} and
\S \ref{Section: Z-boundaries in manifolds topology}, we will look at
applications of $\mathcal{Z}$-com\-pact\-ific\-at\-ion to geometric group
theory and manifold topology. Before that, we address a pair of purely
topological questions:\smallskip

\begin{itemize}
\item \emph{When is a space }$\mathcal{Z}$\emph{-compactifiable?}

\item \emph{To what extent is the }$\mathcal{Z}$\emph{-boundary of a given
space unique?} (Examples \ref{Example: Decomposition Z-compactification} and
\ref{Example: suspending a Newman manifold} show that a space can admit
nonhomeomorphic $\mathcal{Z}$-boundaries.)
\end{itemize}

\subsection{Existence and uniqueness for $\mathcal{Z}$%
-com\-pact\-ific\-at\-ions and $\mathcal{Z}$%
-boundaries\label{Section: Existence and uniqueness of Z-compactifications}}

If $Y$ admits a $\mathcal{Z}$-com\-pact\-ific\-at\-ion $\overline{Y}=Y\sqcup
Z$, then as noted above, $Y$ must be an ANR; and since $Y\hookrightarrow
\overline{Y}$ is a homotopy equivalence and $\overline{Y}$ is a compact ANR,
Proposition \ref{Proposition: ANR facts} implies that $Y$ has finite homotopy
type. Prying a bit deeper, a homotopy $H:\overline{Y}\times\left[  0,1\right]
\rightarrow\overline{Y}$ that instantly homotopes $\overline{Y}$ off $Z$ can
be \textquotedblleft truncated\textquotedblright\ (with the help of the
Homotopy Extension Property) to homotope arbitrarily small closed
neighborhoods of infinity (in $Y$) into compact subsets. Hence, $Y$ is
necessarily inward tame.

By combining the results noted in Examples
\ref{Example: tame X p.h.e. to a mapping telescope} and
\ref{Example: Z-compactifying mapping telescopes}, every inward tame ANR is
proper homotopy equivalent to one that is $\mathcal{Z}$-compactifiable.
Unfortunately, $\mathcal{Z}$-compactifiability is not an invariant of proper
homotopy type. The following result begins to make that clear.

\begin{proposition}
\label{Proposition: Z-compactifialble implies absolute inward tame}Every
$\mathcal{Z}$-compactifiable space that is sharp at infinity is absolutely
inward tame.
\end{proposition}

\begin{proof}
If $\overline{Y}=Y\sqcup Z$ is a $\mathcal{Z}$-com\-pact\-ific\-at\-ion and
$N$ is a closed ANR neighborhood of infinity in $Y$. Then $\overline{N}\equiv
N\sqcup Z$ is a $\ \mathcal{Z}$-com\-pact\-ific\-at\-ion of $N$, hence a
compact ANR, and therefore homotopy equivalent to a finite complex $K$. Since
$N\hookrightarrow\overline{N}$ is a homotopy equivalence, $N\simeq K$.
\end{proof}

\begin{remark}
\label{Remark: standard trick}\emph{By employing the standard trick of
considering} $Y\times\mathcal{Q}$\emph{ (to ensure sharpness at infinity),}
\emph{Proposition
\ref{Proposition: Z-compactifialble implies absolute inward tame} provides an
alternative proof that a }$\mathcal{Z}$\emph{-compactifiable ANR must be
inward tame. This also uses the straightforward observation that, if
}$\overline{Y}=Y\sqcup Z$ \emph{is a} $\mathcal{Z}$%
\emph{-com\-pact\-ific\-at\-ion of} $Y$\emph{, then} $\overline{Y}%
\times\mathcal{Q}=(Y\times\mathcal{Q)}\sqcup(Z\times\mathcal{Q)}$ \emph{is a
}$Z$\emph{-com\-pact\-ific\-at\-ion of} $Y\times\mathcal{Q}$\emph{. That
observation will be used numerous times, as we proceed.}
\end{remark}

\begin{theorem}
\label{Theorem: Z-boundaries have shape of the end}Suppose $Y$ admits a
$\mathcal{Z}$-com\-pact\-ific\-at\-ion $\overline{Y}=Y\sqcup Z$. Then
$Z\in\mathcal{S}\emph{h\!}\left(  \varepsilon\left(  Y\right)  \right)  $.

\begin{proof}
Arguing as in Remark \ref{Remark: standard trick}, we may assume without loss
of generality that $Y$ is sharp at infinity. Choose a cofinal sequence
$\left\{  N_{i}\right\}  $ of closed ANR neighborhoods of infinity in $Y$, and
for each $i$ let $\overline{N}_{i}$ be the compact ANR $N_{i}\sqcup Z$. The
homotopy commutative diagram%
\[
\begin{diagram} N_{0 } & & \lInto & & N_{1} & & \lInto & & N_{2} & & \lInto & & \cdots \\ & \luTo & & \ldInto & & \luTo & & \ldInto & & \luTo \\ & & \overline{N}_{0} & & \lInto & & \overline{N}_{1} & & \lInto & & \overline{N}_{2} & & \cdots \end{diagram}
\]
where each up arrow is a homotopy inverse of the corresponding $N_{i}%
\hookrightarrow\overline{N}_{i}$, shows that the lower sequence defines
$\mathcal{S}\emph{h\!}\left(  \varepsilon\left(  Y\right)  \right)  $. But,
since the inverse limit of that sequence is $Z$ (since $\cap\overline{N}%
_{i}=Z$), the sequence also defines the shape of $Z.$
\end{proof}

\begin{corollary}
[Uniqueness of $\mathcal{Z}$-boundaries up to shape]%
\label{Corollary: Uniquenss for Z-boundaries}All $\mathcal{Z}$-boundaries of a
given space $Y$ are shape equivalent. Even more, if $Y$ and $Y^{\prime}$ are
$\mathcal{Z}$-compactifiable and proper homotopy equivalent at infinity, then
each $\mathcal{Z}$-boundary of $Y$ is shape equivalent to each $\mathcal{Z}%
$-boundary of $Y^{\prime}$.
\end{corollary}
\end{theorem}

\begin{proof}
Combine the above theorem with Theorem
\ref{Theorem: shape of ends vs. p.h.e. at infinity}.
\end{proof}

\begin{example}
We can now verify the comment at the end of Example
\ref{Example: Z-compactifying the Davis manifold}. For any $\mathcal{Z}%
$-boundary $Z$ of a Davis manifold $\mathcal{D}^{n}$, $\operatorname*{pro}%
$-$\pi_{1}\left(  Z\right)  $ must match the nonstable $\operatorname*{pro}%
$-$\pi_{1}\left(  \varepsilon\left(  \mathcal{D}^{n}\right)  \right)  $
established in \S \ref{Subsection: Examples of fundamental groups at infinity}%
. So, by Proposition
\ref{Prop: For a compact ANR shape invariants are homotopy invariants}, $Z$
cannot be an ANR.
\end{example}

Next we examine the existence question for $\mathcal{Z}$%
-com\-pact\-ific\-at\-ions. By the above results we know that, for reasonably
nice $X$, absolute inward tameness is necessary; moreover, prospective
$\mathcal{Z}$-boundaries must come from $\mathcal{S}\emph{h\!}\left(
\varepsilon\left(  X\right)  \right)  $. It turns out that this is not enough.
The outstanding result on this topic is due to Chapman and Siebenmann
\cite{CS}. It provides a complete characterization of $\mathcal{Z}%
$-compactifiable Hilbert cube manifolds and a model for more general
characterization theorems.

Chapman and Siebenmann modeled their theorem on Siebenmann's Collaring Theorem
for finite-dimensional manifolds---but there are significant differences.
First, there is no requirement of a stable fundamental group at infinity;
therefore, a more flexible formulation of $\sigma_{\infty}\left(  X\right)  $,
like that developed in Theorem
\ref{Theorem: characterization of pseudocollarable}, is required. Second,
unlike finite-dimensional manifolds, inward tame Hilbert cube manifolds can be
infinite-ended. In fact, $\mathcal{Z}$-com\-pact\-ifi\-able Hilbert cube
manifolds can be infinite-ended ($\mathbb{T}_{3}\times\mathcal{Q}$ is a simple
example); therefore, we do not want to be restricted to the 1-ended case. This
generality requires an even more flexible approach to the definition of
$\sigma_{\infty}\left(  X\right)  $. For the sake of simplicity, we delay that
explanation until the final stage of the coming proof. We recommend that
during the first reading, a tacit assumption of 1-endedness be included.

The third difference is the appearance of a new obstruction lying in the first
derived limit of an inverse sequence of Whitehead groups. The topological
meaning of this obstruction is explained within the sketched proof. For
completeness, we include the algebraic formulation: For an inverse sequence
$\left\{  G_{i},\lambda_{i}\right\}  $ of abelian groups, the \emph{derived
limit}\footnote{The definition of derived limit can be generalized to include
nonableian groups (see \cite[\S 11.3]{Ge2}), but that is not needed here.} is
the quotient group:%
\[
\underleftarrow{\lim}^{1}\left\{  G_{i},\lambda_{i}\right\}  =\left(
\prod\limits_{i=0}^{\infty}G_{i}\right)  /\left\{  \left.  \left(
g_{0}-\lambda_{1}g_{1},g_{1}-\lambda_{2}g_{2},g_{2}-\lambda_{3}g_{3}%
,\cdots\right)  \right\vert \ g_{i}\in G_{i}\right\}  .
\]
\emph{\smallskip}

\begin{theorem}
[The Chapman-Siebenmann $\mathcal{Z}$-compactification Theorem]%
\label{Theorem: C-S Z-compactification Theorem}A Hilbert cube manifold $X$
admits a $\mathcal{Z}$-com\-pact\-ific\-at\-ion if and only if each of the
following is satisfied.

\begin{enumerate}
\item $X$ is inward tame,

\item $\sigma_{\infty}\left(  X\right)  \in\underleftarrow{\lim}\left\{
\widetilde{K}_{0}(%
\mathbb{Z}
\left[  \pi_{1}\left(  N\right)  \right]  )\mid N\text{ a clean nbd. of
infinity}\right\}  $ is trivial, and

\item $\tau_{\infty}\left(  X\right)  \in\underleftarrow{\lim}^{1}\left\{
Wh(\pi_{1}\left(  N\right)  )\mid N\text{ a clean nbd. of infinity}\right\}  $
is trivial.\medskip
\end{enumerate}
\end{theorem}

\begin{remark}
\label{Remark: Remarks about the C-S Theorem}\textbf{(a)}\emph{ \ By Theorem
\ref{Theorem: Edwards HCM Theorem}, every ANR }$Y$\emph{ becomes a Hilbert
cube manifold upon crossing with }$\mathcal{Q}$\emph{. So, reasoning as in
Remark \ref{Remark: standard trick}, conditions (a)-(c) are also necessary for
}$\mathcal{Z}$\emph{-compactifiability of an ANR (although Condition (b) and
particularly (c) are best measured in }$Y\times\mathcal{Q}$\emph{). For some
time, it was hoped that (a)-(c) would also be sufficient for ANRs; but in
\cite{Gu}, a }$2$\emph{-dimensional polyhedral counterexample was
constructed.\smallskip}

\noindent\textbf{(b)}\emph{ \ For those who prefer finite-dimensional spaces,
Ferry \cite{Fe1} has shown that, if }$P$ \emph{is a }$k$\emph{-dimensional
locally finite polyhedron and }$P\times\mathcal{Q}$\emph{ is }$\mathcal{Z}%
$\emph{-compactifiable, then }$P\times\left[  0,1\right]  ^{2k+5}$\emph{ is
}$\mathcal{Z}$\emph{-compactifiable. May \cite{May} showed that, for the
counterexample }$P_{0}$\emph{ from \cite{Gu}, }$P_{0}\times\left[  0,1\right]
\ $\emph{is }$\mathcal{Z}$\emph{-compactifiable. In still-to-be-published
work, the author has shown that, for an open manifold }$M^{n}$\emph{
satisfying (a)-(c), }$M^{n}\times$\emph{ }$\left[  0,1\right]  $ \emph{is
}$\mathcal{Z}$\emph{-compactifiable.}
\end{remark}

The following are significant and still open.\smallskip

\begin{problem}
Find conditions that must be added to those of Theorem
\ref{Theorem: C-S Z-compactification Theorem} to obtain a characterization of
$\mathcal{Z}$-compactifiability for ANRs.
\end{problem}

\begin{problem}
Determine whether the conditions of Theorem
\ref{Theorem: C-S Z-compactification Theorem} are sufficient in the case of
finite-dimensional manifolds.
\end{problem}

Before describing the proof of Theorem
\ref{Theorem: C-S Z-compactification Theorem}, we make some obvious
adaptations of terminology from \S \ref{Subsection: Siebenmann's thesis} and
\ref{Subsection: Generalizing Siebenman}. A clean neighborhood of infinity $N$
in a Hilbert cube manifold $X$ is an \emph{open collar} if $N\approx
\operatorname*{Bd}_{X}N\times\lbrack0,1)$ and a \emph{homotopy collar} if
$\operatorname*{Bd}_{X}N\hookrightarrow N$ is a homotopy equivalence. $X$ is
\emph{collarable} if it contains an open collar neighborhood of infinity and
\emph{pseudo-collarable} if it contains arbitrarily small homotopy collar
neighborhoods of infinity.\medskip

\begin{proof}
[Sketch of the proof of Theorem \ref{Theorem: C-S Z-compactification Theorem}%
.]The necessity of Conditions (a) and (b) follows from Proposition
\ref{Proposition: Z-compactifialble implies absolute inward tame}; for the
necessity of (c), the reader is referred to \cite{CS}. Here we will focus on
the sufficiency of these conditions.

Assume that $X$ satisfies Conditions (a)-(c). We show that $X$ is
$\mathcal{Z}$-compactifiable by showing that it is homeomorphic at infinity to
$\operatorname*{Tel}\left(  \{K_{i},f_{i}\}\right)  \times\mathcal{Q}$, where
$\left\{  K_{i},f_{i}\right\}  $ is a carefully chosen inverse sequence of
finite polyhedra. Since inverse mapping telescopes are $\mathcal{Z}%
$-compactifiable (Example \ref{Example: Z-compactifying mapping telescopes}),
the result follows.

It is easiest to read the following argument under the added assumption that
$X$ is 1-ended. In the final step, we explain how that assumption can be
eliminated.\smallskip

\noindent\textbf{Step 1. }(Existence of a pseudo-collar structure) Choose a
nested cofinal sequence $\left\{  N_{i}^{\prime}\right\}  $ of clean
neighborhoods of infinity. By Condition (a) each is finitely dominated, so we
may represent $\sigma_{\infty}\left(  X\right)  $ by $\left(  \sigma
_{0},\sigma_{1},\sigma_{2},\cdots\right)  $, where $\sigma_{i}$ is the Wall
finiteness obstruction of $N_{i}^{\prime}$. By (b) each $\sigma_{i}=0$, so
each $N_{i}^{\prime}$ has finite homotopy type. (Said differently, Conditions
(a) and (b) are equivalent to absolute inward tameness.) For each $i$, choose
a finite polyhedron $K_{i}$ and an embedding $K_{i}\hookrightarrow
N_{i}^{\prime}$ that is a homotopy equivalences. By taking neighborhoods
$C_{i}$ of the $K_{i}$, we arrive at a sequence of Hilbert cube manifold pairs
$\left(  N_{i}^{\prime},C_{i}\right)  $, where each inclusion is a homotopy
equivalence. By some Hilbert cube manifold magic it can be arranged that
$C_{i}$ is a $\mathcal{Z}$-set in $N_{i}^{\prime}$ and $\operatorname*{Bd}%
N_{i}^{\prime}\subseteq C_{i}$. From there one finds $N_{i}\subseteq
N_{i}^{\prime}$ for which $\operatorname*{Bd}N_{i}$ is a copy of $C_{i}$ and
$\operatorname*{Bd}N_{i}\hookrightarrow N_{i}$ a homotopy equivalence (see
\cite{CS} for details). Thus $\left\{  N_{i}\right\}  $ is a pseudo-collar
structure.\smallskip

\noindent\textbf{Step 2. }(Pushing the torsion off the end of $X$) By letting
$A_{i}=\overline{N_{i}-N}_{i+1}$ for each $i$, view the end of $X$ as a
countable union $A_{0}\cup A_{1}\cup A_{2}\cup\cdots$ of compact 1-sided
h-cobordisms $\left(  A_{i},\operatorname*{Bd}N_{i},\operatorname*{Bd}%
N_{i+1}\right)  $ of Hilbert cube manifolds. (See Exercise
\ref{Exercise: cobordisms are [one-sided] h-cobordisms}.) By the
triangulability of Hilbert cube manifolds (and pairs), each inclusion
$\operatorname*{Bd}N_{i}\hookrightarrow A_{i}$ has a well-defined torsion
$\tau_{i}\in\operatorname*{Wh}\left(  \pi_{1}\left(  N_{i}\right)  \right)  $.
Together these torsions determine a representative $\left(  \tau_{0},\tau
_{1},\tau_{2},\cdots\right)  $ of $\tau_{\infty}\left(  X\right)  $.
(\textbf{Note: }Determining $\tau_{\infty}\left(  X\right)  $ requires that
Step 1 first be accomplished; there is no $\tau_{\infty}\left(  X\right)  $
without Conditions (a) and (b) being satisfied.)

We would like to alter the choices of the $N_{i}$ by using an infinite
borrowing strategy like that employed in the proof of Theorem
\ref{Theorem: Siebenmann's thesis}. In particular, we would like to borrow a
compact Hilbert cube manifold h-cobordism $B_{0}$ from a collar neighborhood
of $\operatorname*{Bd}N_{1}$ in $A_{1}$ so that $\operatorname*{Bd}%
N_{0}\hookrightarrow A_{0}\cup B_{0}$ has trivial torsion. Then, replacing
$N_{1}$ with $\overline{N_{1}-B}_{0}$, we would like to borrow $B_{1}$ from
$A_{2}$ so that so that $\operatorname*{Bd}N_{1}\hookrightarrow A_{1}\cup
B_{1}$ has trivial torsion. Continuing inductively, we would like to arrive at
an adjusted sequence $N_{0}\supseteq N_{1}\supseteq N_{2}\supseteq\cdots$ of
neighborhoods of infinity for which each $\operatorname*{Bd}N_{i}%
\hookrightarrow A_{i}$ has trivial torsion (where the $A_{i}$ are redefined
using the new $N_{i}$).

The derived limit, $\underleftarrow{\lim}^{1}$, is defined precisely to
measure whether this infinite borrowing strategy can be successfully
completed. In the situation of \ref{Theorem: Siebenmann's thesis}, where the
fundamental group stayed constant from one side of each $A_{i}$ to the other,
there was no obstruction to the borrowing scheme. More generally, as long as
the inclusion induced homomorphisms $\operatorname*{Wh}\left(  \pi_{1}\left(
N_{i+1}\right)  \right)  \rightarrow\operatorname*{Wh}\left(  \pi_{1}\left(
N_{i}\right)  \right)  $ are surjective for all $i$, the strategy works. But,
in general, we must rely on a hypothesis that $\tau_{\infty}\left(  X\right)
$ is the trivial element of $\underleftarrow{\lim}^{1}\left\{  Wh(\pi
_{1}\left(  N_{i}\right)  )\right\}  $. (\emph{Warning}: Even when $\pi
_{1}\left(  N_{i+1}\right)  $ surjects onto $\pi_{1}\left(  N_{i}\right)  $,
$\operatorname*{Wh}\left(  \pi_{1}\left(  N_{i+1}\right)  \right)
\rightarrow\operatorname*{Wh}\left(  \pi_{1}\left(  N_{i}\right)  \right)  $
can fail to be surjective.)\smallskip

\noindent\textbf{Step 3. }(Completion of the proof) Homotopy equivalences
$K_{i}\hookrightarrow\operatorname*{Bd}N_{i}$ and the deformation retractions
of $A_{i}$ onto $\operatorname*{Bd}N_{i}$ determine maps $f_{i+1}%
:K_{i+1}\rightarrow K_{i}$ and homotopy equivalences of triples $\left(
A_{i},\operatorname*{Bd}N_{i},\operatorname*{Bd}N_{i+1}\right)  \allowbreak
\simeq\allowbreak\left(  \operatorname*{Map}\left(  f_{i+1}\right)
,K_{i},K_{i+1}\right)  $. Using the fact that both $\operatorname*{Bd}%
N_{i}\hookrightarrow A_{i}$ and $K_{i}\hookrightarrow\operatorname*{Map}%
\left(  f_{i+1}\right)  $ have trivial torsion (and through more Hilbert cube
manifold magic), we obtain a homeomorphism of triples $\left(  A_{i}%
,\operatorname*{Bd}N_{i},\operatorname*{Bd}N_{i+1}\right)  \allowbreak
\simeq\allowbreak\left(  \operatorname*{Map}\left(  f_{i+1}\right)
\times\mathcal{Q},K_{i}\times\mathcal{Q},K_{i+1}\times\mathcal{Q}\right)  $.
Piecing these together gives a homeomorphism $N_{0}\approx\operatorname*{Tel}%
\left(  \{K_{i},f_{i}\}\right)  \times\mathcal{Q}$, and completes the proof.

As a mild alternative, we could have used the ingredients described above to
construct a proper homotopy equivalence $h:N_{0}\rightarrow\operatorname*{Tel}%
\left(  K_{i},f_{i}\right)  $ and used the triviality of the torsions to argue
that $h$ is an \textquotedblleft infinite simple homotopy
equivalence\textquotedblright, in the sense of \cite{Si2}. Then, by a
variation on Theorem \ref{Theorem: Chapman's homeomorphism theorem},
$\widehat{h}:N_{0}\rightarrow\operatorname*{Tel}\left(  K_{i},f_{i}\right)
\times\mathcal{Q}$ is homotopic to a homeomorphism.\smallskip

\noindent\textbf{Final Step. }(Multi-ended spaces) When $X$ is multi-ended
(possibly infinite-ended), a neighborhood of infinity $N_{i}$ used in defining
$\sigma_{\infty}$ and $\tau_{\infty}$ will have multiple (but finitely many)
components. In that case, define $\widetilde{K}_{0}(%
\mathbb{Z}
\left[  \pi_{1}\left(  N_{i}\right)  \right]  )$ and $Wh(\pi_{1}\left(
N_{i}\right)  )$ to be the direct sums $\bigoplus\widetilde{K}_{0}(%
\mathbb{Z}
\left[  \pi_{1}\left(  C_{j}\right)  \right]  )$ and $\bigoplus Wh(\pi
_{1}\left(  C_{j}\right)  )$, where $\left\{  C_{j}\right\}  $ is the
collection of components of $N_{i}$. With these definitions, a little extra
work, and the fact that reduced projective class groups and Whitehead groups
of free products are the corresponding direct sums, the above steps can be
carried out as in the 1-ended case.
\end{proof}

\begin{remark}
\emph{If desired, one can arrange, in the final lines of Step 3, a
homeomorphism} $k:X\rightarrow\operatorname*{Tel}^{\ast}\left(  K_{i}%
,f_{i}\right)  \times\mathcal{Q}$, \emph{defined on all of }$X$\emph{. The
space on the right is the previous mapping telescope with the addition of a
single mapping cylinder }$\operatorname*{Map}(K_{0}\overset{f_{0}%
}{\longrightarrow}K_{-1})$\emph{. The finite complex }$K_{-1}$\emph{ and the
map }$f_{0}$\emph{ are carefully chosen so that }$X$\emph{ and }%
$\operatorname*{Tel}^{\ast}\left(  K_{i},f_{i}\right)  $\emph{ are infinite
simple homotopy equivalent.}
\end{remark}

Step 1 of the above proof provides a result that is interesting in its own right.

\begin{theorem}
\label{Theorem: HCM pseudo-collarability theorem}A Hilbert cube manifold is
pseudo-collarable if and only if it satisfies Conditions (a) and (b) of
Theorem \ref{Theorem: C-S Z-compactification Theorem} or, equivalently, if and
only if it is absolutely inward tame.
\end{theorem}

It is interesting to compare \ref{Theorem: HCM pseudo-collarability theorem}
to Theorems \ref{Theorem: characterization of pseudocollarable} and
\ref{Theorem: existence of nonpseudocollarable manifolds}.

\section{$\mathcal{Z}$-boundaries in geometric group
theory\label{Section: Z-boundaries in geometric group theory}}

In this section we look at the role of $\mathcal{Z}$-com\-pact\-ific\-at\-ions
and $\mathcal{Z}$-boundaries in geometric group theory.

\subsection{Boundaries of $\delta$-hyperbolic groups}

Following Gromov \cite{Gr}, for a metric space $\left(  X,d\right)  $ with
base point $p_{0}$, define the \emph{overlap function on }$X\times X$ by%
\[
\left(  x\cdot y\right)  =\frac{1}{2}(d\left(  x,p_{0}\right)  +d\left(
y,p_{0}\right)  -d\left(  x,y\right)  ).
\]
Call $\left(  X,d\right)  $ $\delta$-\emph{hyperbolic} if there exists a
$\delta>0$ such that $\left(  x\cdot y\right)  \geq\min\left\{  \left(  x\cdot
z\right)  ,\left(  y\cdot z\right)  \right\}  -\delta$, for all $x,y,z\in X$.

A sequence $\left\{  x_{i}\right\}  $ in a $\delta$-hyperbolic space $\left(
X,d\right)  $ is \emph{convergent at infinity} if $\left(  x_{i},x_{j}\right)
\rightarrow\infty$ as $i,j\rightarrow\infty$, and sequences $\left\{
x_{i}\right\}  $, and $\left\{  y_{i}\right\}  $ are declared to be equivalent
if $\left(  x_{i},y_{i}\right)  \rightarrow\infty$ as $i\rightarrow\infty$.
The set $\partial X$ of all equivalence classes of these sequences makes up
the \emph{Gromov} \emph{boundary} of $X$. An easy to define topology on
$X\sqcup\partial X$ results in a corresponding com\-pact\-ific\-at\-ion
$\widehat{X}=X\cup\partial X$. This boundary and compactification are
well-defined in the following strong sense: if $f:X\rightarrow Y$ is a
quasi-isometry between $\delta$-hyperbolic spaces, then there is a unique
extension $\widehat{f}:\widehat{X}\rightarrow\widehat{Y}$ that restricts to a
homeomorphism between boundaries. This is of particular interest when $G$ is a
finitely generated group endowed with a corresponding word metric. It is a
standard fact that, for any two such metrics, $G\overset{\operatorname*{id}%
}{\rightarrow}G$ is a quasi-isometry; so for a $\delta$-hyperbolic group\ $G$,
the \emph{Gromov boundary} $\partial G$ is well-defined.

Early in the study of $\delta$-hyperbolic groups, it became clear that exotic
topological spaces arise naturally as group boundaries. In addition to spheres
of all dimensions, the collection of known boundaries includes: Cantor sets,
Sierpinski carpets, Menger curves, Pontryagin surfaces, and 2-dimensional
Menger spaces, to name a few. See \cite{Be}, \cite{Dr2} and \cite{KB}. So it
is not surprising that shape theory has a role to play in this area. But, a
priori, Gromov's com\-pact\-ific\-at\-ions and boundaries have little in
common with $\mathcal{Z}$-com\-pact\-ific\-at\-ions and $\mathcal{Z}%
$-boundaries. After all, for a word hyperbolic group, the Gromov
com\-pact\-ific\-at\-ion adds boundary to a \emph{discrete} topological space.

\begin{exercise}
Show that a countably infinite discrete metric space does not admit a
$\mathcal{Z}$-com\-pact\-ific\-at\-ion.
\end{exercise}

\noindent Nevertheless, in 1991, Bestvina and Mess introduced the use of
$\mathcal{Z}$-boundaries and $\mathcal{Z}$-com\-pact\-ific\-at\-ions to the
study of $\delta$-hyperbolic groups. For a discrete metric space $\left(
X,d\right)  $ and a constant $\rho$, the \emph{Rips complex} $P_{\rho}\left(
G\right)  $ is the simplicial complex obtained by declaring the vertex set to
be $X$ and filling in an $n$-simplex for each collection $\left\{
x_{0},x_{1,}\cdots,x_{n}\right\}  $ with $d\left(  x_{i},x_{j}\right)
\leq\rho$, for all $0\leq i,j\leq n$. Clearly, a Rips complex $P_{\rho}\left(
G\right)  $ for a finitely generated group $G$ admits a proper, cocompact
$G$-action and $G\hookrightarrow P_{\rho}\left(  G\right)  $ is a
quasi-isometry. So when $G$ is $\delta$-hyperbolic there is a canonical
com\-pact\-ific\-at\-ion $\overline{P_{\rho}\left(  G\right)  }=P_{\rho
}\left(  G\right)  \cup\partial G$. Furthermore, it was shown by Rips that,
for $\delta$-hyperbolic $G$ and large $\rho$, $P_{\rho}\left(  G\right)  $ is
contractible. Using this and some finer homotopy properties of $P_{\rho
}\left(  G\right)  $, Bestvina and Mess proved the following.

\begin{theorem}
[{\cite[Th.1.2]{BM}}]\label{Theorem: Bestvina-Mess} Let $G$ be a $\delta
$-hyperbolic group and $\rho\geq4\delta+2$, then $\overline{P_{\rho}\left(
G\right)  }=P_{\rho}\left(  G\right)  \cup\partial G$ is a $\mathcal{Z}$-com\-pact\-ific\-at\-ion.
\end{theorem}

Implications of Theorem \ref{Theorem: Bestvina-Mess} are cleanest when
$P_{\rho}\left(  G\right)  $ is a cocompact $EG$ complex. Since
contractibility and a proper cocompact action have already been established,
only freeness is needed, and that is satisfied if and only if $G$ is torsion-free.

\begin{corollary}
\label{Corollary: Corollary to Bestvina-Mess theorem}Let $G$ be a torsion-free
$\delta$-hyperbolic group. Then

\begin{enumerate}
\item every cocompact $EG$ complex is inward tame and proper homotopy
equivalent to $P_{\rho}\left(  G\right)  ,$

\item for every cocompact $EG$ complex $X$, $\mathcal{S}\emph{h\!}\left(
\varepsilon\left(  X\right)  \right)  =\mathcal{S}\emph{h\!}\left(  \partial
G\right)  $,

\item $\operatorname*{pro}$-$H_{\ast}\left(  \varepsilon(G);R\right)  $,
$\check{H}_{\ast}\left(  \varepsilon(G);R\right)  $ and $\check{H}^{\ast
}\left(  \varepsilon(G);R\right)  $ are isomorphic to the corresponding
invariants of $\partial G$,

\item for 1-ended $G$, $\operatorname*{pro}$-$\pi_{\ast}\left(  \varepsilon
(G)\right)  $ and $\check{\pi}_{\ast}\left(  \varepsilon(G)\right)  $ are well
defined and isomorphic to the corresponding invariants of $\mathcal{S}%
\emph{h\!}\left(  \partial G\right)  $,

\item $H^{\ast}\left(  G;RG\right)  \cong\check{H}^{\ast-1}\left(  \partial
G;R\right)  $ for any coefficient ring $R.$
\end{enumerate}
\end{corollary}

\begin{proof}
[Proof of Corollary]The discussion in \S \ref{Subsection: groups of type F}
explains why all cocompact $EG$ complexes are proper homotopy equivalent.
Since one such space, $P_{\rho}\left(  G\right)  $, is $\mathcal{Z}%
$-compactifiable and therefore inward tame, they are all inward tame. By
Theorem \ref{Theorem: Z-boundaries have shape of the end}, $\mathcal{S}%
\emph{h\!}\left(  \varepsilon\left(  P_{\rho}\left(  G\right)  \right)
\right)  \allowbreak=\allowbreak\mathcal{S}\emph{h\!}\left(  \partial
G\right)  $, so Theorem \ref{Theorem: shape of ends vs. p.h.e. at infinity}
completes (b). Assertion (c) is a consequence of Theorem
\ref{Theorem: Sh(e(X)) determines end invariants}, while (d) is similar,
except that Theorem \ref{Theorem: Some groups that are semistable} (a
significant ingredient) is used to assure well-definedness. Assertion (e), a
statement about group cohomology with coefficients in $RG$, requires some
algebraic topology that is explained in \cite{BM}; it is a consequence of (c)
and builds upon earlier work by Geoghegan and Mihalik \cite{Ge2},
\cite{GM}.\medskip
\end{proof}

For the most part, Corollary
\ref{Corollary: Corollary to Bestvina-Mess theorem} is all about the shape of
$\partial G$ and the relationship between a $\mathcal{Z}$-boundary and its
complement. There are other applications of boundaries of $\delta$-hyperbolic
groups that use more specific properties of $\partial G$. Here is a small
sampling:\smallskip

\begin{itemize}
\item \cite{BM} provides formulas relating the cohomological dimension of a
torsion-free $G$ to the topological dimension of $\partial G$. (Clearly, the
latter is not a shape invariant.)

\item The semistability of $G$ was deduced by proving that $\partial G$ has no
cut points \cite{Swa}, and therefore is locally connected, by results from
\cite{BM}.

\item By work from \cite{Tuk},\cite{Gab},\cite{CaJu} and \cite{Fre}, $\partial
G\approx\mathbb{S}^{1}$ if and only if $G$ is virtually the fundamental group
of a closed hyperbolic surface.

\item Bowditch \cite{Bo} has obtained a JSJ-decomposition theorem for $\delta
$-hyperbolic groups by analyzing cut pairs in $\partial G$.

\item See \cite{KB} for many more examples.
\end{itemize}

\subsection{Boundaries of CAT(0) groups}

Another widely studied class of groups are the \emph{CAT(0) groups}, i.e.,
groups $G$ that act geometrically (properly and cocompactly by isometries) on
a proper CAT(0) space. If $X$ is such a CAT(0) space, the visual boundary
$\partial_{\infty}X$ is called a \emph{group boundary} for $G$. Since a given
$G$ may act geometrically on multiple proper CAT(0) spaces, it is not
immediate that its boundary is topologically well-defined; and, in fact, it is
not. The first example of this phenomenon was displayed by Croke and Kleiner
\cite{CrKl}. Their work was expanded upon by Wilson \cite{Wi}, who showed that
their group admits a continuum of topologically distinct boundaries. Mooney
\cite{Mo1} discovered additional examples from the category of knot groups,
and in \cite{Mo2}\ produced another collection of examples with boundaries of
arbitrary dimension $k\geq1$. This situation suggests that CAT(0) boundaries
(being ill-defined) might not be useful---but that is not the case. One reason
is the following \textquotedblleft approximate
well-definedness\textquotedblright\ result.

\begin{theorem}
[Uniqueness of CAT(0) boundaries up to shape]%
\label{Theorem: Shape uniqueness of CAT(0) boundaries}All CAT(0) boundaries of
a given CAT(0) group $G$ are shape equivalent.

\begin{proof}
If $G$ is torsion-free, then a geometric $G$-action on a proper CAT(0) space
$X$ is necessarily free, so $X$ is an $EG$ space. It follows that all CAT(0)
spaces on which $G$ acts geometrically are proper homotopy equivalent. So, by
Corollary \ref{Corollary: Uniquenss for Z-boundaries}, all CAT(0) boundaries
of $G$ have the same shape.

If $G$ has torsion there is more work to be done, but the idea is the same. In
\cite{On}, Ontaneda showed that any two proper CAT(0) spaces on which $G$ acts
geometrically are proper homotopy equivalent, so again their boundaries have
the same shape.
\end{proof}
\end{theorem}

As an application of Theorem
\ref{Theorem: Shape uniqueness of CAT(0) boundaries}, Corollary
\ref{Corollary: Corollary to Bestvina-Mess theorem} can be repeated for
torsion-free CAT(0) groups, with two exceptions: \textbf{(a)} we must omit
reference to the Rips complex since it is not known to be an $EG$ for CAT(0)
groups, and \textbf{(b)} in general, $\operatorname*{pro}$-$\pi_{\ast}\left(
\varepsilon(G)\right)  $ and $\check{\pi}_{\ast}\left(  \varepsilon(G)\right)
$ are not known to be well-defined since the following is open.

\begin{conjecture}
[CAT(0) Semistability Conjecture]Every 1-ended CAT(0) group $G$ is semistable.
\end{conjecture}

It is worth noting that $\partial G,$ itself, provides an approach to this
conjecture. By a result from shape theory \cite[Th.7.2.3]{DySe}, $G$ is
semistable if and only if $\partial G$ has the shape of a locally connected
compactum. (This is true in much greater generality.)

Before moving away from CAT(0) group boundaries, we mention a few more
applications.\smallskip

\begin{itemize}
\item The Bestvina-Mess formulas, mentioned earlier, relating the
cohomological dimension of a torsion-free $G$ to the topological dimension of
$\partial G$ are also valid for CAT(0) $G.$

\item Swenson \cite{Swe} has shown that a CAT(0) group $G$ with a cut point in
$\partial G$ has an infinite torsion subgroup.

\item Ruane \cite{Ru} has shown that for CAT(0) $G$, if $\partial G$ is a
circle, then $G$ is virtually $%
\mathbb{Z}
\times%
\mathbb{Z}
$ or the fundamental group of a closed hyperbolic surface; and if $\partial G$
is a suspended Cantor set, then $G$ is virtually $\mathbb{F}_{2}\times%
\mathbb{Z}
$.

\item Swenson and Papasoglu \cite{SP} have, in a manner similar to Bowditch's
work on $\delta$-hyperbolic groups, used cut pairs in $\partial G$ to obtain a
JSJ-\allowbreak de\-comp\-o\-sit\-ion result for CAT(0) groups.
\end{itemize}

\subsection{A general theory of group boundaries}

Motivated by their usefulness in the study of $\delta$-hyperbolic and CAT(0)
groups, Bestvina \cite{Be} developed an axiomatic approach to group boundaries
which unified the existing theories and provided a framework for defining
group boundaries more generally. We begin with the original definition, then
introduce some variations.\medskip

A $\mathcal{Z}$\emph{-structure on a group }$G$ is a topological pair $\left(
\overline{X},Z\right)  $ satisfying the following four conditions:\smallskip

\begin{enumerate}
\item $\overline{X}$ is a compact ER,

\item $Z$ is a $\mathcal{Z}$-set in $\overline{X}$,

\item $X=\overline{X}-Z$ admits a proper, free, cocompact $G$-action, and

\item the $G$-action on $X$ satisfies the following \emph{nullity condition}:
for every compactum $A\subseteq X$ and every open cover $\mathcal{U}$ of
$\overline{X}$, all but finitely many $G$-translates of $A$ are $\mathcal{U}%
$-small, i.e., are contained in some element of $\mathcal{U}$.\smallskip
\end{enumerate}

\noindent A pair $\left(  \overline{X},Z\right)  $ that satisfies (a)-(c), but
not necessarily (d) is called a \emph{weak }$\mathcal{Z}$\emph{-structure} on
$G$, while a $\mathcal{Z}$-structure on $G$ that satisfies the
additional\emph{ }condition:\smallskip

\begin{enumerate}
\item[(e)] the $G$-action on $X$ extends to a $G$-action on $\overline{X}%
$,\smallskip
\end{enumerate}

\noindent is called an $E\mathcal{Z}$\emph{-structure }(an\emph{ equivariant
}$\mathcal{Z}$\emph{-structure}) on $G$. A \emph{weak }$E\mathcal{Z}%
$\emph{-structure }is a weak $\mathcal{Z}$-structure that satisfies Condition
(e)\footnote{Bestvina informally introduced the definition of weak
$\mathcal{Z}$-structure in \cite{Be}, where he also commented on his decision
to omit Condition 5) from the definition of $\mathcal{Z}$-structure. Farrell
and Lafont introduced the term $E\mathcal{Z}$-structure in \cite{FL}%
.}.\medskip

\noindent Under the above circumstances, $Z$ is called a $\mathcal{Z}%
$\emph{-boundary}, a \emph{weak }$\mathcal{Z}$\emph{-boundary}, an
$E\mathcal{Z}$\emph{-boundary}, or a \emph{weak }$E\mathcal{Z}$%
\emph{-boundary}, as appropriate.

\begin{example}
[A sampling of $\mathcal{Z}$-structures]\label{Example: List of Z-structures}%
\hspace*{0.5in}

\begin{enumerate}
\item \noindent The $\mathcal{Z}$-com\-pact\-ific\-at\-ion $\overline{P_{\rho
}\left(  G\right)  }\allowbreak=\allowbreak P_{\rho}\left(  G\right)
\cup\partial G$ of Theorem \ref{Theorem: Bestvina-Mess} is an $E\mathcal{Z}%
$-structure whenever $G$ is a torsion-free $\delta$-hyperbolic
group.\smallskip

\item \noindent If a torsion-free group $G$ admits a geometric action on a
finite-dimensional proper CAT(0) space $X$, then $\overline{X}=X\cup
\partial_{\infty}X$ is an $E\mathcal{Z}$-structure\emph{ }for $G$.\smallskip

\item \noindent The Baumslag-Solitar group $BS\!(1,2)\allowbreak
=\allowbreak\left\langle a,b\mid bab^{-1}=a^{2}\right\rangle $ is put forth by
Bestvina as an example that is neither $\delta$-hyperbolic nor CAT(0), but
still admits a $\mathcal{Z}$-structure. The $\mathcal{Z}$-structure described
in \cite{Be} is also an $E\mathcal{Z}$-structure. The traditional cocompact
$EG$ 2-complex for $BS(1,2)$ is homeomorphic to $\mathbb{T}_{3}\times%
\mathbb{R}
$, where $\mathbb{T}_{3}$ is the uniformly trivalent tree. Given the Euclidean
product metric, $\mathbb{T}_{3}\times%
\mathbb{R}
$ is CAT(0), so adding the visual boundary gives a weak $\mathcal{Z}%
$-structure, with a suspended Cantor set as boundary. (Since the action of
$BS\left(  1,2\right)  $ on $\mathbb{T}_{3}\times%
\mathbb{R}
$ is not by isometries, one cannot conclude that $BS\left(  1,2\right)  $ is
CAT(0)). This weak $\mathcal{Z}$-structure does not satisfy the nullity
condition---instead it provides a nice illustration of the failure of that
condition. Nevertheless, by using this structure as a starting point, a
genuine $\mathcal{Z}$-structure (in fact more than one) can be
obtained.\smallskip

\item \noindent Januszkiewicz and \'{S}wi\c{a}tkowski \cite{JS} have developed
a theory of \textquotedblleft systolic\textquotedblright\ spaces and groups
that parallels, but is distinct from, CAT(0) spaces and groups. Among systolic
groups are many that are neither $\delta$-hyperbolic nor CAT(0). A delicate
construction by Osajda and Przytycki in \cite{OP} places $E\mathcal{Z}%
$-structures on all torsion-free systolic groups.\smallskip

\item \noindent Dahmani \cite{Dah} showed that, if a group $G$ is hyperbolic
relative to a collection of subgroups, each of which admits a $\mathcal{Z}%
$-structure, then $G$ admits a $\mathcal{Z}$-structure.\smallskip

\item \noindent Tirel \cite{Ti} showed that if $G$ and $H$ each admit
$\mathcal{Z}$-structures (resp., $E\mathcal{Z}$-structures), then so do
$G\times H$ and $G\ast H$.\smallskip

\item \noindent In \cite{Gu3}, this author initiated a study of weak
$\mathcal{Z}$-structures on groups. Examples of groups shown to admit weak
$\mathcal{Z}$-structures include all type F groups that are simply connected
at infinity and all groups that are extensions of a type F group by a type F
group.\medskip
\end{enumerate}
\end{example}

\begin{exercise}
Verify the assertion made in Item (b) of Example
\ref{Example: List of Z-structures}.
\end{exercise}

\begin{exercise}
For $G\times H$ in Item (f), give an easy proof of the existence of weak
$\mathcal{Z}$-structures (resp., weak $E\mathcal{Z}$-structures). As with Item
(c), the difficult part is the nullity condition.\medskip
\end{exercise}

Given the wealth of examples, it becomes natural to ask whether all reasonably
nice groups admits $\mathcal{Z}$-structures. The following helps define
\textquotedblleft reasonably nice\textquotedblright.

\begin{proposition}
A group $G$ that admits a weak $\mathcal{Z}$-structure must have type $F$.

\begin{proof}
If $\left(  \overline{X},Z\right)  $ is a weak $\mathcal{Z}$-structure on $G$,
then $X=\overline{X}-Z$ is an $EG$ space and $X\rightarrow G\backslash X$ is a
covering projection. Since being an ENR is a local property, $G\backslash X$
is an ENR; it is also compact and aspherical. By Proposition
\ref{Proposition: ANR facts}, $G\backslash X$ is homotopy equivalent to a
finite complex $K$, which is a $K\!\left(  G,1\right)  $.\medskip
\end{proof}
\end{proposition}

\begin{Question}
[all are open]Does every group of type F admit a $\mathcal{Z}$-structure? an
$E\mathcal{Z}$-structure? a weak $\mathcal{Z}$-structure? a weak
$E\mathcal{Z}$-structure?\medskip
\end{Question}

The first of the above questions was asked explicitly by Bestvina in
\cite{Be}, where he also mentions the version for weak $\mathcal{Z}%
$-structures. The latter of those two was also mentioned in \cite{BM}, where
the weak $E\mathcal{Z}$-version is explicitly asked. The $E\mathcal{Z}%
$-version, was suggested by Farrell and Lafont in \cite{FL}.

In \cite{Be}, Bestvina prefaced his posing of the $\mathcal{Z}$-structure
Question with the warning: \emph{\textquotedblleft There seems to be no
systematic method of constructing boundaries of groups in general, so the
following is probably hopeless.\textquotedblright}\ In the years since that
question was posed, a general strategy has still not emerged. However, there
have been successes (such as those noted in Example
\ref{Example: List of Z-structures}) when attention is focused on a specific
group or class of groups. In private conversations and in presentations,
Bestvina has suggested some additional groups for consideration; notable among
these are $Out\left(  \mathbb{F}_{n}\right)  $ and the various
Baumslag-Solitar groups $BS\left(  m,n\right)  $. Farrell and Lafont have
specifically asked about $E\mathcal{Z}$-structures for torsion-free finite
index subgroups of $SL_{n}\left(
\mathbb{Z}
\right)  $. A less explicit, but highly important class of groups, are the
fundamental groups of closed aspherical manifolds (or more generally,
Poincar\'{e} duality groups)---the hope being that well-developed tools from
manifold topology might provide an advantage.

Bestvina \cite[Lemma 1.4]{Be} has shown that if $G$ admits a $\mathcal{Z}%
$-structure $\left(  \overline{X},Z\right)  $, then every cocompact $EG$
complex $Y$ can be incorporated into a $\mathcal{Z}$-structure $\left(
\overline{Y},Z\right)  $. In particular, every cocompact $EG$ complex
satisfies the hypotheses of Theorem
\ref{Theorem: C-S Z-compactification Theorem}. So it seems natural to begin
with:\medskip

\begin{Question}
\label{Question: inward tameness for groups}Must the universal cover of a
finite aspherical complex be inward tame? absolutely inward tame?\medskip
\end{Question}

Remarkably, nothing seems to be known here. An early version of the question
goes back to \cite{Ge1}, with more explicit formulations found in \cite{Gu1}
and \cite{Fe1}.

Since, for fixed $G$, all cocompact $EG$ spaces are proper homotopy
equivalent, we can view \emph{inward tameness} as a property possessed by some
(possibly all) type $F$ groups. Moreover, if $G$ is inward tame, we can use
\S \ref{Subsection: Shape of the end of a space} to define the \emph{shape of
the end of }$G$. Specifically, for $X$ a cocompact $EG$, $\mathcal{S}%
\emph{h\!}\left(  \varepsilon\left(  G\right)  \right)  =\mathcal{S}%
\emph{h\!}\left(  \varepsilon\left(  X\right)  \right)  $. If $A\in
\mathcal{S}\emph{h\!}\left(  \varepsilon\left(  G\right)  \right)  $, we might
even view $A$ as a \textquotedblleft pre-$\mathcal{Z}$%
-boundary\textquotedblright\ and $\left(  X,A\right)  $ as a \textquotedblleft
pre-$\mathcal{Z}$-structure\textquotedblright\ for $G$.\medskip

As for applications of the various sorts of $\mathcal{Z}$-boundaries, we list
a few.

\begin{itemize}
\item As noted in the previous paragraph, even pre-$\mathcal{Z}$-boundaries
are well-defined up to shape. So a result like Corollary
\ref{Corollary: Corollary to Bestvina-Mess theorem} can be stated here, with
the same exceptions as noted above for CAT(0) groups.

\item In \cite{Be}, it is shown that the Bestvina-Mess formulas relating the
cohomological dimension of a torsion-free $G$ to the topological dimension of
$\partial G$ are again valid for $\mathcal{Z}$-boundaries. For this, the full
strength of Bestvina's definition of $\mathcal{Z}$-structure is used.

\item Carlsson and Pedersen \cite{CP} and Farrell and Lafont \cite{FL} have
shown that groups admitting an $E\mathcal{Z}$-structure satisfy the Novikov Conjecture.
\end{itemize}

\subsection{Further generalizations}

A pair of generalizations to the various ($E$)$\mathcal{Z}$-struc\-ture and
boundary definitions can be found in the literature. See, for example,
\cite{Dr3}.\medskip

\noindent\textbf{i)} \ Replace the requirement that $\overline{X}$ be an ER
with the weaker requirement that it be an AR.\smallskip

\noindent\textbf{ii)} \ Drop the freeness requirement for the $G$-action on
$X$.\medskip

\noindent Change i) simply allows $\overline{X}$ to be infinite-dimensional;
by itself that may be of little consequence. After all, $X$ is still a
cocompact $EG$, so there exists a finite $K\!\left(  G,1\right)  $ complex
$K$. If $Z$ is finite-dimensional, Bestvina's boundary swapping trick
(\cite[Lemma 1.4]{Be}) produces a new $\mathcal{Z}$-structure $\left(
\overline{Y},Z\right)  $ in which $\overline{Y}$ is an $ER$. This motivates
the question:\medskip

\begin{Question}
If $\left(  \overline{X},Z\right)  $ is a $\mathcal{Z}$-structure on a group
$G$ in the sense of \cite{Be}, except that $\overline{X}$ is only required to
be an ANR, must $Z$ still be finite-dimensional? (Compare to \cite[Th.12]%
{Swe}, which shows that a CAT(0) group boundary is finite-dimensional,
regardless of the CAT(0) space it bounds.)\footnote{\textsc{Added in proof.
}An affirmative answer to this question was recently obtained by Molly
Moran.}\medskip
\end{Question}

Change ii) is more substantial; it allows for groups with torsion.
$\mathcal{Z}$-structures of this sort are plentiful in the categories of
$\delta$-hyperbolic and CAT(0) groups, with Coxeter groups the prototypical
examples; so this generalization is very natural. There are, however,
complications. When $G$ has torsion, the notion of a cocompact $EG$ complex
must be replaced by that of a cocompact (or $G$-finite) $\underline{E}G$
complex, where $G$ may act with fixed points, subject to the requirement that
stabilizers of all finite subgroups are contractible subcomplexes. This notion
is fruitful and cocompact $\underline{E}G$ complexes, when they exist, are
well-defined up to $G$-equivariant homotopy equivalence, and more importantly
(from the point of view of these notes) up to proper homotopy equivalence.

In order to obtain the sorts of conclusions we are concerned with here,
positive answers to the following, questions would be of interest.

\begin{Question}
Suppose $G$ admits a $\mathcal{Z}$-structure $\left(  \overline{X},Z\right)
$, but with the $G$-action on $X$ not required to be free. If $\left(
\overline{X^{\prime}},Z^{\prime}\right)  $ is another such $\mathcal{Z}%
$-structure, is $X\overset{p}{\simeq}X^{\prime}$? More specifically, does
there exist a cocompact $\underline{E}G$ complex and must $X$ be proper
homotopy equivalent to that complex?
\end{Question}

\section{$\mathcal{Z}$-boundaries in manifold
topology\label{Section: Z-boundaries in manifolds topology}}

In this section we look specifically at $\mathcal{Z}$%
-com\-pact\-ific\-at\-ions and $\mathcal{Z}$-boundaries of manifolds, with an
emphasis on open manifolds and manifolds with compact boundary. In
\S \ref{Section: Z-boundaries in geometric group theory} we noted the
occurrence of many exotic group boundaries: Cantor sets, suspended Cantor
sets, Hawaiian earrings, Sierpinski carpets, and Pontryagin surfaces, to name
a few. By contrast, we will see that a $\mathcal{Z}$-boundary of an
$n$-manifold with compact boundary is always a homology $\left(  n-1\right)
$-manifold. That does not mean the $\mathcal{Z}$-boundary is always
nice---recall Example \ref{Example: Z-compactifying the Davis manifold}---but
it does mean that manifold topology forces some significant regularity on
potential $\mathcal{Z}$-boundaries. Here we take a look at that result and
some related applications. First, a quick introduction to homology manifolds.

\subsection{Homology manifolds}

If $N^{n}$ is an $n$-manifold with boundary, then each $x\in
\operatorname*{int}N^{n}$ has \emph{local homology }%
\[
\widetilde{H}_{\ast}\left(  N^{n},N^{n}-x\right)  \cong\widetilde{H}_{\ast
}\left(
\mathbb{R}
^{n},%
\mathbb{R}
^{n}-\mathbf{0}\right)  \cong\left\{
\begin{array}
[c]{cc}%
\mathbb{Z}%
& \text{if }\ast=n\\
0 & \text{otherwise}%
\end{array}
\right.
\]
and each $x\in\partial N^{n}$ has local homology
\[
\widetilde{H}_{\ast}\left(  N^{n},N^{n}-x\right)  \cong\widetilde{H}_{\ast
}\left(
\mathbb{R}
_{+}^{n},%
\mathbb{R}
_{+}^{n}-\mathbf{0}\right)  \equiv0.
\]
This motivates the notion of a \textquotedblleft homology
manifold\textquotedblright.

Roughly speaking, $X$ is a \emph{homology n}-\emph{manifold} if
\[
\widetilde{H}_{\ast}\left(  X,X-x\right)  \allowbreak\cong\allowbreak
\widetilde{H}_{\ast}\left(
\mathbb{R}
^{n},%
\mathbb{R}
^{n}-\mathbf{0}\right)
\]
for all $x\in X$, and a \emph{homology n}-\emph{manifold with boundary} if%
\[
\widetilde{H}_{\ast}\left(  X,X-x\right)  \allowbreak\cong\allowbreak
\widetilde{H}_{\ast}\left(
\mathbb{R}
^{n},%
\mathbb{R}
^{n}-\mathbf{0}\right)  \text{\quad or\quad}\widetilde{H}_{\ast}\left(
X,X-x\right)  \allowbreak\cong\allowbreak\widetilde{H}_{\ast}\left(
\mathbb{R}
_{+}^{n},%
\mathbb{R}
_{+}^{n}-\mathbf{0}\right)  \equiv0
\]
for all $x\in X$. In the latter case we define%
\[
\partial X\equiv\left\{  x\in X\mid\widetilde{H}_{\ast}\left(  X,X-x\right)
=0\right\}
\]
and call this set the \emph{boundary of} $X.$\footnote{All homology here is
with $%
\mathbb{Z}
$-coefficients. With the same strategy and an arbitrary coefficient ring, we
can also define $R$-\emph{homology manifold }and $R$-\emph{homology manifold
with boundary.}}

The reason for the phrase \textquotedblleft roughly speaking\textquotedblright%
\ in the above paragraph is because ordinary singular homology theory does not
always detect the behavior we are looking for. This issue is analogous to what
happened in shape theory; there, when singular theory told us that the
homology of the Warsaw circle $W$ was the same as that of a point, we
developed \v{C}ech homology theory to better capture the circle-like nature of
$W$. In the current setting, we again need to adjust our homology theory to
match our goals. Without going into detail, we simply state that, for current
purposes Borel-Moore homology, or equivalently Steenrod homology (see
\cite{BoMo}, \cite{Fe2}, or \cite{Mil}), should be used. Moreover, since
Borel-Moore homology of a pair requires that $A$ be closed in $X$, we
interpret $\widetilde{H}_{\ast}\left(  X,X-x\right)  $ to mean
$\underrightarrow{\lim}\widetilde{H}_{\ast}\left(  X,X-U\right)  $ where $U$
varies over all open neighborhoods of $x$.

With the above adjustment in place, we are nearly ready to discuss the
essentials of homology manifolds. Before doing so we note that there is an
entirely analogous theory of \emph{cohomology manifolds}, in which
Alexander-\v{C}ech theory is the preferred cohomology theory. We also note
that both Borel-Moore homology and Alexander-\v{C}ech cohomology theories
agree with the singular theories when $X$ is an ANR. An ANR homology manifold
is often called a \emph{generalized manifold}---a class of objects that plays
an essential role in geometric topology.

\begin{example}
Let $\Sigma^{n}$ be a non-simply connected $n$-manifold with the same $%
\mathbb{Z}
$-homology as $\mathbb{S}^{n}$, e.g., the boundary of a Newman contractible
$\left(  n+1\right)  $-manifold. Then $X=\operatorname*{cone}\left(
\Sigma^{n}\right)  =\Sigma^{n}\times\left[  0,1\right]  /\left\{  \Sigma
^{n}\times1\right\}  $ is a homology $\left(  n+1\right)  $-manifold with
boundary, where $\partial X=\Sigma^{n}\times0$. The double of $X^{n+1}$, the
suspension of $\Sigma^{n+1}$, is a homology manifold that is homotopy
equivalent to $\mathbb{S}^{n+1}$. Both of these are ANRs, hence generalized
manifolds, but neither is an actual manifold.
\end{example}

\begin{example}
\label{Example: crushing out a cell-like set}Let $A$ be a compact proper
subset of the interior of an $n$-manifold $M^{n}$and let $Y=M^{n}/A$ be the
quotient space obtained by identifying $A$ to a point. If $A$ is cell-like
(i.e., has trivial shape), then $X$ is a generalized $n$-manifold. In many
cases $Y\approx M^{n}$, for example, when $A$ is a tame arc or ball. In other
cases---for example, $A$ a wild arc with non-simply connected complement or
$A$ a Newman contractible $n$-manifold embedded in $M^{n}$, $Y$ is not a manifold.
\end{example}

\begin{exercise}
Verify the unproven assertions in the above two exercises.
\end{exercise}

\begin{remark}
\emph{The subject of }Decomposition Theory\emph{ is motivated by Example
\ref{Example: crushing out a cell-like set}. There, the following question is
paramount: Given a pairwise disjoint collection }$\mathcal{G}$\emph{ of
cell-like compacta in a manifold }$M^{n\text{ }}$\emph{satisfying a certain
niceness condition (an }upper semicontinuous decomposition\emph{), when is the
quotient space} $M/\mathcal{G}$ \emph{a manifold? Although the premise sounds
simple and very specific, results from this area have had broad-ranging
impacts on geometric topology, including: existence of exotic involutions on
spheres, existence of exotic manifold factors (non-manifolds }$X$\emph{ for
which }$X\times%
\mathbb{R}
$\emph{ is a manifold), existence of non-PL triangulations of manifolds, and a
solution to the 4-dimensional Poincar\'{e} Conjecture. The Edwards-Quinn
Manifold Recognition Theorem, which will be used shortly, belongs to this
subject. References to the \textquotedblleft Moore-Bing school of
topology\textquotedblright\ usually indicate work by R.L. Moore, R H Bing, and
their mathematical descendents in this area. See }\cite{Daver}\emph{ for a
comprehensive discussion of this topic.}
\end{remark}

\begin{exercise}
\label{Exercise: minimal non-ANR homology manifold}Here we describe a simple
non-ANR homology manifold. Let $H^{n}$ be non-simply connected $n$-manifold
with the homology of a point and a boundary homeomorphic to $\mathbb{S}^{n-1}%
$, and let $\left\{  B_{i}^{n}\right\}  _{i=1}^{\infty}$ be a sequence of
pairwise disjoint round $n$-balls in $\mathbb{S}^{n}$ converging to a point
$p$. Create $X$ by removing the interiors of the $B_{i}^{n}$ and replacing
each with $H_{i}^{n}\approx H^{n}$. Topologize $X$ so that each neighborhood
of $p$ in $X$ contains all but finitely many of the $H_{i}^{n}$. The result is
a homology manifold (some knowledge of Borel-Moore homology is needed to
verify this fact).

Explain why $X$ is not an ANR. Then show that $X$ does not satisfy the
definition of homology manifold if singular homology is used.
\end{exercise}

\begin{exercise}
Show that the $\mathcal{Z}$-boundary attached to the Davis manifold described
in Example \ref{Example: Z-compactifying the Davis manifold} is homeomorphic
to the non-ANR homology manifold described in Exercise
\ref{Exercise: minimal non-ANR homology manifold} (some attention must be paid
to orientations).
\end{exercise}

For now, the reader may wish to treat the following theorem as a set of
axioms; \cite{AG} shows how the classical literature can be woven together to
obtain proofs.

\begin{theorem}
[Fundamental facts about (co)homology manifolds]%
\label{Theorem: Fundamental properties of homology manifolds}\hspace*{0.5in}

\begin{enumerate}
\item A space $X$ is a homology $n$-manifold if and only if it is a cohomology
$n$-manifold.

\item The boundary of a (co)homology $n$-manifold is a (co)homology $\left(
n-1\right)  $-\allowbreak man\-i\-fold without boundary.

\item The union of two (co)homology $n$-manifolds with boundary along a common
boundary is a (co)homology $n$-manifold.

\item (Co)Homology manifolds are locally path connected.
\end{enumerate}
\end{theorem}

\begin{corollary}
\label{Corollary: homology properties Z-compactified manifolds}Let $M^{n}$ be
an open $n$-manifold (or even just an open generalized manifold) and
$\overline{M^{n}}=M^{n}\cup Z$ be a $\mathcal{Z}$-com\-pact\-ific\-at\-ion. Then

\begin{enumerate}
\item $\overline{M^{n}}$ is a homology $n$-manifold with boundary,

\item $\partial\overline{M^{n}}=Z$, and

\item $Z$ is a homology $\left(  n-1\right)  $-manifold.
\end{enumerate}
\end{corollary}

\begin{proof}
[Proof of Corollary]For (a) and (b) we need only check that $H_{\ast}\left(
\overline{M^{n}},\overline{M^{n}}-z\right)  \equiv0$ at each $z\in Z$. Since
$\overline{M^{n}}$ is an ANR, we are free to use singular homology in place of
Borel-Moore theory. A closed subset of a $\mathcal{Z}$-set is a $\mathcal{Z}%
$-set, so $\left\{  z\right\}  $ is a $\mathcal{Z}$-set in $\overline{M^{n}}$,
and hence, $\overline{M^{n}}-z\hookrightarrow\overline{M^{n}}$ is a homotopy
equivalence. The desired result now follows from the long exact sequence for pairs.

Item (c) is now immediate from Theorem
\ref{Theorem: Fundamental properties of homology manifolds}.
\end{proof}

\begin{exercise}
Show that if $M^{n}$ is a CAT(0) $n$-manifold, then every metric sphere in
$M^{n}$ is a homology $\left(  n-1\right)  $-manifold.
\end{exercise}

\begin{remark}
\label{Remark: Poincare duality for homology manifolds}\emph{In addition to
Theorem \ref{Theorem: Fundamental properties of homology manifolds}, it is
possible to define orientation for homology manifolds and prove a version of
Poincar\'{e} duality for the orientable ones. With those tools, one can also
prove, for example, that any }$\mathcal{Z}$\emph{-boundary of a contractible
open n-manifold has the (Borel-Moore) homology of an }$\left(  n-1\right)
$\emph{-sphere.}
\end{remark}

Before moving to applications, we state without proof one of the most
significant results in this area. A nice exposition can be found in
\cite{Daver}.

\begin{theorem}
[Edwards-Quinn Manifold Recognition Theorem]Let $X^{n}$ ($n\geq5$) be a
generalized homology $n$-manifold without boundary and suppose $X^{n}$
contains a nonempty open set $U\approx%
\mathbb{R}
^{n}$. Then $X^{n}$ is an $n$-manifold if and only if it satisfies the
disjoint disks property (DDP).
\end{theorem}

A space $X$ satisfies the DDP if, for any pair of maps $f,g:D^{2}\rightarrow
X$ and any $\varepsilon>0$, there exist $\varepsilon$-approximations
$f^{\prime}$ and $g^{\prime}$ of $f$ and $g$, so that $f^{\prime}\left(
D^{2}\right)  \cap g^{\prime}\left(  D^{2}\right)  =\varnothing$.

\subsection{Some applications of $\mathcal{Z}$-boundaries to manifold
topology}

Most results in this section come from \cite{AG}. Here we provide only the
main ideas; for details, the reader should consult the original paper. For the
sake of brevity, we focus on high-dimensional results. In many cases,
low-dimensional analogs are true for different reasons.

Let $\overline{M^{n}}=M^{n}\cup Z$ be a $\mathcal{Z}$-com\-pact\-ific\-at\-ion
of an open $n$-manifold. Since $\overline{M^{n}}$ need not be a manifold with
boundary, the following is a pleasant surprise.

\begin{theorem}
\label{Theorem: Gluing along a Z-boundary}Suppose $\overline{M^{n}}=M^{n}\cup
Z$ and $\overline{N^{n}}=N^{n}\cup Z^{\prime}$ are $\mathcal{Z}$%
-com\-pact\-ific\-at\-ion of open $n$-manifolds ($n>4$) and $h:Z\rightarrow
Z^{\prime}$ is a homeomorphism. Then $P^{n}=\overline{M^{n}}\cup_{h}%
\overline{N^{n}}$ is a closed $n$-manifold.

\begin{proof}
[Sketch of proof]Theorem
\ref{Theorem: Fundamental properties of homology manifolds} asserts that
$P^{n}$ is a homology $n$-manifold. From there one uses delicate properties of
homology manifolds to prove that $P^{n}$ is locally contractible at each point
on the \textquotedblleft seam\textquotedblright, $Z=Z^{\prime}$; hence,
$P^{n}$ is an ANR. Another delicate, but more straightforward, argument
\ (this part using the fact that $Z$ and $Z^{\prime}$ are $\mathcal{Z}$-sets)
verifies the DDP for $P^{n}$. Open subsets of $P^{n}$ homeomorphic to $%
\mathbb{R}
^{n}$ are plentiful in the manifolds $M^{n}$ and $N^{n}$, so Edwards-Quinn can
be applied to complete the proof.
\end{proof}
\end{theorem}

\begin{corollary}
\label{Corollary: Doubling along a Z-boundary}The double of $\overline{M^{n}}$
along $Z$ is an $n$-manifold. If $M^{n}$ is contractible, that double is
homeomorphic to $\mathbb{S}^{n}$, and there is an involution of $\mathbb{S}%
^{n}$ with $Z$ as its fixed set.

\begin{proof}
[Sketch of proof]$\operatorname{Double}(\overline{M^{n}})\approx\mathbb{S}%
^{n}$ will follow from the Generalized Poincar\'{e} conjecture if we can show
that it is a simply connected manifold with the homology of an $n$-sphere. The
involution interchanges the two copies of $\overline{M^{n}}$.

That $\operatorname{Double}(\overline{M^{n}})$ has the homology of
$\mathbb{S}^{n}$ is a consequence of Mayer-Vietoris and Remark
\ref{Remark: Poincare duality for homology manifolds}. Since $\overline{M^{n}%
}$ is simply connected, simple connectivity of $\operatorname{Double}%
(\overline{M^{n}})$ would follow directly from van Kampen's Theorem if the
intersection between the two copies was nice. Instead a controlled variation
on the traditional proof of van Kampen's Theorem is employed. Use the fact
that homology manifolds are locally path connected to divide an arbitrary loop
into loops lying in one or the other copy of $\overline{M^{n}}$, where they
can be contracted. Careful control is needed, and the fact that $\overline
{M^{n}}$ is locally contractible is important.
\end{proof}
\end{corollary}

\begin{theorem}
\label{Theorem: Z-boundaries determine interiors}If contractible open
manifolds $M^{n}$ and $N^{n}$ ($n>4$) admit $\mathcal{Z}$%
-com\-pact\-ific\-at\-ions with homeomorphic $\mathcal{Z}$-boundaries, then
$M^{n}\approx N^{n}$.
\end{theorem}

\begin{proof}
[Sketch of proof]Let $Z$ denote the common $\mathcal{Z}$-boundary. The
argument used in Corollary \ref{Corollary: Doubling along a Z-boundary} shows
that the union of these com\-pact\-ific\-at\-ions along $Z$ is $\mathbb{S}%
^{n}$. Let $W^{n+1}=\mathbb{B}^{n+1}-Z$ and note that $\partial W^{n+1}%
=M^{n}\sqcup N^{n}$, providing a noncompact cobordism $\left(  W^{n+1}%
,M^{n},N^{n}\right)  $. The proof is completed by applying the Proper
s-cobordism Theorem \cite{Si2} to conclude that $W^{n+1}\approx M^{n}%
\times\left[  0,1\right]  $. That requires some work. First show that
$M^{n}\hookrightarrow W^{n+1}$ is a proper homotopy equivalence. (The fact
that $Z$ is a $\mathcal{Z}$-set in $\mathbb{B}^{n+1}$ is key.) Then, to
establish that $M^{n}\hookrightarrow W^{n+1}$ is an infinite simple homotopy
equivalence, some algebraic obstructions must be checked. Fortunately, there
are \textquotedblleft naturality results\textquotedblright\ from \cite{CS}
that relate those obstructions to the $\mathcal{Z}$-compactifiability
obstructions for $M^{n}$ and $W^{n+1}$ (as found in Theorem
\ref{Theorem: C-S Z-compactification Theorem}). In particular, since the
latter vanish, so do the former.
\end{proof}

The following can be obtained in a variety of more elementary ways;
nevertheless, it provides a nice illustration of Theorem
\ref{Theorem: Z-boundaries determine interiors}.

\begin{corollary}
If a contractible open $n$-manifold $M^{n}$ can be $\mathcal{Z}$%
-com\-pact\-ified by the addition of an $\left(  n-1\right)  $-sphere, then
$M^{n}\approx%
\mathbb{R}
^{n}$.
\end{corollary}

The Borel Conjecture posits that closed aspherical manifolds with isomorphic
fundamental groups are necessarily homeomorphic. Our interest in contractible
open manifolds led to the following.

\begin{conjecture}
[Weak Borel Conjecture]Closed aspherical manifolds with isomorphic fundamental
groups have homeomorphic universal covers.
\end{conjecture}

Theorem \ref{Theorem: Z-boundaries determine interiors} provides the means for
a partial solution.

\begin{theorem}
\label{Theorem: Weak Borel for mflds with z-boundaries}The Weak Borel
Conjecture is true for those $n$-manifolds ($n>4$) whose fundamental groups
admits $\mathcal{Z}$-structures.
\end{theorem}

\begin{proof}
Let $P^{n}$ and $Q^{n}$ be aspherical manifolds, and $\left(  \overline
{X},Z\right)  $ a $\mathcal{Z}$-structure on $\pi_{1}\left(  P^{n}\right)
\cong\pi_{1}\left(  Q^{n}\right)  $. By Bestvina's boundary swapping trick
\cite[Lemma 1.4]{Be}, both $\widetilde{P}^{n}$ and $\widetilde{Q}^{n}$ can be
$\mathcal{Z}$-compactified by the addition of a copy of $Z$. Now apply Theorem
\ref{Theorem: Z-boundaries determine interiors}.
\end{proof}

\begin{remark}
\emph{Aspherical manifolds to which Theorem
\ref{Theorem: Weak Borel for mflds with z-boundaries} applies include those
with hyperbolic and CAT(0) fundamental groups. We are not aware of
applications outside of those categories.}

\emph{Recently, Bartels and L\"{u}ck \cite{BL} proved the full-blown Borel
Conjecture for }$\delta$\emph{-hyperbolic groups and CAT(0) groups that act
geometrically on finite-dimensional CAT(0) spaces. Not surprisingly, their
proof is more complicated than that of Theorem
\ref{Theorem: Weak Borel for mflds with z-boundaries}.}
\end{remark}

\subsection{$E\mathcal{Z}$-structures in manifold topology}

As discussed in \S \ref{Section: Z-boundaries in geometric group theory}, the
notion of an $E\mathcal{Z}$-structure was formalized by Farrell and Lafont in
\cite{FL}. Among their applications was a new proof of the Novikov Conjecture
for $\delta$-hyperbolic and CAT(0) groups. That result had been obtained
earlier by Carlsson and Pedersen \cite{CP} using similar ideas. We will not
attempt to discuss the Novikov Conjecture here, except to say that it is
related to, but much broader (and more difficult to explain) than the Borel Conjecture.

For a person with interests in manifold topology, one of the more intriguing
aspects of Farrell and Lafont's work is a technique they develop which takes
an arbitrary $\mathcal{Z}$-structure $\left(  \overline{X},Z\right)  $ on a
group $G$ and replaces it with one of the form $\left(  \mathbb{B}%
^{n},Z\right)  $, where $n$ is necessarily large, $Z$ is a topological copy of
the original $\mathcal{Z}$-boundary lying in $\mathbb{S}^{n-1}$, and the new
$EG$ is the $n$-manifold with boundary $\mathbb{B}^{n}-Z$. The beauty here is
that, once the structure is established, all of the tools of high-dimensional
manifold topology are available. In their introduction, they challenge the
reader to find other applications of these manifold $\mathcal{Z}$-structures,
likening them to the action of a Kleinian group on a compactified hyperbolic
$n$-space.

\section{Further reading}

Clearly, we have just scratched the surface on a number of topics addressed in
these notes. For a broad study of geometric group theory with a point of view
similar to that found in these notes, Geoghegan's book, \emph{Topological
methods in group theory \cite{Ge2}}, is the obvious next step.

For those interested in the topology of noncompact manifolds, Siebenmann's
thesis \cite{Si} is still a fascinating read. The main result from that
manuscript can also be obtained from the series of papers \cite{Gu1},
\cite{GT1}, \cite{GT2}, which have the advantage of more modern terminology
and greater generality. Steve Ferry's \emph{Notes on geometric topology
}(available on his website) contain a remarkable collection of fundamental
results in manifold topology. Most significantly, from our perspective, those
notes do not shy away from topics involving noncompact manifolds. There one
can find clear and concise discussions of the Whitehead manifold, the Wall
finiteness obstruction, Stallings' characterization of euclidean space,
Siebenmann's thesis, and much more.

The complementary articles \cite{Si2} and \cite{CS} fit neither into the
category of manifold topology nor that of geometric group theory; but they
contain fundamental results and ideas of use in each area. Researchers whose
work involves noncompact spaces of almost any variety are certain to benefit
from a familiarity with those papers. Another substantial work on the topology
of noncompact spaces, with implications for both manifold topology and
geometric group theory,\ is the book by Hughes and Ranicki, \emph{Ends of
complexes \cite{HR}}.

For the geometric group theorist specifically interested in the interplay
between shapes, group boundaries, $\mathcal{Z}$-sets, and $\mathcal{Z}%
$-com\-pact\-ific\-at\-ions, the papers by Bestvina-Mess \cite{BM}, Bestvina
\cite{Be}, and the follow-up by Dranishnikov give a quick entry into that
subject; while Geoghegan's earlier article, \emph{The shape of a group
\cite{Ge1}}, provides a first-hand account of the origins of many of those
ideas. For general applications of $\mathcal{Z}$-com\-pact\-ific\-at\-ions to
manifold topology, the reader may be interested in \cite{AG}; and for more
specific applications to the Novikov Conjecture, \cite{FL} is a good starting point.

\section{Appendix A: Basics of ANR theory}

Before beginning this appendix, we remind the reader that all spaces discussed
in these notes are assumed to be separable metric spaces.

A locally compact space $X$ is an ANR (\emph{absolute neighborhood retract})
if it can be embedded into $%
\mathbb{R}
^{n}$ or, if necessary, $%
\mathbb{R}
^{\infty}$ (a countable product of real lines) as a closed set in such a way
that there exists a retraction $r:U\rightarrow X$, where $U$ is a neighborhood
of $X$. If the entire space $%
\mathbb{R}
^{n}$ or $%
\mathbb{R}
^{\infty}$ retracts onto $X$, we call $X$ an AR (absolute retract). If $X$ is
finite-dimensional, all mention of $%
\mathbb{R}
^{\infty}$ can be omitted. A finite-dimensional ANR is often called an ENR
(\emph{Euclidean neighborhood retract}) and a finite-dimensional AR an ER.

Use of the word \textquotedblleft absolute\textquotedblright\ in ANR (or AR)
stems from the following standard fact: If one embedding of $X$ as a closed
subset of $%
\mathbb{R}
^{n}$ or $%
\mathbb{R}
^{\infty}$ satisfies the defining condition, then so do all such embeddings.
An alternative definition for ANR (and AR) is commonly found in the
literature. To help avoid confusion, we offer that approach as Exercise
\ref{Exercise: ANE approach to ANRs}. Texts \cite{Bor1} and \cite{Hu} are
devoted entirely to the theory of ANRs; readers can go to either for details.

With a little effort (Exercise \ref{Exercise: AR = contractible ANR}) it can
be shown that an AR is just a contractible ANR, so there is no loss of
generality if focusing on ANRs.

A space $Y$ is \emph{locally contractible }if every neighborhood $U$ of a
point $y\in Y$ contains a neighborhood $V$ of $y$ that contracts within $U$.
It is easy to show that every ANR is locally contractible. A partial converse
gives a powerful characterization of finite-dimensional ANRs.

\begin{theorem}
\label{Theorem: Characterization of f.d. ANRs}A locally compact
finite-dimensional space $X$ is an ANR if and only if it is locally contractible.
\end{theorem}

\begin{example}
By Theorem \ref{Theorem: Characterization of f.d. ANRs}, manifolds,
finite-dimensional locally finite polyhedra and CW complexes, and
finite-dimensional proper CAT(0) spaces are all ANRs.
\end{example}

\begin{example}
It is also true that Hilbert cube manifolds, infinite-dimensional locally
finite polyhedra and CW complexes, and infinite-dimensional proper CAT(0)
spaces are all ANRs. Proofs would require some additional effort, but we will
not hesitate to make use of these facts.
\end{example}

Rather than listing key results individually, we provide a mix of facts about
ANRs in a single Proposition. The first several are elementary, and the final
item is a deep result. Each is an established part of ANR theory.

\begin{proposition}
[Standard facts about ANRs]\label{Proposition: ANR facts}\hspace*{0.5in}

\begin{enumerate}
\item Being an ANR is a local property: every open subset of an ANR is an ANR,
and if every element of $X$ has an ANR neighborhood, then $X$ is an ANR.

\item If $X=A\cup B$, where $A,B,$ and $A\cap B$ are compact ANRs, then $X$ is
a compact ANR.

\item Every retract of an ANR is an ANR; every retract of an AR is an AR.

\item (Borsuk's Homotopy Extension Property) Every $h:\left(  Y\times\left\{
0\right\}  \right)  \cup\left(  A\times\left[  0,1\right]  \right)
\allowbreak\rightarrow\allowbreak X$, where $A$ is a closed subset of a space
$Y$ and $X$ is ANR, admits an extension $H:Y\times\left[  0,1\right]
\rightarrow X$.

\item (West, \cite{We}) Every ANR is proper homotopy equivalent to a locally
finite CW complex; every compact ANR is homotopy equivalent to a finite complex.
\end{enumerate}
\end{proposition}

\begin{remark}
\emph{Items 4) and 5) allow us to extend the tools of algebraic topology and
homotopy theory normally reserved for CW complexes to ANRs. For example,
Whitehead's Theorem, that a map between CW complexes which induces
isomorphisms on all homotopy groups is a homotopy equivalence, is also true
for ANRs. In a very real sense, this sort of result is the motivation behind
ANR theory.}
\end{remark}

\begin{exercise}
\label{Exercise: ANE approach to ANRs}A locally compact space $X$ is an ANE
(\emph{absolute neighborhood extensor}) if, for any space $Y$ and any map
$f:A\rightarrow X$, where $A$ is a closed subset of $Y$, there is an extension
$F:U\rightarrow X$ where $U$ is a neighborhood of $A$. If an extension to all
of $Y$ is always possible, then $X$ is an AE (\emph{absolute extensor}). Show
that being an ANE (or AE) is equivalent to being an ANR (or AR).
\emph{Hint:}\textbf{ }The Tietze Extension Theorem will be helpful.
\end{exercise}

\begin{exercise}
\label{Exercise: AR = contractible ANR}With the help of Exercise
\ref{Exercise: ANE approach to ANRs} and the Homotopy Extension Property,
prove that an ANR is an AR if and only if it is contractible.
\end{exercise}

\begin{exercise}
\label{Exercise: union of balls in a CAT(0) space}A useful property of
Euclidean space is that every compactum $A\subseteq%
\mathbb{R}
^{n}$ has arbitrarily small compact polyhedral neighborhoods. Using the tools
of Proposition \ref{Proposition: ANR facts}, prove the following CAT(0)
analog: every compactum $A$ in a proper CAT(0) space $X$ has arbitrarily small
compact ANR neighborhoods. \emph{Hint:}\textbf{ }Cover $A$ with compact metric
balls. (For examples of ANRs that do not have this property, see \cite{Bo0}
and \cite{Mol}.)
\end{exercise}

\section{Appendix B: Hilbert cube manifolds}

This appendix is a very brief introduction to Hilbert cube manifolds. A
primary goal is to persuade the uninitiated reader that there is nothing to
fear. Although the main results from this area are remarkably strong (we
sometimes refer to them as \textquotedblleft Hilbert cube
magic\textquotedblright), they are understandable and intuitive. Applying them
is often quite easy.

The \emph{Hilbert cube} is the infinite product $\mathcal{Q=}\prod
_{i=1}^{\infty}\left[  -1,1\right]  $ with metric $d\left(  \left(
x_{i}\right)  ,\left(  y_{i}\right)  \right)  =\sum\frac{\left\vert
x_{i}-y_{i}\right\vert }{2^{i}}$. A \emph{Hilbert cube manifold} is a
separable metric space $X$ with the property that each $x\in X$ has a
neighborhood homeomorphic to $\mathcal{Q}$. Hilbert cube manifolds are
interesting in their own right, but our primary interest stems from their
usefulness in working with spaces that are not necessarily
infinite-dimensional---often locally finite CW complexes or more general ANRs.
Two classic examples where that approach proved useful are:\medskip

\begin{itemize}
\item Chapman \cite{Ch1} used Hilbert cube manifolds to prove the topological
invariance of Whitehead torsion for finite CW complexes, i.e., homeomorphic
finite complexes are simple homotopy equivalent.

\item West \cite{We} used Hilbert cube manifolds to solve a problem of Borsuk,
showing that every compact ANR is homotopy equivalent to a finite CW complex.
(See Proposition \ref{Proposition: ANR facts}.)\medskip
\end{itemize}

The ability to attack a problem about ANRs using Hilbert cube manifolds can be
largely explained using the following pair of results.

\begin{theorem}
[Edwards, \cite{Ed}]\label{Theorem: Edwards HCM Theorem}If $A$ is an ANR, then
$A\times\mathcal{Q}$ is a Hilbert cube manifold.
\end{theorem}

\begin{theorem}
[Triangulability of Hilbert Cube Manifolds, Chapman, \cite{Ch}]%
\label{Theorem: Triangulability of HCMs}If $X$ is a Hilbert cube manifold,
then there is a locally finite polyhedron $K$ such that $X\approx
K\times\mathcal{Q}$.
\end{theorem}

A typical (albeit, simplified) strategy for solving a problem involving an ANR
$A$ might look like this:\medskip

\noindent\textbf{A)} Take the product of $A$ with $\mathcal{Q}$ to get a
Hilbert cube manifold $X=A\times\mathcal{Q}$.\smallskip

\noindent\textbf{B)} Triangulate $X$, obtaining a polyhedron $K$ with
$X\approx K\times\mathcal{Q}$.\smallskip

\noindent\textbf{C)} The polyhedral structure of $K$ together with a variety
of tools available in a Hilbert cube manifolds (see below) make solving the
problem easier.\smallskip

\noindent\textbf{D)} Return to $A$ by collapsing out the $\mathcal{Q}$-factor
in $X=A\times\mathcal{Q}$.\medskip

\noindent In these notes, most of our appeals to Hilbert cube manifold
topology are of this general sort. That is not to say the strategy always
works---the main result of \cite{Gu} (see Remark
\ref{Remark: Remarks about the C-S Theorem}(a)) is one relevant example.

Tools available in a Hilbert cube manifold are not unlike those used in
fin\-ite-dim\-ens\-ion\-al manifold topology. We list a few such properties,
without striving for best-possible results.

\begin{proposition}
[Basic properties of Hilbert cube manifolds]%
\label{Proposition: Basic properties of HCMs}Let $X$ be a connected Hilbert
cube manifold.

\begin{enumerate}
\item (Homogeneity) For any pair $x_{1},x_{2}\in X$, there exists a
homeomorphism $h:X\rightarrow X$ with $h\left(  x_{1}\right)  =x_{2}$.

\item (General Position) Every map $f:P\rightarrow X$, where $P$ is a finite
polyhedron can be approximated arbitrarily closely by an embedding.

\item (Regular Neighborhoods) Each compactum $C\subseteq X$ has arbitrarily
small compact Hilbert cube manifold neighborhoods $N\subseteq X$. If $C$ is a
nicely embedded polyhedron, $N$ can be chosen to strong deformation retract
onto $P$.
\end{enumerate}
\end{proposition}

\begin{exercise}
\label{Exercise: homogeneity of Q}As a special case, assertion (a) of
Proposition \ref{Proposition: Basic properties of HCMs} implies that
$\mathcal{Q}$ itself is homogeneous. This remarkable fact is not hard to
prove. A good start is to construct a homeomorphism $h:$
$\mathcal{Q\rightarrow Q}$ with $h\left(  1,1,1,\cdots\right)  =\left(
0,0,0,\cdots\right)  $. To begin, think of a homeomorphism $k:\left[
-1,1\right]  \times\left[  -1,1\right]  $ taking $\left(  1,1\right)  $ to
$\left(  0,1\right)  $, and use it to obtain $h_{1}:$ $\mathcal{Q\rightarrow
Q}$ with $h_{1}\left(  1,1,1,\cdots\right)  =\left(  0,1,1,\cdots\right)  $.
Complete this argument by constructing a sequence of similarly chosen homeomorphisms.
\end{exercise}

\begin{example}
Here is another special case worth noting. Let $K$ be an arbitrary locally
finite polyhedron---for example, a graph. Then $K\times\mathcal{Q}$ is homogeneous.
\end{example}

The material presented here is just a quick snapshot of the elegant and
surprising world of Hilbert cube manifolds. A brief and readable introduction
can be found in \cite{Ch}. Just for fun, we close by stating two more
remarkable theorems that are emblematic of the subject.

\begin{theorem}
[Toru\'{n}czyk, \cite{To}]An ANR $X$ is a Hilbert cube manifold if and only if
it satisfies the General Position property (Assertion (b)) of Proposition
\ref{Proposition: Basic properties of HCMs}.
\end{theorem}

\begin{theorem}
[Chapman, \cite{Ch}]\label{Theorem: Chapman's homeomorphism theorem}A map
$f:K\rightarrow L$ between locally finite polyhedra is an (infinite) simple
homotopy equivalence if and only if $f\times\operatorname*{id}_{\mathcal{Q}%
}:K\times\mathcal{Q}\rightarrow L\times\mathcal{Q}$ is (proper) homotopic to a homeomorphism.
\end{theorem}

\end{document}